\documentclass[a4paper,twoside,10pt]{article}

\usepackage[english]{babel}
\usepackage{lmodern,color}
\usepackage[T1]{fontenc}
\usepackage{indentfirst}
\usepackage{amsmath,amssymb}

\usepackage[a4paper,top=3cm,bottom=3cm,left=3cm,right=3cm,%
bindingoffset=5mm]{geometry}
\usepackage{lmodern}
\usepackage[T1]{fontenc}
\usepackage{lettrine}
\usepackage{booktabs}

\usepackage{authblk}
\usepackage{emp}
\usepackage{amsthm}
\usepackage{amsmath,amssymb,amsthm}
\usepackage{braket}
\usepackage{chngpage}
\usepackage{cases}
\usepackage{graphicx}
\usepackage{epstopdf}
\usepackage{tabularx}
\usepackage{multirow}
\usepackage{amsmath}       
\usepackage{amssymb}       
\usepackage{chapterbib}    
\usepackage{color}
\usepackage{comment}
\usepackage{bm,bbm}
\usepackage{overpic}
\usepackage{times}
\usepackage{epsfig}
\usepackage{graphicx,graphics,rotating}
\usepackage{subfigure}
\usepackage{multicol}
\usepackage{multirow}

\theoremstyle{plain}
\newtheorem{theorem}{Theorem}[section]

\theoremstyle{definition}

\newtheorem{remark}[theorem]{Remark}

\theoremstyle{plain}

\definecolor{MyDarkGreen}{rgb}{0,0.45,0}

\newcommand{\K}{\mathbb{K}}

\newcommand{\INTP}{\footnotesize{\texttt{I}}}
\newcommand{\SPAN}[1]{\mbox{\textrm{span}}\big\{#1\big\}}
\newcommand{\REAL}{\mathbbm{R}}

\newcommand{\half}{\frac{1}{2}}

\newcommand{\qv}{\mathbf{q}}

\newcommand{\vv}{\mathbf{v}}
\newcommand{\wv}{\mathbf{w}}
\newcommand{\xv}{\mathbf{x}}
\newcommand{\yv}{\mathbf{y}}

\newcommand{\as}{a}
\newcommand{\bs}{b}

\newcommand{\fs}{f}

\newcommand{\ms}{m}

\newcommand{\qs}{q}

\newcommand{\us}{u}
\newcommand{\vs}{v}
\newcommand{\ws}{w}
\newcommand{\xs}{x}
\newcommand{\ys}{y}

\newcommand{\Vs}{V}

\newcommand{\PS}[1]{\mathbbm{P}_{#1}}

\newcommand{\HONE}  {H^1}
\newcommand{\HONEzr}{H^1_0}

\newcommand{\LTWO}  {L^2}

\newcommand{\HS}[1] {H^{#1}}

\newcommand{\CS}[1] {C^{#1}}

\newcommand{\VS}[1] {V^{#1}}

\renewcommand{\P} {\textsf{P}}            
\newcommand  {\F} {f}
\newcommand  {\E} {e}
\renewcommand{\S} {\sigma} 

\newcommand{\hh}{h}
\newcommand{\Th}{\Omega_{\hh}}

\newcommand{\xvP}{\xv_{\P}}        
\newcommand{\xvF}{\xv_{\F}}        
\newcommand{\xvE}{\xv_{\E}}        
\newcommand{\xvS}{\xv_{\S}}        

\newcommand{\DIM} {d}              

\newcommand{\hP}{\hh_{\P}}
\newcommand{\hF}{\hh_{\F}}
\newcommand{\hE}{\hh_{\E}}

\newcommand{\hS}{\hh_{\S}}

\newcommand{\mP}{\ABS{\P}}
\newcommand{\mF}{\ABS{\F}}
\newcommand{\mE}{\ABS{\E}}

\newcommand{\mS}{\ABS{\S}}

\newcommand{\dV}{\,dV}
\newcommand{\dS}{\,d\S}
\newcommand{\dx}{\,d\xv}

\newcommand{\norPF}{\mathbf{n}_{\P,\F}}

\newcommand{\fsh}{\fs_{\hh}}

\newcommand{\uss} {\us^{s}}
\newcommand{\ussh}{\ush^{s}}
\newcommand{\ussht}{\widetilde{\us}_{\hh}^{s}}
\newcommand{\ush}{\us_{\hh}}
\newcommand{\usht}{\widetilde{\us}_{\hh}}

\newcommand{\vsh}{\vs_{\hh}}

\newcommand{\vsI}{\vs^{\INTP}}

\newcommand{\wsh}{\ws_{\hh}}
\newcommand{\wsht}{\widetilde{\ws}_{\hh}}

\newcommand{\vvh}{\vv_{\hh}}

\newcommand{\wvh}{\wv_{\hh}}

\newcommand{\asP}{\as^{\P}}
\newcommand{\bsP}{\bs^{\P}}

\newcommand{\ash}{\as_{\hh}}
\newcommand{\bsh}{\bs_{\hh}}

\newcommand{\ashP}{\as^P_{\hh}}

\newcommand{\bsht} {\widetilde{\bs}_{\hh}}
\newcommand{\bshtP}{\widetilde{\bs}^{\P}_{\hh}}

\newcommand{\SP} {S^{\P}}

\newcommand{\nlen}{\hspace{-0.2mm}}

\newcommand{\snorm}  [2]{|#1|_{#2}}

\newcommand{\norm}   [2]{|\nlen|#1|\nlen|_{#2}}

\newcommand{\NORM}   [2]{\left|\nlen\left|#1\right|\nlen\right|_{#2}}

\newcommand{\abs}    [1]{|#1|}

\newcommand{\ABS}    [1]{\left|#1\right|}

\newcommand{\Vhks}{\VS{\hh}_{k}}
\newcommand{\Vhk}{\VS{\hh}_{k}}

\newcommand{\Vhkt}{{\widetilde{V}}^{\hh}_{k}}
\newcommand{\calM} [1]{\mathcal{M}_{#1}}
\newcommand{\calMs}[1]{\mathcal{M}^{*}_{#1}}

\newcommand{\Pin}[1]{\Pi^{\nabla, \P}_{#1}}
\newcommand{\Piz}[1]{\Pi^{0, \P}_{#1}}

\newcommand{\PiFz}[1]{\Pi^{0,\F}_{#1}}
\newcommand{\PiFn}[1]{\Pi^{\nabla,\F}_{#1}}

\newcommand{\cbot}{c_*}
\newcommand{\ctop}{c^*}

\newcommand{\restrict}[2]{{#1}_{|{#2}}}

\newcommand{\EOD}{\end{document}}

\newcommand{\eigs}{\varepsilon}
\newcommand{\eigv}{\eigs}

\begin{document}


  \title{The virtual element method for eigenvalue problems with
    potential terms on polytopal meshes}

  \author[a]{O. Certik}
  \author[b]   {F. Gardini}
  \author[c] {G. Manzini}
  \author[d]  {G. Vacca}

  \affil[a]{
    Group CCS-2,
    Computer, Computational and Statistical Division,
    Los Alamos National Laboratory,
    Los Alamos, New Mexico - 87545, USA;
    \emph{e-mail: certik@lanl.gov}
  }
  \affil[b]{
    Dipartimento di Matematica \emph{F. Casorati},
    Universit\`a di Pavia,
    Via Ferrata, 5 - 27100 Pavia, Italy;
    \emph{e-mail: francesca.gardini@unipv.it}
  }
  \affil[c]{
    Group T-5,
    Theoretical Division,
    Los Alamos National Laboratory,
    Los Alamos, New Mexico - 87545, USA;
    \emph{e-mail: gmanzini@lanl.gov}
  }
  \affil[d]{
    Dipartimento di Matematica e Applicazioni, 
    Universit\`a di Milano Bicocca, 
    Via R. Cozzi, 55 - 20125 Milano, Italy;
    \emph{e-mail: giuseppe.vacca@unimib.it}
  }


\maketitle

\begin{abstract}
    We extend the conforming virtual element method to the numerical
    resolution of eigenvalue problems with potential terms on a
    polytopal mesh.
    An important application is that of the Schr\"odinger equation
    with a pseudopotential term.
    This model is a fundamental element in the numerical resolution of
    more complex problems from the Density Functional Theory.
    The VEM is based on the construction of the discrete bilinear
    forms of the variational formulation through certain polynomial
    projection operators that are directly computable from the degrees
    of freedom.
    The method shows a great flexibility with respect to the meshes
    and provide a correct spectral approximation with optimal convergence rates.
    This point is discussed from both the theoretical and the
    numerical viewpoint.
    The performance of the method is numerically investigated by
    solving the Quantum Harmonic Oscillator problem with the harmonic
    potential and a singular eigenvalue problem with zero potential
    for the first eigenvalues.
  \end{abstract}

\raggedbottom

\section{Introduction}
\label{sec:introduction}

The numerical treatment of the Schr\"odinger equation with local
pseudopotentials is one of the most expensive step in solving
electronic-structures in large-scale Density Functional calculations,
e.g., see~\cite{Bader:1991,Yang-Ayers:2003,Gross:2013}.
These calculations make it possible to determine properties of
materials from quantum-mechanical first principles (ab initio), hence
without the need of adaptable parameters.
A widely used approach for solving the Schr\"odinger equation in
large-scale quantum-mechanical physical systems is provided by the
\textit{the plane wave (PW) pseudopotential
  method}~\cite{pickett1989pseudopotential} and its many variants.
PW methods are spectral methods based on an expansion on Fourier basis
functions (the plane waves).
Such methods are generally accurate, but their computer implementation
may be inefficient as it normally relies on the \textit{fast Fourier
  transform (FFT)}, whose nonlocal communication pattern compromises
the method's scalability on parallel architectures.
Moreover, the strictly uniform resolution of a plane waves expansion
makes resolution adaptation infeasible, thus requiring to consider a
big number of PWs to capture the highly oscillatory behavior in the
atomic region.
Such a high resolution is unnecessary in transition regions between
atoms where the solution to the Schr\"odinger equation is normally much
smoother. 
This fact may eventually lead to computationally inefficient
discretizations~\cite{PASK20178}.

An alternative to the PW approach has been offered in recent years by
real-space methods such as finite differences, finite volumes, and
finite elements.
Such methods are based on the direct approximation of the solution of
the Schr\"odinger equation on a computational grid.
In particular, the finite element method (FEM) is a variational method
based on the expansion of the solution in shape basis functions,
usually piecewise polynomials that are strictly locally defined in
each mesh element.
As noted in~\cite{PASK20011,Pask-Sterne:2005}, the strictly local
nature of the shape functions has several important consequences.
First, the FEM produces sparse matrices that can efficiently be
treated by standard iterative methods (preconditioned Krylov
schemes)~\cite{Cai:2013:HPI};
the computational mesh can be refined near the atom locations where
the eigenfunctions are expected to vary the most in order to increase
the efficiency of the representation;
highly scalable implementations are possible on parallel machines.
Moreover, the accuracy of the method can easily be improved by
increasing the polynomial degree of the shape functions and
systematically enhanced by adding other nonpolynomial functions that
incorporates in the local approximation some physical insight from
the eigenfunction behavior~\cite{Sukumar:2009:CEF,Pask:2011:PUF}.

A very recent and important development in the field of the FEM
consists in the virtual element method (VEM), which was introduced in
\cite{BeiraodaVeiga-Brezzi-Cangiani-Manzini-Marini-Russo:2013} as a
variational reformulation of the Mimetic Finite Difference (MFD)
method~\cite{BdV-Lipnikov-Manzini:book,Lipnikov-Manzini-Shashkov:2014,BeiraodaVeiga-Lipnikov-Manzini:2011}.
The VEM is a very successful approach for the construction of
numerical approximation of any order of accuracy and regularity on
general polygonal/polyhedral meshes.
Despite its relative recentness (the first paper was published in
2013), the VEM has been developed
successfully for a large range of mathematical and engineering problems
\cite{
Chernov-BdV-Mascotto-Russo:2016,%
Stokes:divfree,%
Mascotto:2017b,%
BeiraodaVeiga-Manzini:2015,
Cangiani-Manzini-Russo-Sukumar:2015,%
Beirao2018,%
Antonietti-BdV-Scacchi-Verani:2016,%
Benedetto-Berrone:2016,%
Wriggers:2016,%
Chi-BdV-Paulino:2017,%
Vacca:2017,%
BdV-Brezzi-Dassi-Marini-Russo:2017,%
ADLP-HR,%
Vacca:2018},
extended to curved edges~\cite{BdV-FRusso-Vacca:2017},
and three dimensional problems
\cite{Ahmad-Alsaedi-Brezzi-Marini-Russo:2013,Dassi:2017,DASSI2018},
using also mixed spaces~\cite{Caceres2017}
and 
nonconforming spaces
\cite{Ayuso-Lipnikov-Manzini:2016,Mascotto-Perugia-Pichler:2018,Antonietti-Manzini-Verani:2018:M3AS:journal,Cangiani:2016:SINUM:journal,Cangiani:2017:IMAJNA:journal}. 
High-order and higher-order continuity schemes have been presented
in~\cite{Chernov-BdV-Mascotto-Russo:2016} and
~\cite{Brezzi-Marini:2013,BdV-Manzini:2013,Antonietti-BdV-Scacchi-Verani:2016},
respectively.

The VEM is indeed a finite element method, so all good properties of
the FEM when applied to the Schrodinger equation still hold.
In addition, we can exploit the great flexibility of the method, which
comes from the fact that the shape functions used in the variational
formulation are not known in a closed form, but are defined as the
solution of suitable differential problems.
This fact is also the motivation of the name ``virtual''.

The construction of the method and its practical implementations
relies on the special choice of the degrees of freedom rather than the
explicit knowledge of the local shape functions.
The degrees of freedom allow the calculation of certain projection
operators from local virtual element spaces into polynomial subspaces.
Using such operators, we can properly construct the discrete bilinear
forms that approximate the continuous ones of the variational
formulation in the virtual element framework.

The present work is the first instance of a long term project that
aims to extend the virtual element approach to the real-space
numerical approximation of the equations of the Density Functional
Theory.
We start here by considering the Schr\"odinger equation with local
pseudopotentials and Dirichlet/Neumann boundary conditions.
Despite its simplicity, this model allows us to compute the spectrum
of the classical Quantum Harmonic Oscillator.
We emphasize that the zero potential case, which corresponds to the
standard eigenvalue problem for the Laplace operator, also provides
the classical ``atom in a box'' model.
Previous works investigating the VEM for eigenvalue problems regard
the approximation of the Steklov eigenvalue problem
\cite{Mora-Rivera:2015, Mora-Rivera:2017}, the Laplace eigenvalue
problem~\cite{Gardini-Vacca:2017, Gardini-Manzini-Vacca:2018} 
with conforming and nonconforming virtual elements, respectively, 
the acoustic vibration
problem~\cite{Beirao-Mora-Rivera-Rodriguez:2017}, the vibration
problem of Kirchhoff plates~\cite{Mora-Rivera-Velasquez:2017},
the transmission eigenvalue problem  
\cite{Mora-Velasquez:2018}
whereas~\cite{Cangiani-Gardini-Manzini:2011} deals with the Mimetic
Finite Difference approximation of the eigenvalue problem in mixed
form.

The outline of the paper is as follows.
In Section \ref{sec:continuous problem} we recall the eigenvalue
problem under investigation, introducing the classical variational
formulation and the necessary notations.
Section \ref{sec:spaces} details the proposed discretization procedure. 
The approximation spaces and all the bilinear forms that define the
discrete problem, are introduced and described.
Section \ref{sec:analysis} deals with the theoretical analysis, which
leads to the optimal error estimates of Theorems~\ref{th:convl2},
~\ref{theorem:double:convergence:rate}, and~\ref{thm:h1estimate}.
In Section \ref{sec:tests} we present several numerical tests, which
highlight the actual performance of our approach. 
Finally, in Section~\ref{sec:conclusion} we offer our final remarks
and conclusions.


\section{The continuous eigenvalue problem}
\label{sec:continuous problem}
\subsection{Technicalities and definitions}
We use the standard definition and notation of Sobolev spaces, norms
and seminorms as given in~\cite{Adams:1975}.
Hence, the Sobolev space $\HS{s}(\omega)$ consists of functions
defined on the open bounded connected subset $\omega$ of
$\REAL^{\DIM}$, $\DIM=1,2,3$, that are square integrable and whose
weak derivatives up to order $s$ are also square integrable.
As usual, if $s=0$, we prefer the notation $\LTWO(\omega)$.
Norm and seminorm in $\HS{s}(\omega)$ are denoted by
$\norm{\cdot}{s,\omega}$ and $\snorm{\cdot}{s,\omega}$, respectively,
and $(\cdot,\cdot)_{\omega}$ denote the $\LTWO$-inner product.
The subscript $\omega$ may be omitted when $\omega$ is the whole computational
domain $\Omega$.

If $\ell\geq0$ is an integer number, $\PS{\ell}(\omega)$ is the space
of polynomials of degree up to $\ell$ defined on $\omega$, with the
convention that $\PS{-1}(\omega)=\{0\}$.
The $\LTWO$-orthogonal projection onto the polynomial space
$\PS{\ell}(\omega)$ is denoted by
$\Pi^{0,\omega}_{\ell}\,:\,\LTWO(\omega)\to\PS{\ell}(\omega)$.
Space $\PS{\ell}(\omega)$ is the span of the finite set of
\emph{scaled monomials of degree up to $\ell$}, that are given by
\begin{align}
  \calM{\ell}(\omega) =
  \left\{\,
    \left( {(\xv-\xv_{\omega})}\slash{\hh_{\omega}} \right)^{\alpha}
    \textrm{~with~}\abs{\alpha}\leq\ell
    \,\right\},
\end{align}
where 
\begin{itemize}
\item $\xv_{\omega}$ denotes the center of gravity of $\omega$ and
  $\hh_{\omega}$ its characteristic length, as, for instance, the edge
  length, the face diameter, or the cell diameter for $\DIM=1,2,3$;
\item $\alpha=(\alpha_1,\ldots,\alpha_{\DIM})$ is the
  $\DIM$-dimensional multi-index of nonnegative integers $\alpha_i$
  with degree $\abs{\alpha}=\alpha_1+\ldots+\alpha_{\DIM}\leq\ell$ and
  such that
  $\xv^{\alpha}=\xs_1^{\alpha_1}\cdots\xs_{\DIM}^{\alpha_{\DIM}}$ for
  any $\xv\in\REAL^{\DIM}$.
\end{itemize}
We will also use the set of \emph{scaled monomials of degree exactly
  equal to $\ell$}, denoted by $\calM{\ell}^{*}(\omega)$ and obtained
by setting $\abs{\alpha}=\ell$ in the definition above.

Finally, we use the letter $C$ in the estimates throughout the paper
to denote a strictly positive constant that is independent of the mesh size $\hh$,
but may depend on the problem constants, like the coercivity and
continuity constants, or other discretization constants like the mesh
regularity constant, the stability constants, etc. 
Constant $C$ generally has a different value at each occurrence.

\subsection{The continuous model}
Let $\Omega \subseteq\REAL^{\DIM}$ for $\DIM=2,3$ 
be the computational domain  and let $\Gamma$ be the boundary of $\Omega$.
We are interested in the numerical approximation of the eigenvalues
$\eigv\in\REAL$ and the eigenfunctions $\us:\Omega\to\REAL$, $\us\neq
0$, solving the following problem in strong form:
\begin{align}
  -\half\Delta\us(\xv) + V(\xv)\us(\xv)  = \eigv\us(\xv) \qquad\xv\in\Omega,
  \label{eq:Schrodinger}
\end{align}
where $V(\xv)$ is a smooth real-valued scalar potential function.
In the context of atomic and molecular quantum theory, $\eigv$ and
$\us$ are the energy level and the corresponding wavefunction of the
Hamilton operator $\mathcal{H}:=-\half\Delta+V(\xv)$, and potential
$V(\xv)$ collects all local and nonlocal terms from the Density
Functional Theory \cite{Bader:1991,Yang-Ayers:2003,Gross:2013}.
For a proper mathematical formulation, problem~\eqref{eq:Schrodinger}
is supplemented with suitable boundary conditions that, depending on
the problem, can be of Dirichlet, Neumann, or periodic type (if $\Omega$ is a parallelepiped). 
In the following we consider for sake of simplicity homogeneous  Dirichlet boundary conditions. 
The other cases easily follow the same construction.
Furthermore we assume that $V(\xv)$ is uniformly bounded from below and above, i.e.,
there exist two strictly positive constant $V_*$ and $V^*$ such that
$V_*\leq V(\xv)\leq V^*$ for almost every $\xv\in\Omega$.

\medskip
\noindent
The variational formulation of~\eqref{eq:Schrodinger} reads as:
\emph{Find $\eigv\in\REAL$ and $\us\in\HONEzr(\Omega)$,
  $\NORM{\us}{\LTWO(\Omega)}=1$, such that}
\begin{equation}
  \as(\us,\vs)=\eigv\bs(\us,\vs)\qquad\forall\vs\in\HONEzr(\Omega),
  \label{eq:eigPbm}
\end{equation}
where the bilinear form
$\as:\HONE(\Omega)\times\HONE(\Omega)\to\REAL$ is given by
\begin{align}
  \as(\us,\vs)
  = \int_{\Omega}\big( \half\nabla\vs(\xv)\cdot\nabla\us(\xv) + V(\xv)\us(\xv)\vs(\xv) \big)\,\dx
  \qquad\forall\us,\,\vs\in\HONE(\Omega),
  \label{eq:exact:a}
\end{align}
and the bilinear form $\bs:\LTWO(\Omega)\times\LTWO(\Omega)\to\REAL$
is the $\LTWO$-inner product on $\Omega$, i.e.,
\begin{align}
  \bs(\us,\vs)=\int_{\Omega}\us(\xv)\vs(\xv)\dx
  \qquad\forall\us,\,\vs\in\LTWO(\Omega).
  \label{eq:exact:b}
\end{align}

\begin{remark}
  From the standard eigenvalue theory,
  see~\cite{Boffi:2010,Babuska-Osborn:1991}, we know that 
  \begin{description}
  \item $(i)$ problem~\eqref{eq:eigPbm} admits a discrete infinite set
    of eigenvalues forming a positive increasing divergent
    sequence;
  \item $(ii)$ the corresponding eigenfunctions are an orthonormal
    basis of $\HONEzr(\Omega)$ with respect to the $\LTWO$-inner product
    and the scalar product associated with the bilinear form
    $\as(\cdot,\cdot)$;
  \item $(iii)$ the eigenvalues may have multiplicity bigger than one,
    but in such a case the corresponding eigenspace must have finite
    dimension.
  \end{description}
\end{remark}

\medskip
We also consider the source problem with 
homogeneous Dirichlet
boundary conditions: 
\emph{Find $\uss\in\HONEzr(\Omega)$ such that}
\begin{align}
  \as(\uss,\vs)=\bs(\fs,\vs)\qquad\forall\vs\in\HONEzr(\Omega),
  \label{eq:schrodinger:source-pblm:var-form}
\end{align}
where we assume that $\fs\in\LTWO(\Omega)$.
Well-posedness of problem~\eqref{eq:schrodinger:source-pblm:var-form},
i.e., existence and uniqueness of its solution, is proved by using
the Lax-Milgram Lemma~\cite{Ern-Guermond:2013} since, due to the
boundedness assumption on the potential field $\Vs$, the bilinear form
$\as$ in~\eqref{eq:exact:a} is coercive and the bilinear form $\bs$ in~\eqref{eq:exact:b} is continue. 
Moreover, due to regularity result~\cite{Agmon:1965, Grisvard:1992}, 
there exists $r>0$, depending only on $\Omega$, such that $\uss\in H^{1+r}(\Omega)$. 
Eventually, the following stability estimate holds
\begin{equation}
\label{eq:regularity}
\snorm{\uss}{1+r}\le C \norm{f}{0}.
\end{equation}

\section{The virtual element method}
\label{sec:spaces}

We are interested in developing the virtual element approximation of
the eigenvalue problem in variational form~\eqref{eq:eigPbm}.
To this end, we first discuss which meshes can be used for the
numerical formulation and introduce a proper set of regularity
assumptions.
Then, we define the local and global virtual element spaces, the
degrees of freedom and the bilinear forms $\ash$ and $\bsh$ that
approximate $\as$ and $\bs$.
Finally, we review the estimate of the convergence rate for the related
VEM approximation of the source problem.

\subsection{Mesh definition and regularity assumptions}
\label{subsec:mesh:regularity:assumptions}
Let $\Th$ denote a decomposition of the computational domain $\Omega$
into a finite set of polytopal elements $\P$.
As usual, the subindex $\hh$ that labels the mesh $\Th$ is the maximum of
the diameters $\hP=\sup_{\xv,\yv\in\P}\abs{\xv-\yv}$ of the elements
of the mesh.
We assume that the elements are nonoverlapping and for each element
$\P$ we denote its $(\DIM-1)$-dimensional nonintersecting boundary by
$\partial\P$; its center of gravity by $\xvP$; its $\DIM$-dimensional
measure by $\mP$.
The boundary of $\P$ is formed by straight edges when $\DIM=2$ and
flat faces when $\DIM=3$.
On 3D polyhedral meshes, we denote the midpoint and length of each
mesh edge $\E$ by $\xvE$ and $\hE$, respectively, and the center of
gravity, diameter and area of each face $\F$ are denoted by $\xvF$,
$\hF$, and $\mF$, respectively.
In the 2D case, we do not make any special distinction between the
terms ``\emph{edge}'' and ``\emph{face}'', which we consider as
synonyms.
To unify the notation we may use the symbol $\S$ instead of $\E$ or
$\F$ and, for example, refer to the geometric objects forming the
elemental boundary $\partial\P$ by the term \emph{side} instead of
\emph{edge}/\emph{face}.
According to such notation, we denote the center of gravity, diameter,
and measure of side $\S$ by $\xvS$, $\hS$, and $\mS$, respectively.


Consider the set $\mathcal{T}=\{\Th\}_{\hh}$ formed by the
decompositions of $\Omega$ for $\hh\to0$.
The convergence analysis of the conforming VEM we want to consider in
this work requires some regularity assumptions that must be satisfied
by all the members of mesh family $\mathcal{T}=\{\Th\}_{\hh}$.
For completeness we state these assumptions for both $\DIM=2$ and $\DIM=3$ case, although those for $\DIM=2$ can be derived from those for $\DIM=3$ by reducing the
spatial dimension.

\medskip
\noindent
\textbf{(A0)}~\textbf{Mesh regularity assumptions}.

\begin{itemize}
  \medskip
\item $\mathbf{\DIM=3}$.
  There exists a positive constant $\varrho$ independent of $\hh$
  (and, hence, of $\Th$) such that for every polyhedral element
  $\P\in\Th$ it holds that
  \begin{description}
  \item[$(i)$] $\P$ is star-shaped with respect to a ball with radius
    $\ge\varrho\hP$;
  \item[$(ii)$] every face $\F\in\P$ is star-shaped with respect to a
    disk with radius $\ge\varrho\hF$;
  \item[$(iii)$] for every edge $\E\in\partial\F$ of every face
    $\F\in\partial\P$ it holds that
    $\hE\geq\varrho\hF\geq\varrho^2\hP$;
  \end{description}

  \medskip
\item $\mathbf{\DIM=2}$.
  There exists a positive constant $\varrho$ independent of $\hh$
  (and, hence, of $\Th$) such that for every polygonal element
  $\P\in\Th$ it holds that
  \begin{description}
  \item[$(i)$] $\P$ is star-shaped with respect to a disk with radius
    $\ge\varrho\hP$;
  \item[$(ii)$] for every edge $\E\in\partial\P$ it holds that
    $\hE\geq\varrho\hP$;
  \end{description}
\end{itemize}

The scaling assumption implies that the number of edges and faces in
each elemental boundary is uniformly bounded over the whole mesh
family $\{\Th\}$.
The star-shapedness property implies that elements and faces are
\emph{simply connected} subsets of $\REAL^{\DIM}$ and
$\REAL^{\DIM-1}$, respectively.

\medskip
\noindent

\subsection{The conforming virtual element space}
\label{subsec:conforming:virtual:element:space}
We construct the local conforming virtual element space by resorting
to the so-called \emph{enhancement strategy} introduced
in~\cite{Ahmad-Alsaedi-Brezzi-Marini-Russo:2013}.
The construction of the conforming virtual element space in the
multidimensional case for $\DIM\geq 3$ is recursive.
We discuss here only the more general case for $\DIM=3$, while the
case for $\DIM=2$ follows from a simple dimensional reduction.

To this end, on every polygonal face $\F$ of the boundary $\partial\P$
and for any integer number $k\geq 1$ we first define the finite
element space
\begin{equation}\label{eq:def:Vhk0f}
  \Vhkt(\F) = \Big\{\,
  \vsh\in\HONE(\F)\cap\CS{0}(\overline{\F})\,:\,
  \restrict{\vsh}{\partial\F}\in\CS{0}(\partial\F),\,
  \restrict{\vsh}{\E}\in\PS{k}(\E)\,\forall\E\subset\partial\F,\,
  \Delta\vsh\in\PS{k}(\F)\,
  \Big\}.
\end{equation}
It is worth noting that the space of polynomials of degree up to
$k$ defined on $\F$ is a subspace of $\Vhkt(\F)$.
Then, we introduce the set of continuous linear functionals from
$\Vhkt(\F)$ to $\REAL$ that for every virtual function $\vsh$ of
$\Vhkt(\F)$ provide:
\begin{description}
\item[]\textbf{(D1)} the values of $\vsh$ at the vertices of $\F$;

\item[]\textbf{(D2)} the moments of $\vsh$ of order up to $k-2$ on
  each one-dimensional edge $\E\in\partial\F$:
  \begin{equation}\label{eq:dofs:01f}
    \frac{1}{\mE}\int_{\E}\vsh\,\ms\dS,
    \,\,\forall\ms\in\calM{k-1}(\E),\,
    \forall\E\in\partial\F;
  \end{equation}

\item[]\textbf{(D3)} the moments of $\vsh$ of order up to $k-2$ on
  each two-dimensional face $\F$:
  \begin{equation}\label{eq:dofs:02f}
    \frac{1}{\mF}\int_{\F}\vsh\,\ms\dS,
    \,\,\forall\ms\in\calM{k-2}(\F).
  \end{equation}
\end{description}
Finally, we introduce the elliptic projection operator
$\PiFn{k}:\Vhkt(\F)\to\PS{k}(\F)$ that for any $\vsh\in\Vhkt(\F)$ is
defined by:
\begin{align}
\label{eq:PiFn}
  \int_{\F}\nabla\PiFn{k}\vsh\cdot\nabla\qs\dx =
  \int_{\F}\nabla\vsh\cdot\nabla\qs\dx\quad\forall\qs\in\PS{k}(\F)
\end{align}
together with the additional conditions:
\begin{align}
  \int_{\partial\P}(\PiFn{k}\vsh-\vsh)\dS&=0 
  \qquad\textrm{if~}k=1,    \label{eq:def:Pib:1}\\[0.5em]
  \int_{\P}       (\PiFn{k}\vsh-\vsh)\dx&=0 
  \qquad\textrm{if~}k\geq 2.\label{eq:def:Pib:k}
\end{align}
As proved in~\cite{BeiraodaVeiga-Brezzi-Cangiani-Manzini-Marini-Russo:2013, Cangiani-Manzini-Sutton:2017}, the polynomial
projection $\PiFn{k}\vsh$ is computable from the values of the linear
functionals \textbf{(D1)}-\textbf{(D3)}.
Furthermore, $\PiFn{k}$ is a polynomial-preserving operator, i.e.,
$\PiFn{k}\qs=\qs$ for every $\qs\in\PS{k}(\F)$.

\medskip
The \emph{local conforming virtual element space of order $k$} on the
polygonal face $\F$ is the subspace of $\Vhkt(\F)$ defined as
\begin{equation}\label{eq:def:Vhk0F}
  \Vhk(\F) = \Big\{\,
  \vsh\in\Vhkt(\F)\textrm{~such~that~}
  (\vsh-\PiFn{k}\vsh,\ms)_{\F}=0
  \,\,\forall\ms\in\calMs{k-1}(\F)\cup\calMs{k}(\F)
  \,\Big\}.
\end{equation}
Space $\Vhk(\F)$ has the two important properties that we outline
below:
\begin{description}
\item[$(i)$] it still contains the space of polynomials of degree at
  most $k$;
\item[$(ii)$] the values provided by the set of continuous linear
  functionals \textbf{(D1)}-\textbf{(D3)} uniquely determine every
  function $\vsh$ of $\Vhk(\F)$ and can be taken as the
  \emph{degrees of freedom} of $\vsh$.
\end{description}

\medskip
Property $(i)$ above is a direct consequence of the space definition,
while property $(ii)$ follows from the unisolvency of the degrees of
freedom \textbf{(D1)}-\textbf{(D3)} that was proved
in~\cite{Ahmad-Alsaedi-Brezzi-Marini-Russo:2013}.
\begin{remark}
\label{rm:l2-proj-face}
Additionally, from the space definition, 
we have that the $\LTWO$-orthogonal projection $\PiFz{k}\vsh$ is
computable exactly using only the degrees of freedom of $\vsh$,
and
$\PiFz{k}\vsh=\PiFn{k}\vsh$ for $k=1,2$.
\end{remark}

To define the conforming virtual element space on the polyhedral cell
$\P$, we first need to introduce the ''\emph{extended}'' virtual
element space:
\begin{equation}\label{eq:def:Vhk0p}
  \Vhkt(\P) = \Big\{\,
  \vsh\in\HONE(\P)\cap\CS{0}(\overline{\P})\,:\,
  \restrict{\vsh}{\partial\P}\in\CS{0}(\partial\P),\,
  \restrict{\vsh}{\F}\in\Vhk(\F)\,\forall\F\subset\partial\P,\,
  \Delta\vsh\in\PS{k}(\P)\,
  \Big\}.
\end{equation}
The space $\Vhkt(\P)$ clearly contains the polynomials of degree $k$.
Now we introduce the set of continuous linear functionals from
$\Vhkt(\P)$ to $\REAL$ that 
are the obvious three-dimensional counterpart
of the linear operators  of the bi-dimensional case.
For every virtual function $\vsh$ of
$\Vhkt(\P)$ we provide \cite{Ahmad-Alsaedi-Brezzi-Marini-Russo:2013, Dassi:2017}:
\begin{description}
\item[]\textbf{(D1)} the values of $\vsh$ at the vertices of $\P$;

\item[]\textbf{(D2)} the moments of $\vsh$ of order up to $k-2$ on
  each one-dimensional edge $\E\in\partial\P$:
  \begin{equation}\label{eq:dofs:01p}
    \frac{1}{\mE}\int_{\E}\vsh\,\ms\dS,
    \,\,\forall\ms\in\calM{k-1}(\E),\,
    \forall\E\in\partial\P;
  \end{equation}

\item[]\textbf{(D3)} the moments of $\vsh$ of order up to $k-2$ on
  each two-dimensional face $\F\in\partial\P$:
  \begin{equation}\label{eq:dofs:02p}
    \frac{1}{\mF}\int_{\F}\vsh\,\ms\dS,
    \,\,\forall\ms\in\calM{k-1}(\F),\,
    \forall\F\in\partial\P;
  \end{equation}

\item[]\textbf{(D4)} the moments of $\vsh$ of order up to $k-2$ on
  $\P$:
  \begin{equation}\label{eq:dofs:03p}
    \frac{1}{\mP}\int_{\P}\vsh\,\ms\dx,
    \quad\forall\ms\in\calM{k-2}(\P).
  \end{equation}
\end{description}
Then we introduce the $H^1$-seminorm projection operator
$\Pin{k}:\Vhkt(\P)\to\PS{k}(\P)$ that for any $\vsh\in\Vhkt(\P)$ is defined by:
\begin{align}
\label{eq:Pin}
  \int_{\P}\nabla\Pin{k}\vsh\cdot\nabla\qs\dx =
  \int_{\P}\nabla\vsh\cdot\nabla\qs\dx\quad\forall\qs\in\PS{k}(\P)
\end{align}
coupled with the conditions:
\begin{align}
  \int_{\partial\P}(\Pin{k}\vsh-\vsh)\dS&=0 
  \qquad\textrm{if~}k=1,    \label{eq:def:Pib:1p}\\[0.5em]
  \int_{\P}       (\Pin{k}\vsh-\vsh)\dx&=0 
  \qquad\textrm{if~}k\geq 2.\label{eq:def:Pib:kp}
\end{align}
The polynomial projection $\Pin{k}\vsh$ can be computed in terms of  the values of the linear functionals \textbf{(D1)}-\textbf{(D4)}.
Finally, $\Pin{k}$ is  polynomial-preserving, i.e.,
$\Pin{k}\qs=\qs$ for every $\qs\in\PS{k}(\P)$.

\medskip
We are now ready to introduce the \emph{local conforming virtual
  element space of order $k$} on the polytopal element $\P$, which is
the subspace of $\Vhkt(\P)$ defined as follow:
\begin{equation}\label{eq:def:Vhk0}
  \Vhk(\P) = \Big\{\,
  \vs\in\Vhkt(\P)\textrm{~such~that~}
  (\vsh-\Pin{k}\vsh,\ms)_{\P}=0
  \,\,\forall\ms\in\calMs{k-1}(\P)\cup\calMs{k}(\P)
  \,\Big\}.
\end{equation}
We recall that, by construction, the local space $\Vhk(\P)$
enjoys the following fundamental properties
(see \cite{Ahmad-Alsaedi-Brezzi-Marini-Russo:2013, Dassi:2017}):
\begin{description}
\item[$(i)$] it still contains the space of polynomials of degree at
  most $k$;
\item[$(ii)$] the values provided by the set of continuous linear
  functionals \textbf{(D1)}-\textbf{(D4)} uniquely determine every
  function $\vsh$ of $\Vhk(\P)$ and can be taken as the
  \emph{degrees of freedom} of $\vsh$.
\item[$(iii)$] we can define an interpolation operator in
  $\Vhk(\P)$ with optimal approximation properties so that for
  every $\vs\in\HS{r+1}(\P)$ with $1\leq r\leq k$ the interpolant
  $\vsI$ satisfies the inequality:
  \begin{align}
    \norm{\vs-\vsI}{\LTWO(\P)} + \hP\snorm{\vs-\vsI}{\HONE(\P)}
    \leq C\hP^{r+1}\snorm{\vs}{\HS{r+1}(\P)},
    \label{eq:CFVEM:interpolant}
  \end{align}
  for some positive constant $C$ independent of $\hh$.
\end{description}

\medskip
As for the 2D case, the $\LTWO$-orthogonal projection $\Piz{k}\vsh$ is
computable in terms of the degrees of freedom of $\vsh$,
and $\Piz{k}\vsh=\Pin{k}\vsh$ for $k=1,2$.

\medskip
Finally, the \emph{global conforming virtual element space $\Vhk$
  of order $k\geq1$ subordinate to the mesh $\Th$} is obtained by
gluing together the elemental spaces $\Vhk(\P)$ to form a subspace
of the conforming space $\HONE(\Omega)$.
The formal definition reads as:
\begin{align}
  \Vhk:=\Big\{\,\vsh\in\HONEzr(\Omega)\,:\,\restrict{\vsh}{\P}\in\Vhk(\P)
  \,\,\,\forall\P\in\Th\,\Big\}.
  \label{eq:def:nvem0}
\end{align}
A set of degrees of freedom for $\Vhk$ is given by collecting the
values from the linear functionals \textbf{(D1)}-\textbf{(D4)} for all the mesh elements.
The unisolvence of such degrees of freedom
 is an immediate
consequence of their unisolvence on each local space $\Vhk(\P)$.

\subsection{The VEM for the eigenvalue problem}
\label{subsec:VEM:source:problem}
\medskip
The next step in the construction of our method is to define a discrete version of the bilinear
forms $\as$ and $\bs$ given in \eqref{eq:exact:a} and \eqref{eq:exact:b}.
First of all we split the bilinear form $\as$ into the sum of local terms:
\begin{align}
  &\as(\us,\vs) = \sum_{\P\in\Th}\asP(\us,\vs)
  \textrm{~~where~~}\asP(\us,\vs) = \int_{\P}\big(\half\nabla\us\cdot\nabla\vs+V\us\vs\big)\dx.
  \label{eq:extension}
\end{align}
and we note that
for an arbitrary pair $(\us, \vs) \in \Vhks \times \Vhks$ the quantity $
\as(\us, \vs)$ is not computable.
Then, following a standard procedure in the VEM framework, we consider a computable discrete local
bilinear form $\ash(\cdot,\cdot)$  given
by the sum of elemental contributions
\begin{align}
  \ash(\ush,\vsh) = \sum_{\P\in\Th}\ashP(\ush,\vsh),
\end{align}
where we define
\begin{align}
  \label{eq:discreteforms}
  \ashP(\ush,\vsh) 
  = \half\int_{\P}\Piz{k-1}\nabla\ush\cdot\Piz{k-1}\nabla\vsh\,\dx
  + \int_{\P}V\Piz{k}\ush\,\Piz{k}\vsh\,\dx
  + \SP\Big( \big(I-\Pin{k}\big)\ush, \big(I-\Pin{k}\big)\vsh \Big),
\end{align}
$\SP(\cdot,\cdot)$ being the stabilization term that will be discussed
in the following.
The bilinear form $\ashP$ depends on the orthogonal projections
$\Piz{k-1}\nabla\ush$ and $\Piz{k-1}\nabla\vsh$, which are computable
from the degrees of freedom of $\ush$ and $\vsh$,
respectively~\cite{Ahmad-Alsaedi-Brezzi-Marini-Russo:2013}.
In fact, starting from the definition of the orthogonal projection,
an integration by parts yields: 
\begin{align}
  \int_{\P}\Piz{k-1}\nabla\ush\cdot\qv\dV
  &= \int_{\P}\nabla\ush\cdot\qv\dV\qquad\forall\qv\in\left[\PS{k-1}(\P)\right]^{\DIM}\\[1em]
  &= -\int_{\P}\ush\nabla\cdot\qv\dV + \sum_{\F\in\partial\P}\int_{\F}\ush\norPF\cdot\qv\dS
\end{align}
where $\norPF$ denotes the unit outward normal to $f$.
The first integral on the last right-hand side is computable from
the degrees of freedom \textbf{(D4)} as it is the moment of $\ush$
against a polynomial of degree $k-2$ over $\P$.
The face integrals above are also computable since 
$$
\int_{\F}\ush\norPF\cdot\qv\dS = \int_{\F}\PiFz{k-1}\ush\norPF\cdot\qv\dS, 
$$
and the $\LTWO$-orthogonal projection $\PiFz{k-1}\ush$, as we have seen, is
computable exactly using only the degrees of freedom of $\ush$, c.f. Remark~\ref{rm:l2-proj-face}.

\medskip
The discrete form $\ashP(\cdot,\cdot)$ must satisfy the two fundamental
properties:
\begin{description}
\item[-] {\emph{$k$-consistency}}: for all $\vsh\in\Vhks$ and for all
  $\qs\in\PS{k}(\P)$ it holds
  \begin{align}
    \label{eq:k-consistency}
    \ashP(\vsh,\qs) = \asP(\vsh,\qs);
  \end{align}
\item[-] {\emph{stability}}: there exists two positive constants
  $\alpha_*,\,\alpha^*$, independent of $\hh$ and of $\P$, such that
  \begin{align}
    \label{eq:stability}
    \alpha_*\asP(\vsh,\vsh)
    \leq\ashP(\vsh,\vsh)
    \leq\alpha^*\asP(\vsh,\vsh)\quad\forall\vsh\in\Vhks.
  \end{align}
\end{description} 
Stability is ensured by adding the bilinear form $\SP$, which can be
any symmetric positive definite bilinear form on the element $\P$ for
which there exist two positive constants $\cbot$ and $\ctop$ such that
\begin{align}
  \cbot\asP(\vsh,\vsh)
  \leq\SP(\vsh,\vsh)
  \leq\ctop\asP(\vsh,\vsh)
  \quad\forall\vsh\in\Vhks(\P)\textrm{~with~}\Pin{k}\vsh=0.
  \label{eq:SP:stability}
\end{align}
Note that $\SP(\cdot,\cdot)$ must scale like $\asP(\cdot,\cdot)$,
namely $\SP(\cdot,\cdot)\simeq\hP^{d-2}$ (see also Section \ref{sec:tests}).


Following~\cite{Gardini-Vacca:2017, Gardini-Manzini-Vacca:2018}, we
consider two different discretizations of the eigenvalue
problem~\eqref{eq:eigPbm} that are obtained by considering two possible choices for the discretization of the bilinear form $\bs$ (cf. \eqref{eq:exact:b}).
We split $\bs$ into the local contributions
\begin{align}
  &\bs(\us,\vs) = \sum_{\P\in\Th}\bsP(\us,\vs)
  \textrm{~~where~~}\bsP(\us,\vs) = \int_{\P} \us(\xv)\,\vs(\xv)\dx.
  \label{eq:extensionb}
\end{align}
In the first choice we consider an approximated bilinear form $\bsh$, which satisfies the $k$--consistency
property but not the stability property (extending to $\bsh$ the definitions above). Therefore we simply take
\begin{equation}
\label{eq:discreteb}
b_h^{\P}(u_h,v_h)=\int_{\P}\Pi^{0, \P}_{k}u_h \, \Pi^{0, \P}_{k}v_h \, \text{d}\mathbf{x}.
\end{equation}
The second possible choice consists in considering a discrete bilinear form $\tilde{b}_h(\cdot,\cdot)$ which, as done for the discrete form $a_h(\cdot,\cdot)$, 
enjoys both the $k$--consistency property and the stability property. 
In particular we define 
\begin{equation}
\label{eq:discretebstab}
\tilde{b}_h^{\P}(u_h,v_h)=\int_{\P}\Pi^{0, \P}_{k}u_h \, \Pi^{0, \P}_{k}v_h \, \text{d}\mathbf{x} + \tilde{S}^{\P}\Big((I-\Pi_{k}^{0, \P})u_h,(I-\Pi_{k}^{0, \P})v_h\Big),
\end{equation}
where $\tilde{S}^{\P}$ is any positive definite bilinear form on the element $\P$ such that there exist two uniform positive 
constants $\beta_*$ and $\beta^*$ such that
$$
\beta_* \, b^{\P}(v_h,v_h)\le\tilde{S}^{\P}(v_h,v_h)\le\beta^* \, b^{\P}(v_h,v_h)
\quad\forall v_h\in V_h^k(\P) \text{ with }\Pi_k^{0, \P}v_h=0.
$$

\begin{remark}
\label{eq:scale2}
In analogy with the condition on the form $S^{\P}(\cdot,\cdot)$, 
we require that the form $\tilde{S}^{\P}(\cdot,\cdot)$ scales like $b^{\P}(\cdot,\cdot)$, that is $\tilde{S}^{\P}(\cdot,\cdot)\simeq h^d$. 
\end{remark}

The resulting virtual element scheme read as: 
\emph{Find $(\varepsilon_h,\ush)\in\REAL\times\Vhks$, $\NORM{\ush}{0}=1$, such that}
\begin{align}
  \label{eq:discreteEigPbm}
  \ash(\ush,\vsh) = \varepsilon_{\hh}\bsh(\ush,\vsh)\quad\forall\vsh\in\Vhks.
\end{align}
if we adopt the first choice $\bsh$ for the approximation of $\bs$.
\medskip
\noindent
The second virtual element formulation reads as: 
\emph{Find $(\widetilde{\varepsilon}_{\hh},\usht)\in\REAL\times\Vhks$,
  $\NORM{\usht}{0} = 1$, such that}
\begin{align}\label{eq:discreteEigPbm2}
  \ash(\usht,\vsh) = \widetilde{\varepsilon}_{\hh}\bsht(\usht,\vsh)\quad\forall\vsh\in\Vhks,
\end{align}
where we consider the stabilized bilinear form $\bsht$.

Finally, in what follows, we will also need the discrete source problem corresponding to both discrete
formulations \eqref{eq:discreteEigPbm} and \eqref{eq:discreteEigPbm2}, which reads respectively as:
\emph{Find $\ussh\in \Vhks$ such that}
\begin{align}
  \ash(\ussh,\vsh)=\bsh(\fs,\vsh)\qquad\forall\vsh\in \Vhks.
  \label{eq:schrodinger:source-pblm:var-form1}
\end{align}
and 
\emph{find $\ussht\in \Vhks$ such that}
\begin{align}
  \ash(\ussht,\vsh)=\bsht(\fs,\vsh)\qquad\forall\vsh\in \Vhks,
  \label{eq:schrodinger:source-pblm:var-form2}
\end{align}

The well-posedness of the discrete formulations 
 \eqref{eq:schrodinger:source-pblm:var-form1} and
 \eqref{eq:schrodinger:source-pblm:var-form2}
 stem from the coercivity of the bilinear form $\ash$ and form the continuity of the forms $\bsh$ and $\bsht$.
 
We finally observe that
both bilinear forms  are fully computable for any couple of functions $(\ush, \vsh) \in \Vhks$, since 
the enhancement technique implies that $\Piz{k}\ush$ (resp. $\Piz{k}\vsh$) can be
computed using only the degrees of freedom of $\ush$ (resp. $\vsh$).


\medskip 
The following convergence estimate theorem holds for the approximation 
of the source problem~\cite{Ahmad-Alsaedi-Brezzi-Marini-Russo:2013}.
\begin{theorem}\label{theorem:GDM:apriori:estimate}
  Let $\uss\in\HS{r+1}(\Omega)$  
  be the solution to the
  variational problem~\eqref{eq:schrodinger:source-pblm:var-form} with
  $\fs\in\LTWO(\Omega)$.
  Let $\ussh\in\Vhks$ be the solution of the virtual element method
  \eqref{eq:schrodinger:source-pblm:var-form1}, 
  $\ussht\in\Vhks$ be the solution of the virtual element method
  \eqref{eq:schrodinger:source-pblm:var-form2} and denote by $\fs_{\hh}$ the 
  piecewise $\LTWO$-projection of $\fs$ onto the space $\mathbb{P}_{k}(\P)$.
   Under the mesh regularity assumption~\textbf{(A0)}, 
  let $t=min(k,r)$, and $\vsh \in \{\ussh, \, \ussht\}$ then it holds
  \begin{itemize}
  \item $\HONE$-error estimate:
    \begin{align}
      \label{eq:source:problem:H1:error:bound}
      \snorm{ \uss-\vsh }{\HONE(\Omega)}\leq C \left( 
        \hh^{t}\snorm{\uss}{\HS{r+1}(\Omega)} 
        + \hh\norm{\fs-\fsh}{\LTWO(\Omega)}
      \right).
    \end{align}

    \medskip
  \item $\LTWO$-error estimate (for a convex $\Omega$):
    \begin{align}
      \label{eq:source:problem:L2:error:bound}
      \norm{ \uss-\vsh}{\LTWO(\Omega)} \leq 
      C \left( \hh^{t+1}
        \snorm{\uss}{\HS{r+1}(\Omega)}
        + \hh\norm{\fs-\fsh}{\LTWO(\Omega)}
      \right).
    \end{align}
  \end{itemize}
\end{theorem}

\begin{remark}
\label{rm:estimate-regular-load}
Note that if $\uss$ is an eigenfunction of the continuous 
eigenvalue problem~\eqref{eq:eigPbm}, 
then it solves the continuous source problem~\eqref{eq:schrodinger:source-pblm:var-form} 
with datum $\varepsilon \uss$ and thus it belongs to $\HS{1+r}(\Omega)$ with 
 $\snorm{\uss}{1+r} \le C\norm{\uss}{0}$. Then, the a priori
 error estimates in Theorem ~\ref{theorem:GDM:apriori:estimate} reduce to\medskip
    \begin{itemize}
    \item $\HONE$-error estimate:
    \begin{align*}
    \snorm{ \uss-\vsh }{1} &\leq C \left( 
        \hh^{t}\snorm{\uss}{1+r} 
        + \hh \sum_{P\in\Th} \norm{(I-\Piz{k})\uss}{0}
      \right) 
      \le  \hh^{t}\snorm{\uss}{1+r}  \leq C \hh^{t}\norm{\uss}{0} \leq C \hh^{t},
    \end{align*}
    \item $\LTWO$-error estimate (for a convex $\Omega$):
    \begin{align*}
      \norm{ \uss - \vsh }{0} &\leq 
      C \left( \hh^{t+1}
        \snorm{\uss}{1+r}
        + \hh \sum_{P\in\Th} \norm{(I-\Piz{k})\uss}{0} \right) 
        \leq C \hh^{t+1} \snorm{\uss}{1+r}
        \leq C \hh^{t+1}\norm{\uss}{0} \leq C \hh^{t+1},
    \end{align*}
    \end{itemize}
    since
    \begin{equation*}
      \norm{(I-\Piz{k})\uss}{0}\leq C \hh^{\min\{k+1,1+r\}}\snorm{\uss}{1+r} \leq C \hh^{t+1}\norm{\uss}{0}.
    \end{equation*}
\end{remark}


\section{Convergence analysis and error estimates}
\label{sec:analysis}

\subsection{Spectral approximation for compact operators}
\label{sc:compact}
In this section, we briefly recall the results of the spectral approximation 
theory for compact operators. 
For more general results, we refer to the original papers~\cite{Babuska-Osborn:1991,Boffi:2010,Kato:1976}.  

We introduce a 
natural compact operator associated with problem~\eqref{eq:eigPbm} and
its discrete counterpart, and we recall their connection with the
eigenmode convergence.

\medskip
Let $T\in\mathcal{L}(\LTWO(\Omega))$ be the solution operator 
associated with problem~\eqref{eq:eigPbm}. $T$ is the bounded linear operator 
$T:\LTWO(\Omega)\rightarrow\LTWO(\Omega)$
which maps the forcing term $\fs$ to $\us=:T\fs$:
\begin{align*}
  \label{eq:T}
  \left\{
    \begin{array}{l}
      T\fs\in\HONE_0\textrm{~such~that}\\
      \as(T\fs,\vs) = \bs(\fs,\vs)\quad\forall\vs\in\HONE_0.
    \end{array}
  \right.
\end{align*}
Operator $T$ is self-adjoint and positive definite with respect to the
bilinear forms $\as(\cdot, \cdot)$ and $\bs(\cdot, \cdot)$ on $\HONE(\Omega)$, and compact due to the
compact embedding of $\HONE(\Omega)$ in $\LTWO(\Omega)$.

\medskip
Similarly, let $T_{\hh}\in\mathcal{L}(\LTWO(\Omega))$ 
and $\widetilde{T}_{\hh}\in\mathcal{L}(\LTWO(\Omega))$ 
be the discrete solution operators 
associated with the \emph{stabilized} and \emph{non-stabilized} discrete source problems. 
The former is the bounded linear operator mapping the forcing term
$\fs$ to $\ush=:T_{\hh}\fs$ and satisfies:
\begin{equation*}
  \left\{
    \begin{array}{l}
      \label{eq:Th}
      T_{\hh}\fs\in\Vhks\textrm{~such~that}\\
      \ash(T_{\hh}\fs,\vsh) = \bsh(\fs,\vsh)\quad\forall\vsh\in\Vhks.
    \end{array}
  \right.
\end{equation*}
The latter is the bounded linear operator mapping the forcing term
$\fs$ to $\usht=:\widetilde{T}_{\hh}\fs$ and satisfies:
\begin{align*}
  \left\{
    \begin{array}{l}
      \label{eq:Thtilde}
      \widetilde{T}_{\hh}\fs\in\Vhks\textrm{~such~that}\\
      \ash(\widetilde{T}_{\hh}\fs,\vsh) = \bsht(\fs,\vsh)\quad\forall\vsh\in\Vhks.
    \end{array}
  \right.
\end{align*}

Both operators $T_{\hh}$ and $\widetilde{T}_{\hh}$ are self-adjoint
and positive definite with respect to the discrete bilinear form 
$\ash(\cdot, \cdot)$, $\bsh(\cdot, \cdot)$ and $\ash(\cdot, \cdot)$, $\bsht(\cdot, \cdot)$.
They are also compact since their ranges are finite dimensional.

\medskip
\noindent
The eigensolutions of the continuous problem~\eqref{eq:eigPbm} and the
discrete problems ~\eqref{eq:discreteEigPbm}
and~\eqref{eq:discreteEigPbm2} are respectively related to the
eigenmodes of the operators $T$, $T_{\hh}$, and $\widetilde{T}_{\hh}$.
In particular, $(\varepsilon,\us)$ is an eigenpair of
problem~\eqref{eq:eigPbm} if and only if $T\us=(1\slash{\varepsilon})\us$,
i.e. $(\frac{1}{\varepsilon}, \us)$ is an eigenpair for the operator $T$,
and analogously for problems ~\eqref{eq:discreteEigPbm}
and~\eqref{eq:discreteEigPbm2} and operators $T_{\hh}$ and
$\widetilde{T}_{\hh}$.
Thanks to this correspondence, the convergence analysis can be
derived from the spectral approximation theory for compact operators.
In the rest of this section we refer only to operators $T$ and
$\widetilde{T}_{\hh}$.
Identical considerations hold for operators $T$ and $T_{\hh}$
and we omit them for brevity.

\medskip
A sufficient condition for the correct spectral approximation of a
compact operator $T$ is the uniform convergence to $T$ of the family
of discrete operators $\{\widetilde{T}_{\hh}\}_{\hh}$ (see~\cite[Proposition
7.4]{Boffi:2010}, cf. also~\cite{Babuska-Osborn:1991}):
\begin{equation}
  \label{eq:unifconv}
  \norm{T-\widetilde{T}_{\hh}}{\mathcal{L}(\LTWO(\Omega))}\to 0, \quad\textrm{as~}\hh\to 0,
\end{equation}
or, equivalently, 
\begin{equation}
  \label{eq:unifconv2}
  \norm{(T-\widetilde{T}_{\hh})\fs}{0}\leq C\rho(\hh)\norm{\fs}{0}
  \quad\forall\fs\in\LTWO(\Omega),
\end{equation}
with $\rho(\hh)$ tending to zero as $\hh$ goes to zero. 
Condition~\eqref{eq:unifconv2} usually follows by a-priori estimates
with no additional regularity assumption on $\fs$. 
Besides the convergence of the eigenmodes,
condition~\eqref{eq:unifconv}, or the equivalent
condition~\eqref{eq:unifconv2}, implies that no spurious eigenvalues
may pollute the spectrum.
In fact, each discrete eigenvalue approximates a continuous eigenvalue and 
each continuous eigenvalue is approximated by a number of 
discrete eigenvalues (counted with their multiplicity) that 
corresponds exactly to its multiplicity.

We now report the main results about the spectral approximation for compact operators.
(cf.~\cite[Theorems 7.1--7.4]{Babuska-Osborn:1991}; see
also~\cite[Theorem 9.3--9.7]{Boffi:2010}),
which deal with the order of
convergence of eigenvalues and eigenfunctions.
\begin{theorem}
  \label{thm:BOeigfun}
  Let the uniform convergence~\eqref{eq:unifconv} holds true. 
  Let $\mu$ be an eigenvalue of $T$, with multiplicity $m$, and denote 
  the corresponding eigenspace by $E_{\mu}$. 
  Then, exactly $m$ discrete eigenvalues $\widetilde{\mu}_{1, \hh}, \dots, \widetilde{\mu}_{m, \hh}$
  (repeated according to their respective multiplicities) converges to $\mu$.  
  Moreover, let $\widetilde{E}_{\mu,\hh}$ be the direct sum of the eigenspaces
  corresponding to the discrete eigenvalues
  $\widetilde{\mu}_{1,\hh},\cdots,\widetilde{\mu}_{m,\hh}$ converging to $\mu$.
  Then
  \begin{equation}
    \delta(E_{\mu}, \widetilde{E}_{\mu,\hh})\leq
    C \norm{(T-\widetilde{T}_{\hh})_{|E_{\mu}}}{\mathcal{L}(\LTWO(\Omega))},
    \label{eq:eigenmodeRate}
  \end{equation}
  with 
  \begin{align*}
    \delta(E_{\mu},\widetilde{E}_{\mu,\hh}) = \max( \hat\delta(E_{\mu},\widetilde{E}_{\mu,\hh}) , \hat\delta(\widetilde{E}_{\mu,\hh},E_{\mu}) ) 
  \end{align*}
   where, in general,
   \[
     \hat\delta(U, \, W) = \sup_{\us\in U, \NORM{\us}{0} = 1}\inf_{w\in W}\norm{\us-w}{0}
  \]
  denotes the gap between $U$, $W \subseteq L^2(\Omega)$.
\end{theorem}

Concerning the eigenvalue approximation error, we recall the following result.

\begin{theorem}
  \label{th:BOeig}
  Let the uniform convergence~\eqref{eq:unifconv} holds true. 
  Let $\phi_1,\ldots,\phi_m$ be a basis of the eigenspace $E_{\mu}$ of
  $T$ corresponding to the eigenvalue $\mu$.
  Then, for $i=1,\ldots,m$
  \begin{equation}
    \label{eq:eigRate}
    \abs{\mu - \widetilde{\mu}_{i,\hh}} \le C 
    \displaystyle
    \Big(\,\sum_{j,k=1}^m \abs{ b((T-\widetilde{T}_{\hh})\phi_k,\phi_j) } + 
    \norm{ (T-\widetilde{T}_{\hh})_{|E_{\mu}} }{\mathcal{L}(\LTWO(\Omega))}^2\,\Big),
  \end{equation}
  where $\widetilde{\mu}_{1,\hh},\ldots,\widetilde{\mu}_{m,\hh}$ are the $m$ discrete
  eigenvalues converging to $\mu$ repeated according to their
  multiplicities.
\end{theorem}

\subsection{Convergence analysis for the stabilized formulation}
\label{subsec:convnonstab}

In this section we study the convergence of the discrete eigenmodes
provided by the VEM approximation to the continuous ones.
We will consider the stabilized discrete
formulation~\eqref{eq:discreteEigPbm2}. The analysis can be easily applied 
to the non--stabilized one~\eqref{eq:discreteEigPbm}.

\begin{theorem}
  \label{thm:thm1}
  The family of operators $\widetilde{T}_{\hh}$ 
  converges uniformly to the operator
  $T$, that is,  
  \begin{equation}
    \label{eq:uniformconv}
    \norm{T-\widetilde{T}_{\hh}}{\mathcal{L}(\LTWO(\Omega))}\to 0
    \quad\textrm{for}\quad\hh\to 0.
  \end{equation}
\end{theorem}
\begin{proof}
  The proof follows the same lines as the one of Theorem 6.4 
  in ~\cite{Gardini-Vacca:2017}. We recall it here for the 
  convenience of the reader. 
  Let $\uss$ and $\ussht$ be the solutions to the continuous and the
  discrete source problems~\eqref{eq:schrodinger:source-pblm:var-form}
  and~\eqref{eq:schrodinger:source-pblm:var-form1}, respectively.
  From the $\LTWO$-estimate of Theorem~\ref{theorem:GDM:apriori:estimate} with
  $\fs\in\LTWO(\Omega)$  and the stability estimate in~\eqref{eq:regularity} 
  we have that
  \begin{align*} 
    \norm{ \uss-\ussht }{0} \leq C\hh^{\min(t+1,1)}\norm{\fs}{0}
  \end{align*}
  with $t=\min(k,r)$, $k\geq 1$ being the order of the method and $r$
  being the regularity index of the
  solution $\uss\in\HS{1+r}(\Omega)$ to the continuous source problem. 
  Then it follows that
  \begin{align*}
    \norm{T-\widetilde{T}_{\hh}}{\mathcal{L}(\LTWO(\Omega))} 
    = \sup_{\fs\in\LTWO(\Omega)}\dfrac{\norm{T\fs-\widetilde{T}_{\hh}\fs}{0}}{\norm{\fs}{0}} 
    = \sup_{\fs\in\LTWO(\Omega)}\dfrac{\norm{\uss-\ussht}{0}}{\norm{\fs}{0}}
    \leq C\hh^{\min(t+1,1)}.
  \end{align*}
\end{proof}

We remark that if $f\in\mathcal{E}_{\mu}$ then, thanks to the $L^2$ a priori error estimate 
   in Remark~\ref{rm:estimate-regular-load}, it holds 
   \begin{align*}
   \norm{(T-\widetilde{T}_{\hh})_{|E_{\mu}}}{\mathcal{L}(\LTWO(\Omega))}
   = \sup_{\fs\in\mathcal{E}_{\mu}}\dfrac{\norm{T\fs-\widetilde{T}_{\hh}\fs}{0}}{\norm{\fs}{0}} 
    = \sup_{\fs\in\mathcal{E}_{\mu}}\dfrac{\norm{\uss-\ussht}{0}}{\norm{\fs}{0}}
    \leq C\hh^{t+1}.
   \end{align*}

Putting together Theorem~\ref{thm:BOeigfun},~ Theorem~\ref{thm:thm1}, and 
the above observation, we can state the following result. 
 \begin{theorem}
   Let $\mu$ be an eigenvalue of $T$, with multiplicity $m$, and
   denote the corresponding eigenspace by $E_{\mu}$.  
   Then, exactly $m$ discrete eigenvalues $\widetilde{\mu}_{1, \hh},
   \dots, \widetilde{\mu}_{m, \hh}$ (repeated according to their
   respective multiplicities) converges to $\mu$.
  Moreover, let $\widetilde{E}_{\mu,\hh}$ be the direct sum of the eigenspaces
  corresponding to the discrete eigenvalues
  $\widetilde{\mu}_{1,\hh},\cdots,\widetilde{\mu}_{m,\hh}$ converging to $\mu$.
  Then
  \begin{equation}
    \delta(E_{\mu}, \widetilde{E}_{\mu,\hh})\leq
    C \hh^{t+1}.
  \end{equation}
 \end{theorem}

A direct consequence of the previous result
(cf.~\cite{Babuska-Osborn:1991,Boffi:2010}) is the following one.
\begin{theorem}
  \label{th:convl2}
  Let $\us$ be a unit eigenfunction associated with the eigenvalue
  $\varepsilon$ of multiplicity $m$ and let $\wsht^{(1)},\ldots,\wsht^{(m)}$
  denote linearly independent eigenfunctions associated with the $m$
  discrete eigenvalues of problem \eqref{eq:discreteEigPbm2} converging to $\varepsilon$. 
  Then there exists $\usht\in\SPAN{\wsht^{(1)},\ldots,\wsht^{(m)}}$
  such that
  \begin{align*}
    \norm{\us-\usht}{0}\leq C\hh^{t+1},
  \end{align*}
  where $t=\min\{k,r\}$, being $k$ the order of the method and $r$ the
  regularity index of $\us$.
\end{theorem}

\medskip

We now state the usual double order convergence of the eigenvalues.

\begin{theorem}
  \label{theorem:double:convergence:rate}
  Let $\varepsilon$ be an eigenvalue of problem~\eqref{eq:eigPbm} with multiplicity $m$, 
  and denote by $\widetilde{\varepsilon}_{1,h},\cdots,\widetilde{\varepsilon}_{m,h}$ the $m$ discrete eigenvalues of problem \eqref{eq:discreteEigPbm2} 
  converging towards $\varepsilon$. 
  Then the following optimal double order convergence holds:
  \begin{align}
    \label{eq:double:convergence:rate}
    \abs{\varepsilon - \widetilde{\varepsilon}_{i,\hh}}
    \leq C \hh^{2t}\quad\forall i=1,\ldots,m, 
  \end{align}
  with $t=\min\{k,r\}$, being $k$ the order of the method and $r$ the
  regularity index of the eigenfunction corresponding to $\varepsilon$.
\end{theorem}
\begin{proof}
The proof follows the guidelines of Theorem 6.4 
in~\cite{Gardini-Vacca:2017} 
and Theorems 4.2 and 4.3 in \cite{Mora-Rivera-Velasquez:2017}. 
For an alternative proof see also Theorem 6.6. 
in~\cite{Gardini-Manzini-Vacca:2018}, taking into account that in our case 
the term $\mathcal{N}_{\hh}(\cdot,\cdot)$, relative to the \emph{conformity} error, vanishes.
\end{proof}

Eventually, we state the optimal error estimate for the eigenfunctions in the energy norm. 
\begin{theorem}
\label{thm:h1estimate}
  With the same notation as in Theorem~\ref{th:convl2}, we have
  \begin{align*}
    \snorm{\us-\usht}{1}\le C\hh^{t}, 
  \end{align*}
  where $t=\min(k,r)$, $k$ being the order of the method and $r$ the
  regularity index of $\us\in H^{1+r}(\Omega)$.
\end{theorem}
\begin{proof}
The proof of this result is similar to the one for the finite element method.  
We briefly report it here for the sake of completeness. 
It holds that
  \begin{equation*}
    \us - \usht = \varepsilon T\us - \widetilde{\varepsilon}_{\hh} \widetilde{T}_{\hh}\usht 
    = (\varepsilon - \widetilde{\varepsilon}_{\hh})T\us + \widetilde{\varepsilon}_{\hh}(T - \widetilde{T}_{\hh})\us + \widetilde{\varepsilon}_{\hh} \widetilde{T}_{\hh}(\us - \usht),
  \end{equation*}
  then 
  \begin{align*}
    \snorm{\us -\usht}{1}
    \leq \abs{\varepsilon - \widetilde{\varepsilon}_{\hh}} \, \snorm{T\us}{1} 
    + \widetilde{\varepsilon}_{\hh}\snorm{(T-\widetilde{T}_{\hh})\us}{1} 
    + \widetilde{\varepsilon}_{\hh} \snorm{\widetilde{T}_{\hh}(\us - \usht)}{1}.
  \end{align*}
  The first term at the right-hand side of the previous equation is of
  order $\hh^{2t}$, while the second one is of order $\hh^{t}$.
  Finally, for the last term, using \eqref{eq:stability}, the continuity of the operator $\widetilde{T}_{\hh}$, and Theorem \ref{th:convl2}, we obtain
  \begin{align*}
    \snorm{ \widetilde{T}_{\hh}(\us-\usht)}{1}^2 
    &\leq \dfrac{1}{\alpha_*}\ash(\widetilde{T}_{\hh}(\us - \usht),\widetilde{T}_{\hh}(\us - \usht)) \\
    &= \dfrac{1}{\alpha_*}\bsht(\us-\usht,\widetilde{T}_{\hh}(\us - \usht)) 
    \le C\norm{\us-\usht}{0}^2 
    \le C \hh^{2t+2}.
  \end{align*}
\end{proof}


\section{Numerical experiments}
\label{sec:tests}

In this section, we investigate the behavior of our virtual element
method for the numerical treatment of the eigenvalue
problem~\eqref{eq:eigPbm}.
In particular, we present the performance of the conforming VEM
applied to the eigenvalue problem on a two-dimensional square domains.
We use the ``diagonal'' stabilization \cite{VemGuide} for the
bilinear form $\ashP(\cdot,\cdot)$ (cf. \eqref{eq:discreteforms}) and $\bshtP(\cdot,\cdot)$ (cf. \eqref{eq:discretebstab}, which reads
as follows:
\begin{align*}
  S^{\P}(\vsh,\wsh) &= \sigma_{\P}\vvh^T\wvh,\\
 \widetilde{S}^{\P}(\vsh,\wsh) &= \tau_{\P} h_{\P}^2 \vvh^T\wvh,
\end{align*}
where $\vvh$, $\wvh$ denote the vectors containing the values of the
local degrees of freedom associated to $\vsh$, $\wsh\in\Vhk{}(\P)$
and the parameters $\sigma_{\P}$ and $\tau_{\P}$ are two positive $h$-independent
constants. In the numerical tests we choose $\sigma_{\P}$ as the mean value of the eigenvalues of the matrix stemming from the
consistency term $(\Piz{k-1}\nabla\cdot,\Piz{k-1}\nabla\cdot)_{\P}$
for the grad-grad form (see \eqref{eq:discreteforms}). In the
same way we pick $\tau_{\P}$ as the mean value of the eigenvalues of the matrix
resulting from the term 
$\frac{1}{h_{\P}^2} (\Pi_k^{0, \P} \cdot, \Pi_k^{0, \P} \cdot)_{\P}$ 
for the mass matrix (see \eqref{eq:discretebstab}).

The convergence of the numerical approximation is shown through the
relative error quantity
\begin{align*}
  \text{\textit{Relative approximation error}} := \frac{\abs{\varepsilon-\varepsilon_{\hh}}}{\varepsilon},
\end{align*}
where $\varepsilon$ denotes an eigenvalue of the continuous problem and
$\varepsilon_{\hh}$ its virtual element approximation.

\subsection{Test~1.}

\begin{figure}
  \centering
  \begin{tabular}{cccc}
    \begin{overpic}[scale=0.2]{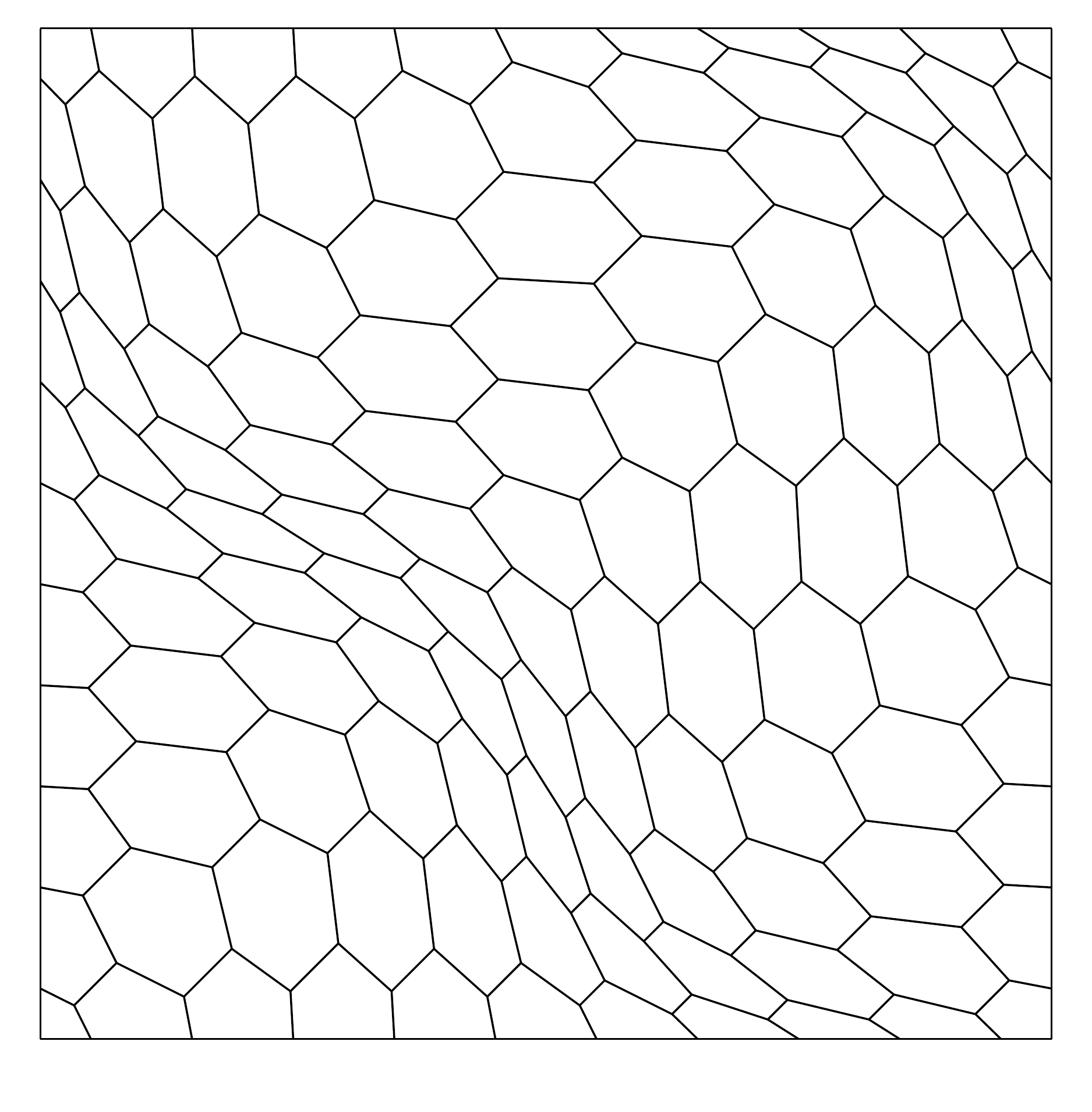}
    \end{overpic} 
    &
    \begin{overpic}[scale=0.2]{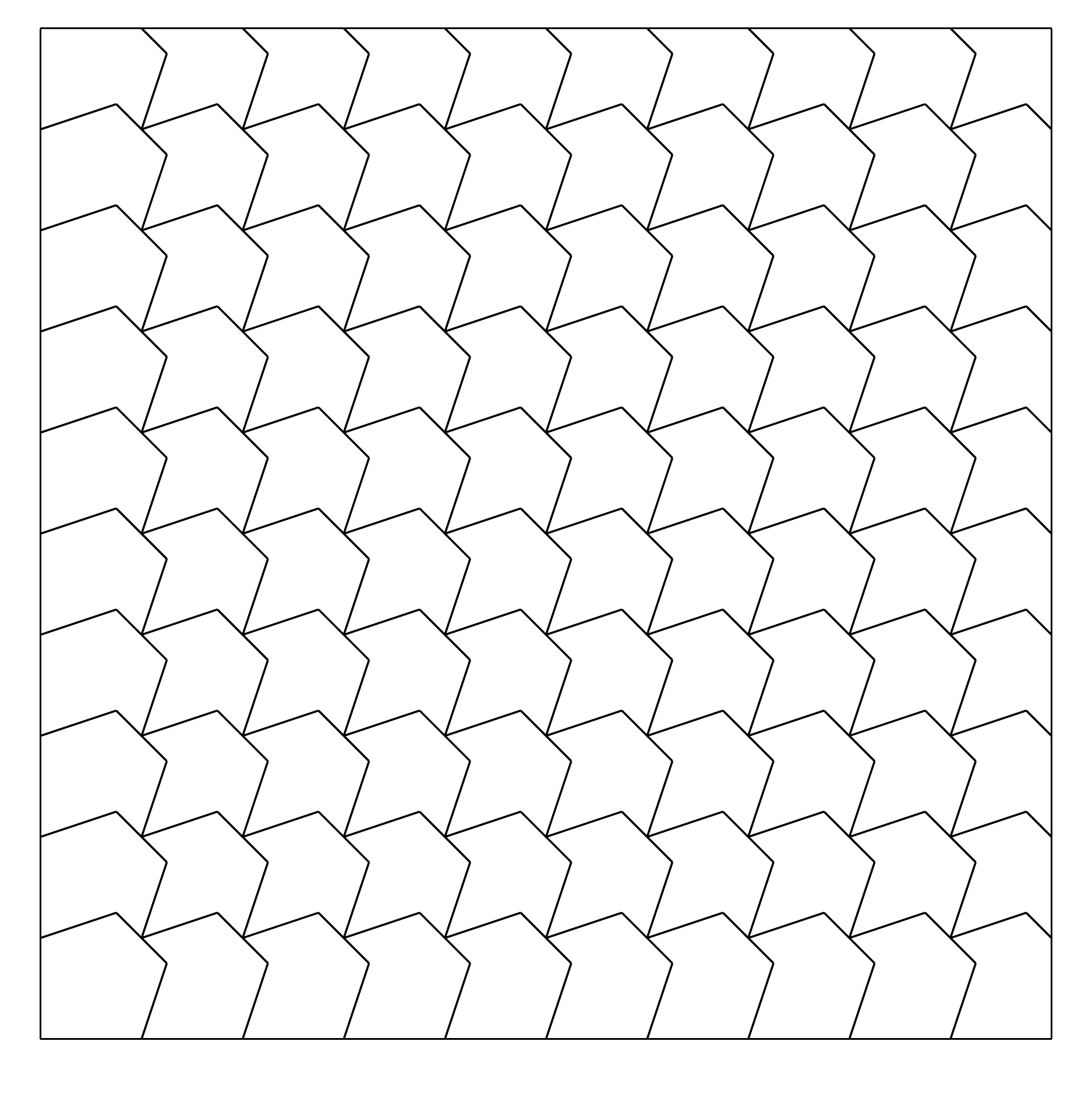}
    \end{overpic}
    &
    \begin{overpic}[scale=0.2]{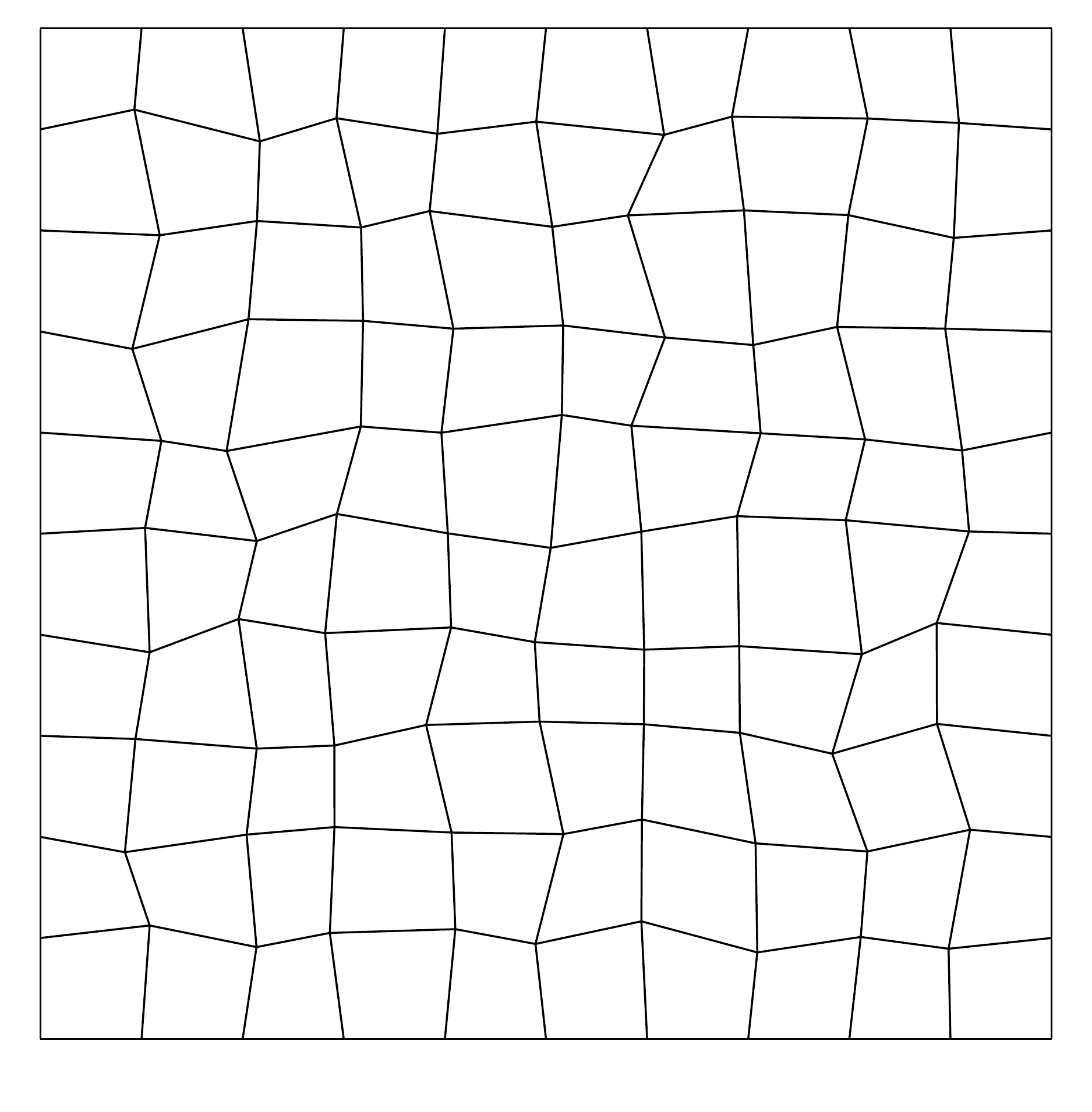}
    \end{overpic} 
    &
    \begin{overpic}[scale=0.2]{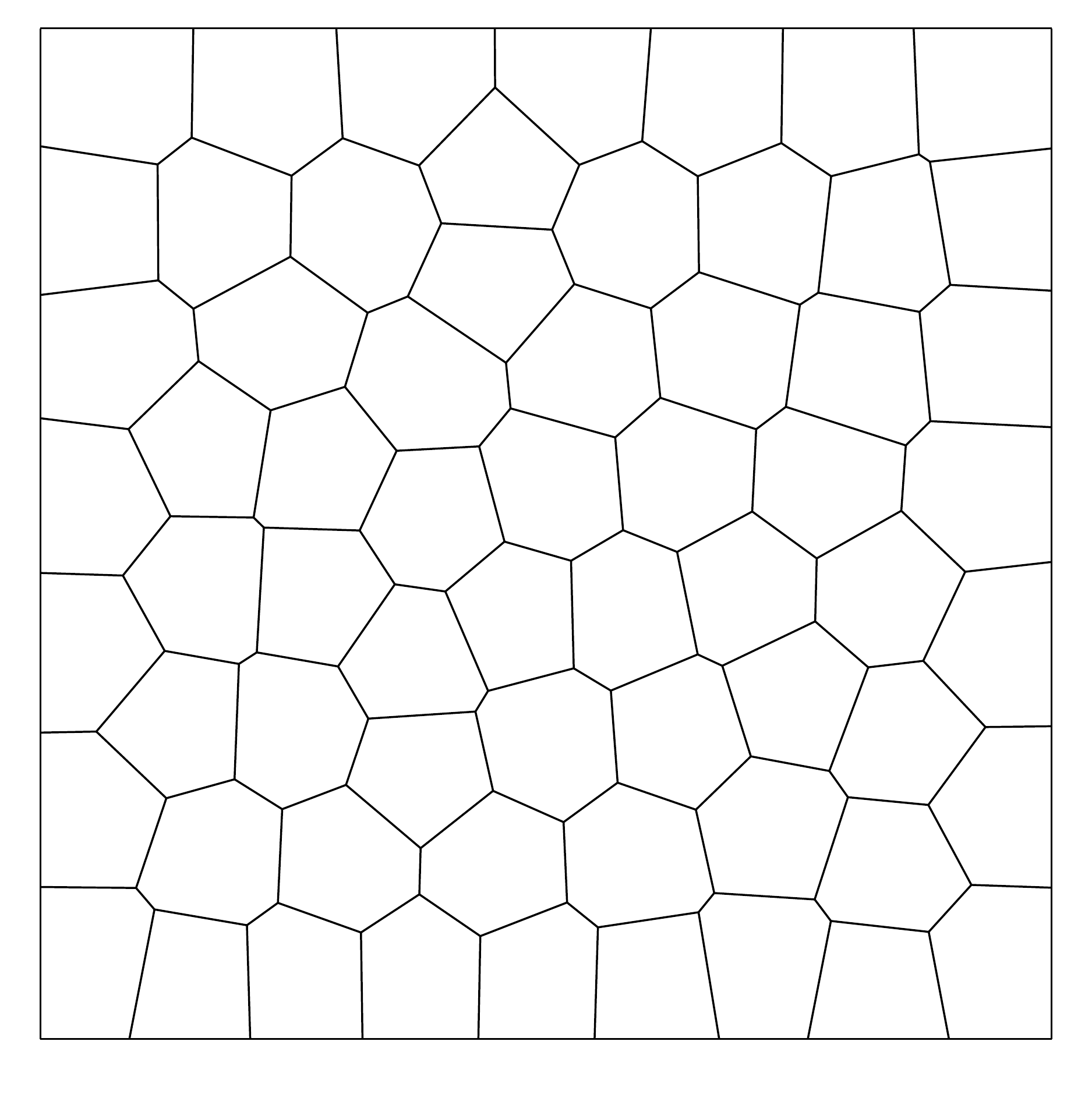}
    \end{overpic}
    \\[-0.25em]
    \textit{Mesh}~1 & \textit{Mesh}~2 & \textit{Mesh}~3 & \textit{Mesh}~4
  \end{tabular}
  \caption{Base meshes of the following mesh families from left to
    right: mainly hexagonal mesh; nonconvex octagonal mesh; randomized
    quadrilateral mesh; Voronoi mesh.}
  \label{fig:Meshes}
\end{figure}

\begin{figure}
  \centering
  \begin{tabular}{cc}
    \begin{overpic}[scale=0.35]{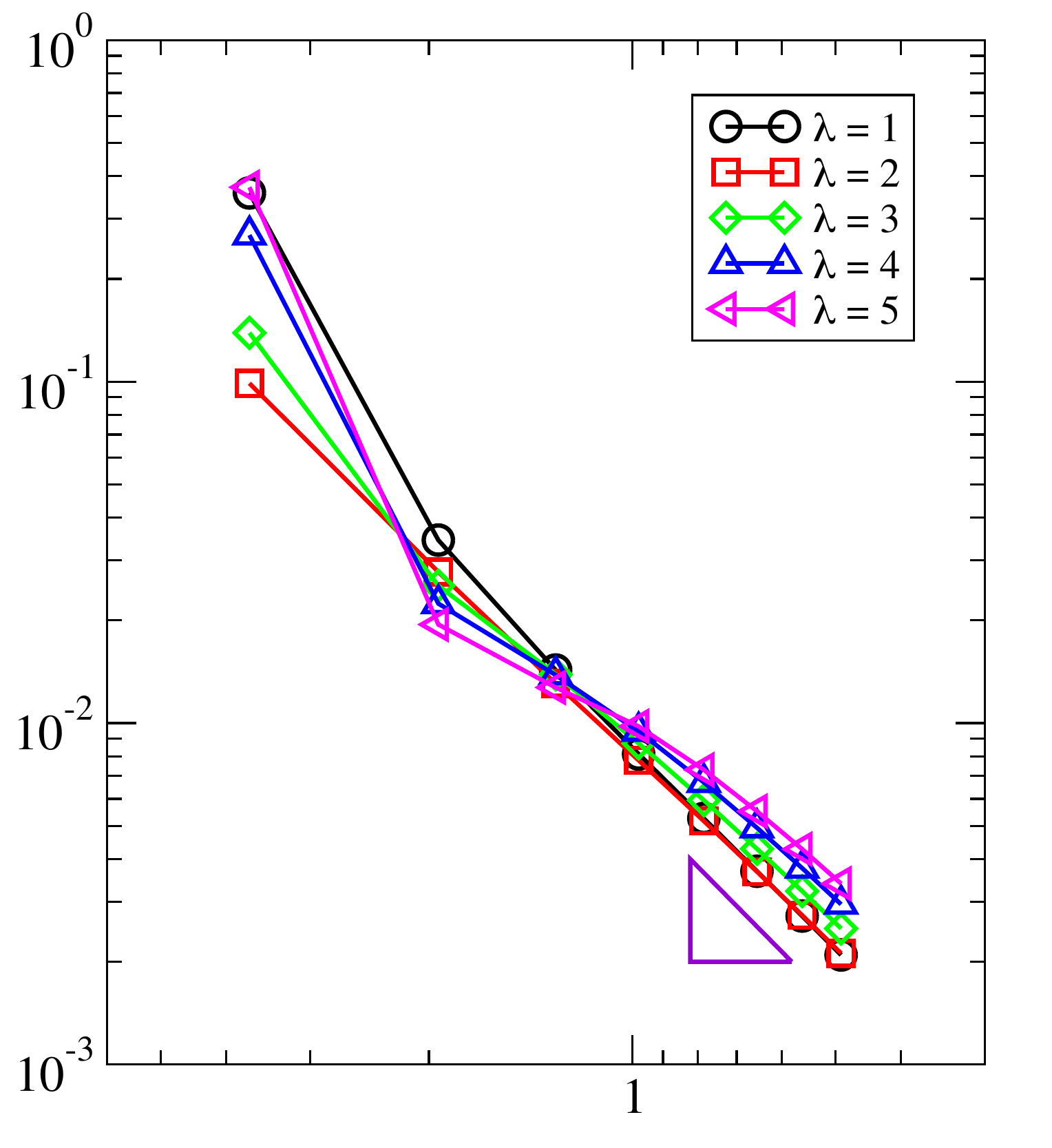}
      \put(-5,14){\begin{sideways}\textbf{Relative approximation error}\end{sideways}}
      \put(32,-5) {\textbf{Mesh size $\mathbf{h}$}}
      \put(56,19){\textbf{2}}
    \end{overpic}
    &\qquad
    \begin{overpic}[scale=0.35]{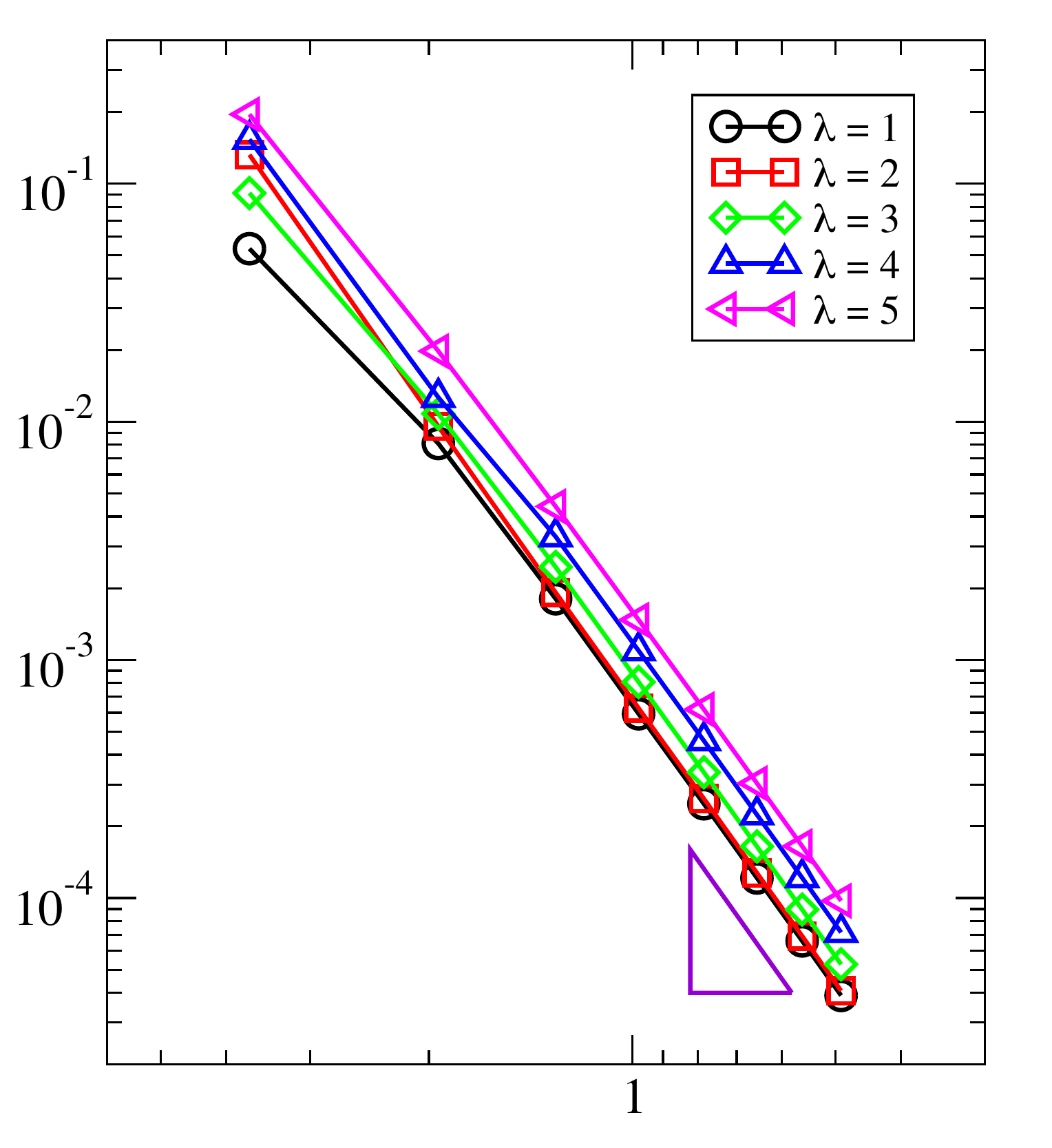}
      \put(-5,15){\begin{sideways}\textbf{Relative approximation error}\end{sideways}}
      \put(32,-5) {\textbf{Mesh size $\mathbf{h}$}}
      \put(56,18){\textbf{4}}
    \end{overpic}
    \\[0.5em]\textbf{$(k=1)$} & \textbf{$(k=2)$}\\[1em]
    \begin{overpic}[scale=0.35]{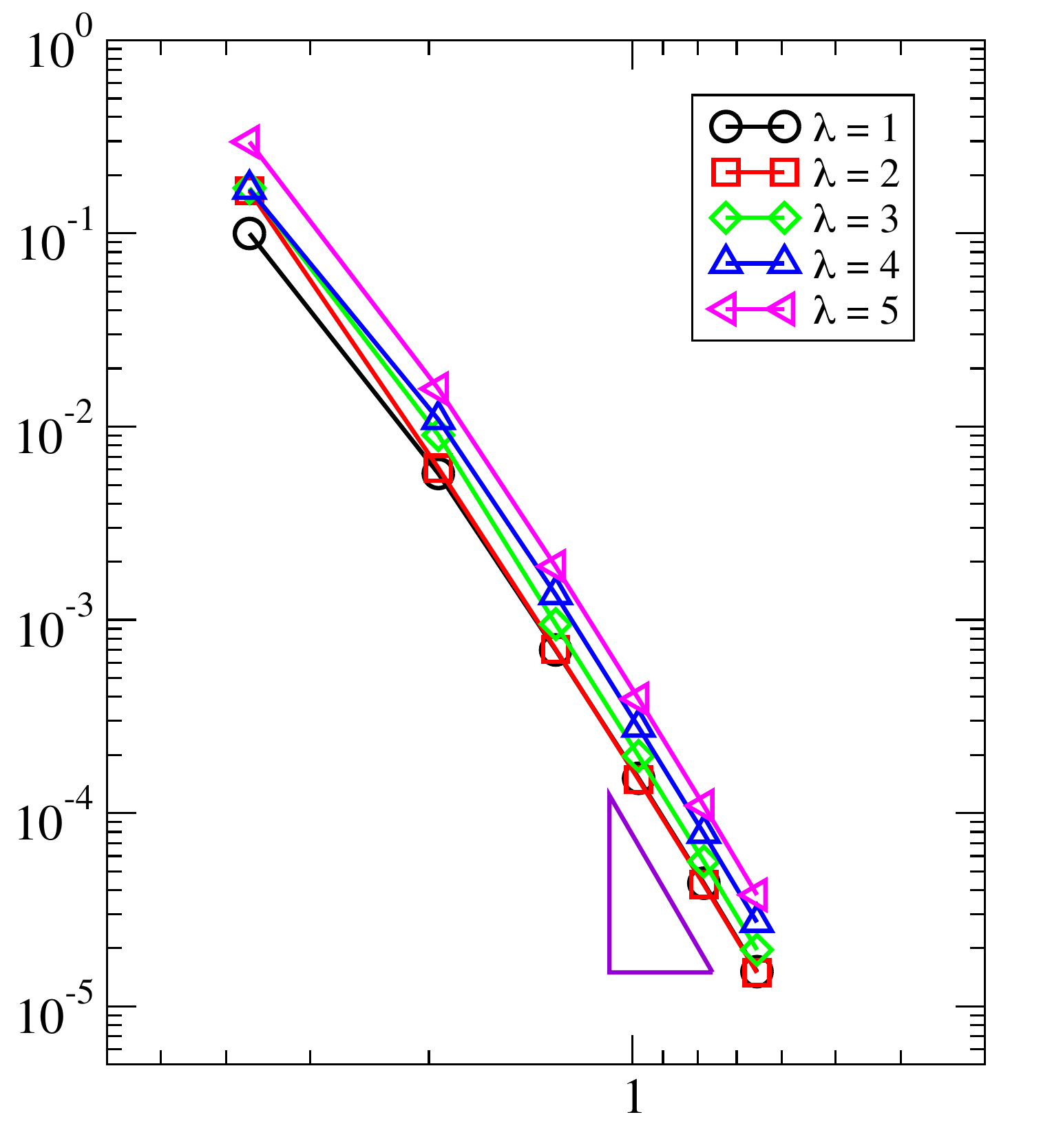}
      \put(-5,15){\begin{sideways}\textbf{Relative approximation error}\end{sideways}}
      \put(32,-5) {\textbf{Mesh size $\mathbf{h}$}}
      \put(49,21){\textbf{6}}
    \end{overpic}
    &\qquad
    \begin{overpic}[scale=0.35]{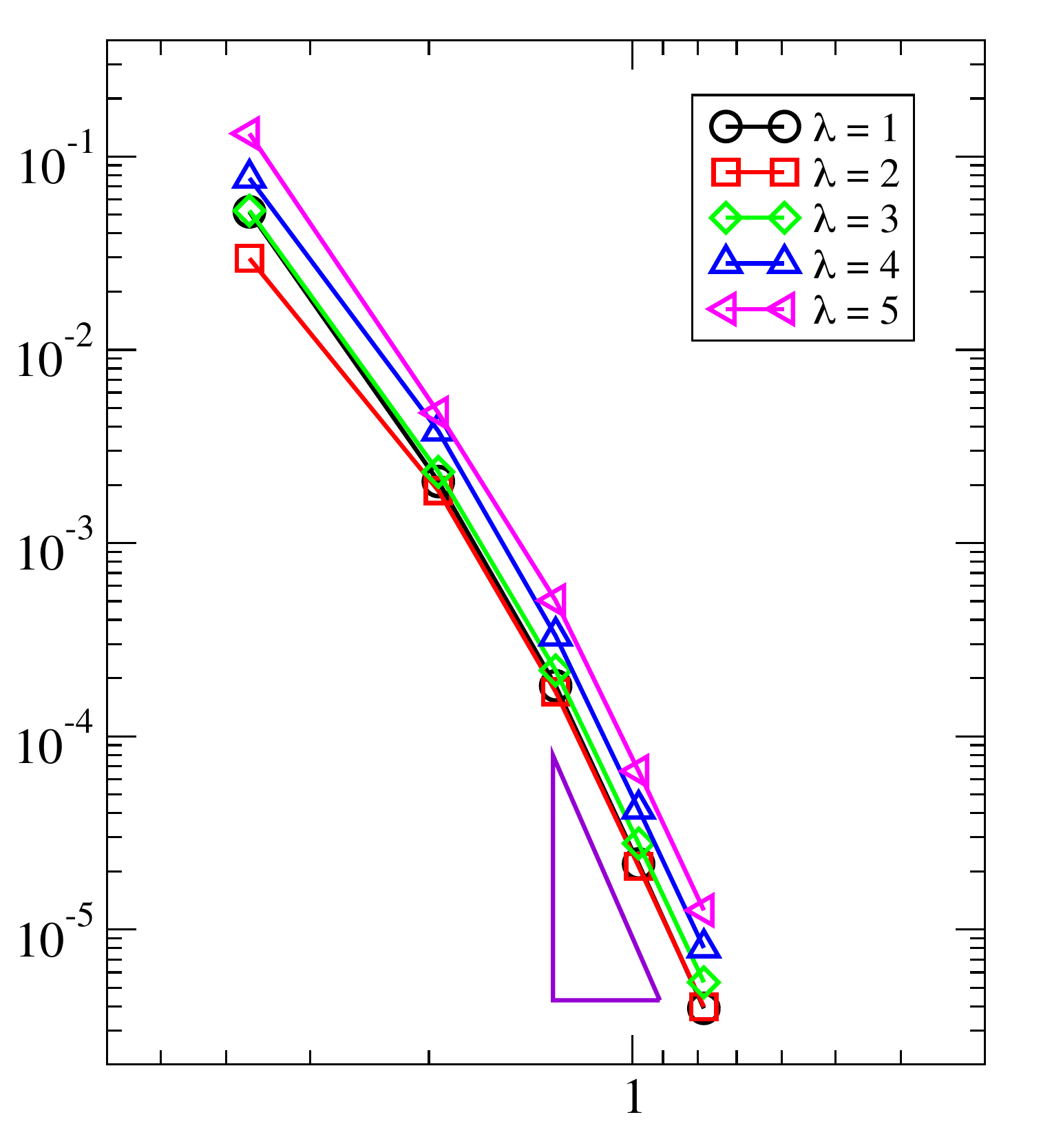}
      \put(-5,15){\begin{sideways}\textbf{Relative approximation error}\end{sideways}}
      \put(32,-5) {\textbf{Mesh size $\mathbf{h}$}}
      \put(44,22){\textbf{8}}
    \end{overpic}
    \\[0.5em]\textbf{$(k=3)$} & \textbf{$(k=4)$}
  \end{tabular}
  \caption{Test Case~1: Convergence plots for the approximation of the
    first five distinct eigenvalues $\lambda=1,2,3,4,5$ using the mainly
    hexagonal mesh and the virtual spaces $\VS{\hh}_{k}$, with $k=1$
    (top-leftmost panel); $k=2$ (top-rightmost panel); $k=3$
    (bottom-leftmost panel); $k=4$ (bottom-rightmost panel).
    The generalized eigenvalue problem uses the nonstabilized bilinear
    form $\bsh(\cdot,\cdot)$. }
  \label{fig:hexa:rates}
\end{figure}

\begin{figure}
  \centering
  \begin{tabular}{cc}
    \begin{overpic}[scale=0.35]{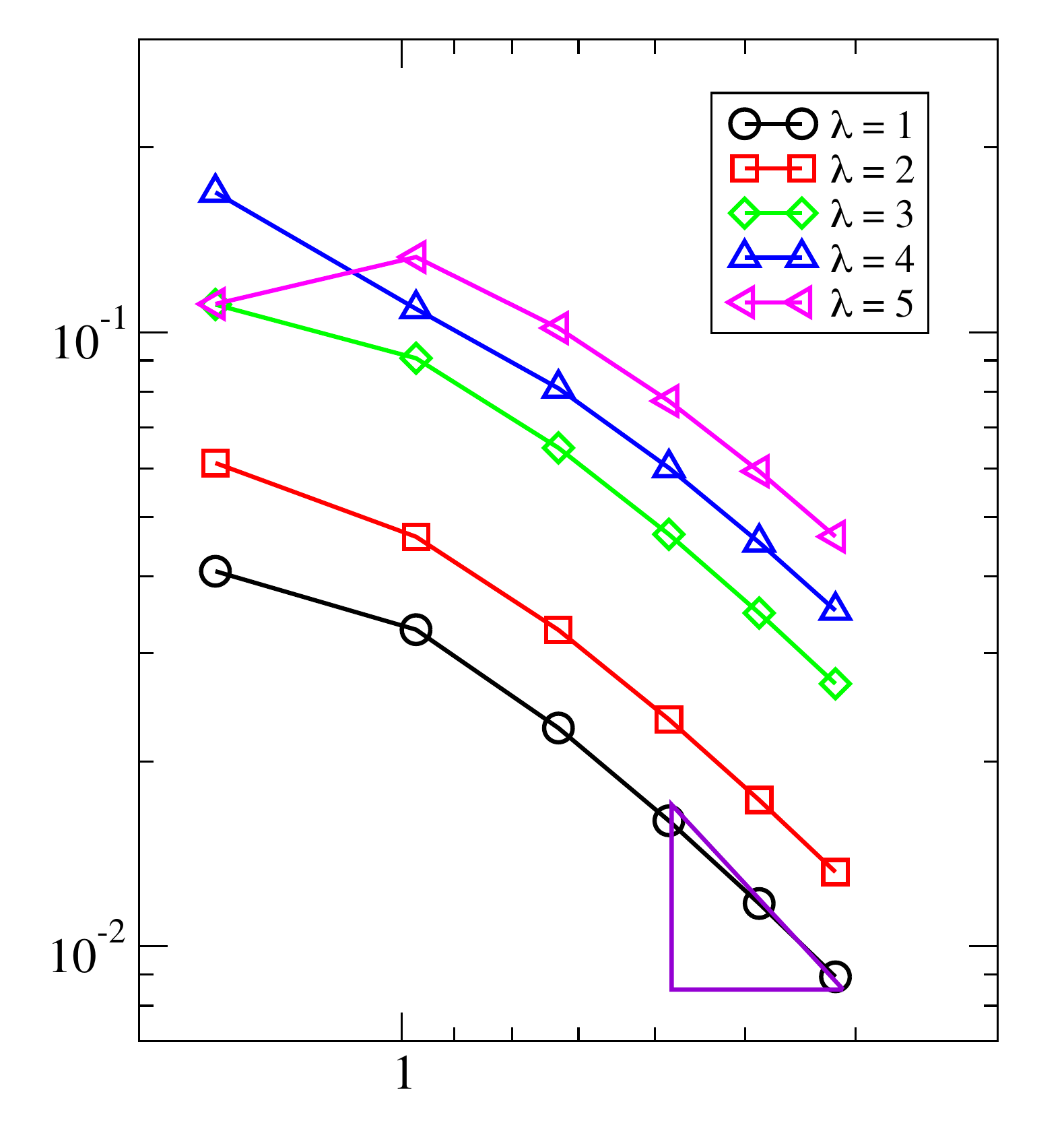}
      \put(-5,14){\begin{sideways}\textbf{Relative approximation error}\end{sideways}}
      \put(32,-5) {\textbf{Mesh size $\mathbf{h}$}}
      \put(56,18){\textbf{2}}
    \end{overpic} 
    &\qquad
    \begin{overpic}[scale=0.35]{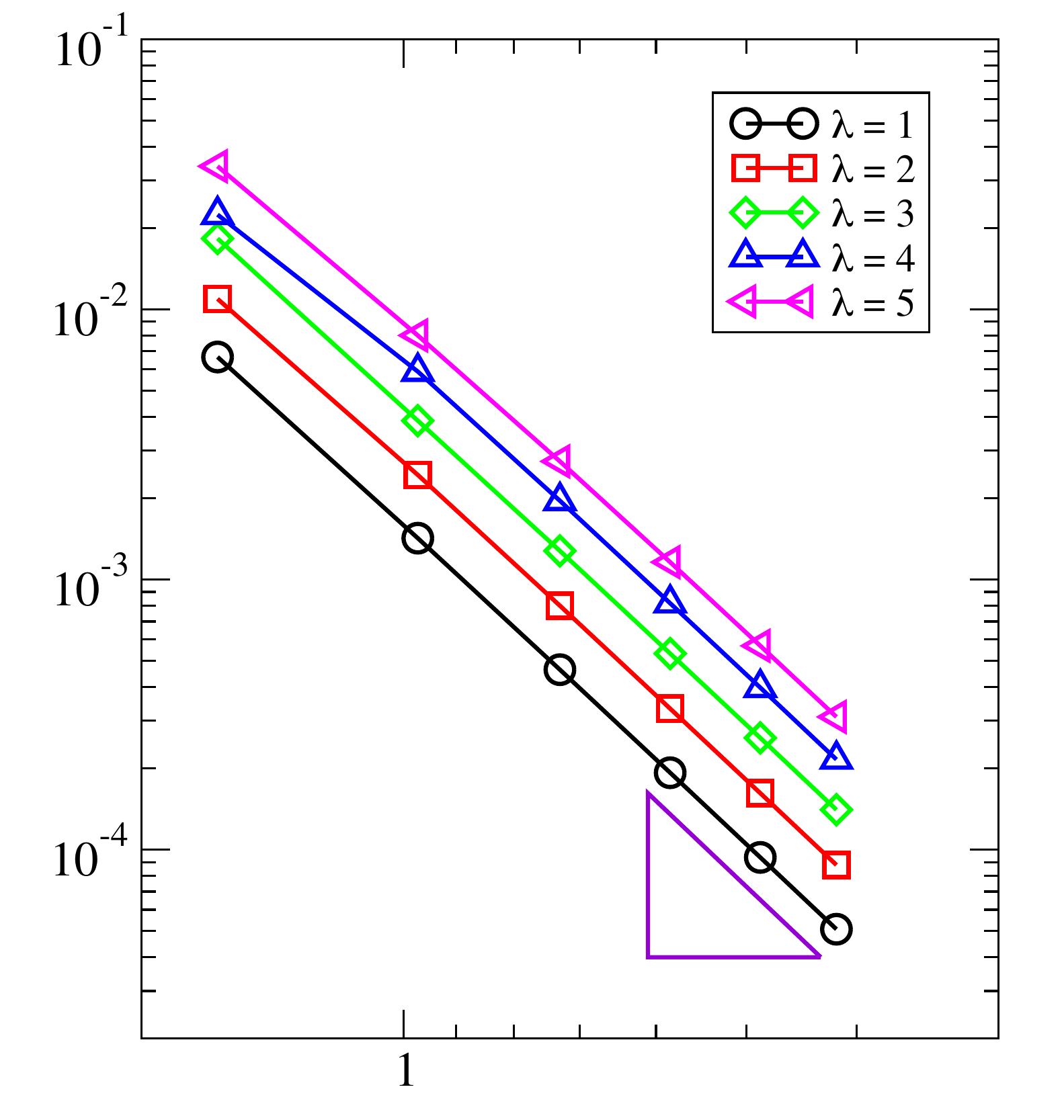}
      \put(-5,15){\begin{sideways}\textbf{Relative approximation error}\end{sideways}}
      \put(32,-5) {\textbf{Mesh size $\mathbf{h}$}}
      \put(53,20){\textbf{4}}
    \end{overpic}
    \\[0.5em]\textbf{$(k=1)$} & \textbf{$(k=2)$}\\[1em]
    \begin{overpic}[scale=0.35]{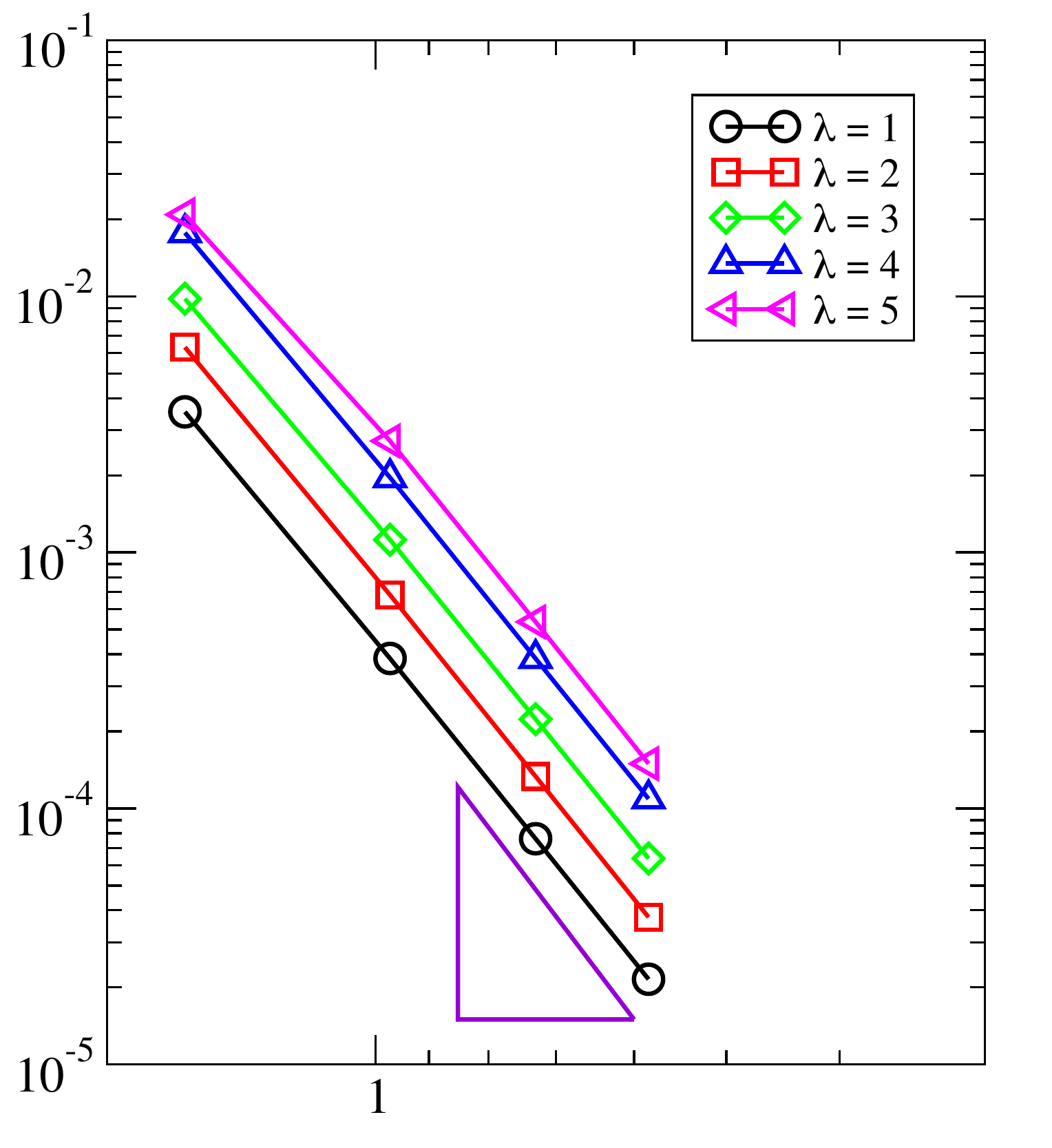}
      \put(-5,15){\begin{sideways}\textbf{Relative approximation error}\end{sideways}}
      \put(32,-5) {\textbf{Mesh size $\mathbf{h}$}}
      \put(35.5,20){\textbf{6}}
    \end{overpic} 
    &\qquad
    \begin{overpic}[scale=0.35]{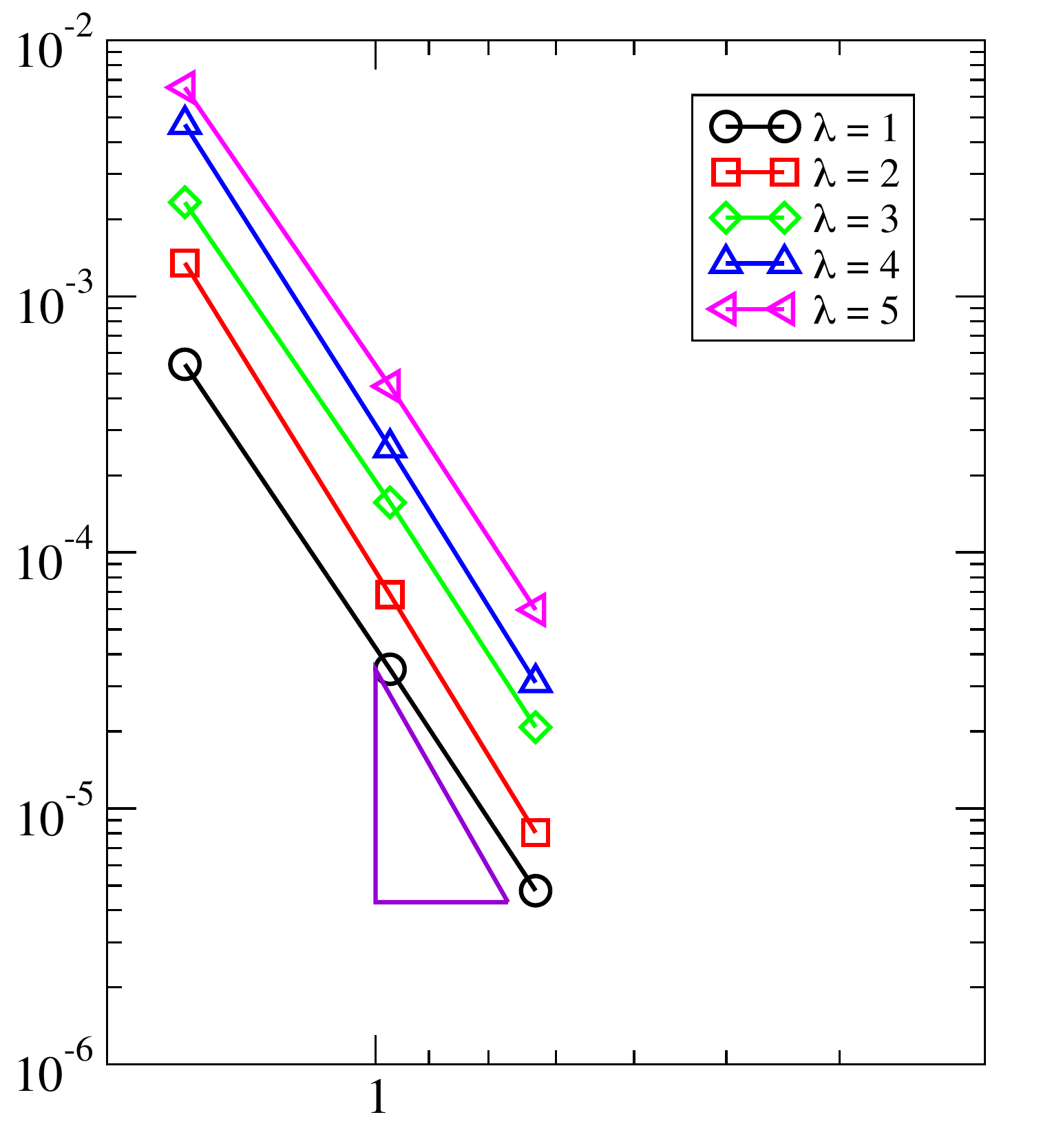}
      \put(-5,15){\begin{sideways}\textbf{Relative approximation error}\end{sideways}}
      \put(32,-5) {\textbf{Mesh size $\mathbf{h}$}}
      \put(28.5,31){\textbf{8}}
    \end{overpic}
    \\[0.5em]\textbf{$(k=3)$} & \textbf{$(k=4)$}
  \end{tabular}
  \caption{Test Case~1:  Convergence plots for the approximation of the
    first five distinct eigenvalues $\lambda=1,2,3,4,5$ using the 
    nonconvex octagon mesh
    and the virtual spaces $\VS{\hh}_{k}$, with $k=1$
    (top-leftmost panel); $k=2$ (top-rightmost panel); $k=3$
    (bottom-leftmost panel); $k=4$ (bottom-rightmost panel).
    The generalized eigenvalue problem uses the nonstabilized bilinear
    form $\bsh(\cdot,\cdot)$.  }
  \label{fig:octa:rates}
\end{figure}

In this test case, we numerically solve the 2D Quantum Harmonic
Oscillator problem that corresponds to the Schrodinger equation with
the harmonic potential $\Vs(\xs,\ys)=(1/2)(\xs^2+\ys^2)$.
The eigenvalues are a suitable combinations of the eigenvalues of the
one dimensional problem  and are given by the natural
numbers $n=1,2,3,\ldots$, each one with multiplicity $n$.
The eigenfunctions of such a problem are obtained through the
two-dimensional tensor product of one-dimensional Hermite functions,
which are given by the Hermite polynomials multiplied by the Gaussian
function $\ws(\xs,\ys)=\exp\big(-(\xs^2+\ys^2)\big)$.
As these eigenfunctions are rapidly decreasing to zero for $\xs$,
$\ys$ tending to infinity due to the Gaussian term, we can assume
homogeneous Dirichlet boundary conditions if the computational domain
is sufficiently large.
For such reason, we solve the eigenvalue problem on the square domain
$\Omega=]-10,10[\times]-10,10[$.
On this domain, we consider four different mesh sequences, hereafter
denoted by:
\begin{itemize}
\item \textit{Mesh~1}, mainly hexagonal mesh with continuously
  distorted cells;
\item \textit{Mesh~2}, nonconvex octagonal mesh;
\item \textit{Mesh~3}, randomized quadrilateral mesh;
\item \textit{Mesh~4}, central Voronoi tessellation.
\end{itemize}
The first mesh of each sequence is shown in Figure~\ref{fig:Meshes}.
These mesh sequences have been widely used in the mimetic finite
difference and virtual element literature, and a detailed description
of their construction can easily be found elsewhere, for example,
see~\cite{BeiraodaVeiga-Lipnikov-Manzini:2011}.

The convergence curves for the four mesh sequences above are reported
in Figures~\ref{fig:hexa:rates}, \ref{fig:octa:rates},
\ref{fig:quads:rates}, and \ref{fig:voro:rates}.
\begin{figure}
  \centering
  \begin{tabular}{cc}
    \begin{overpic}[scale=0.35]{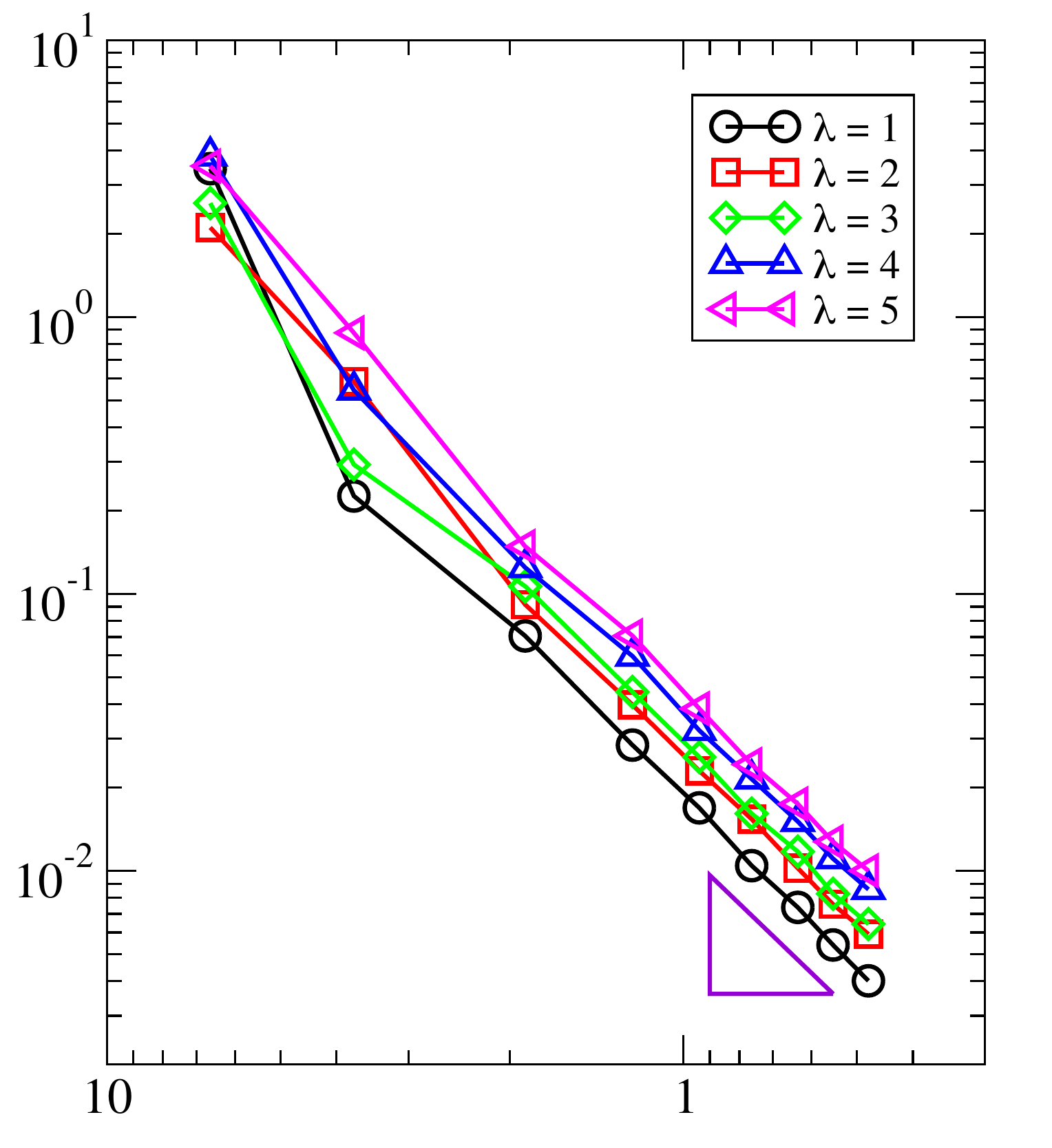}
      \put(-5,14){\begin{sideways}\textbf{Relative approximation error}\end{sideways}}
      \put(32,-5) {\textbf{Mesh size $\mathbf{h}$}}
      \put(57,16){\textbf{2}}
    \end{overpic} 
    &\qquad
    \begin{overpic}[scale=0.35]{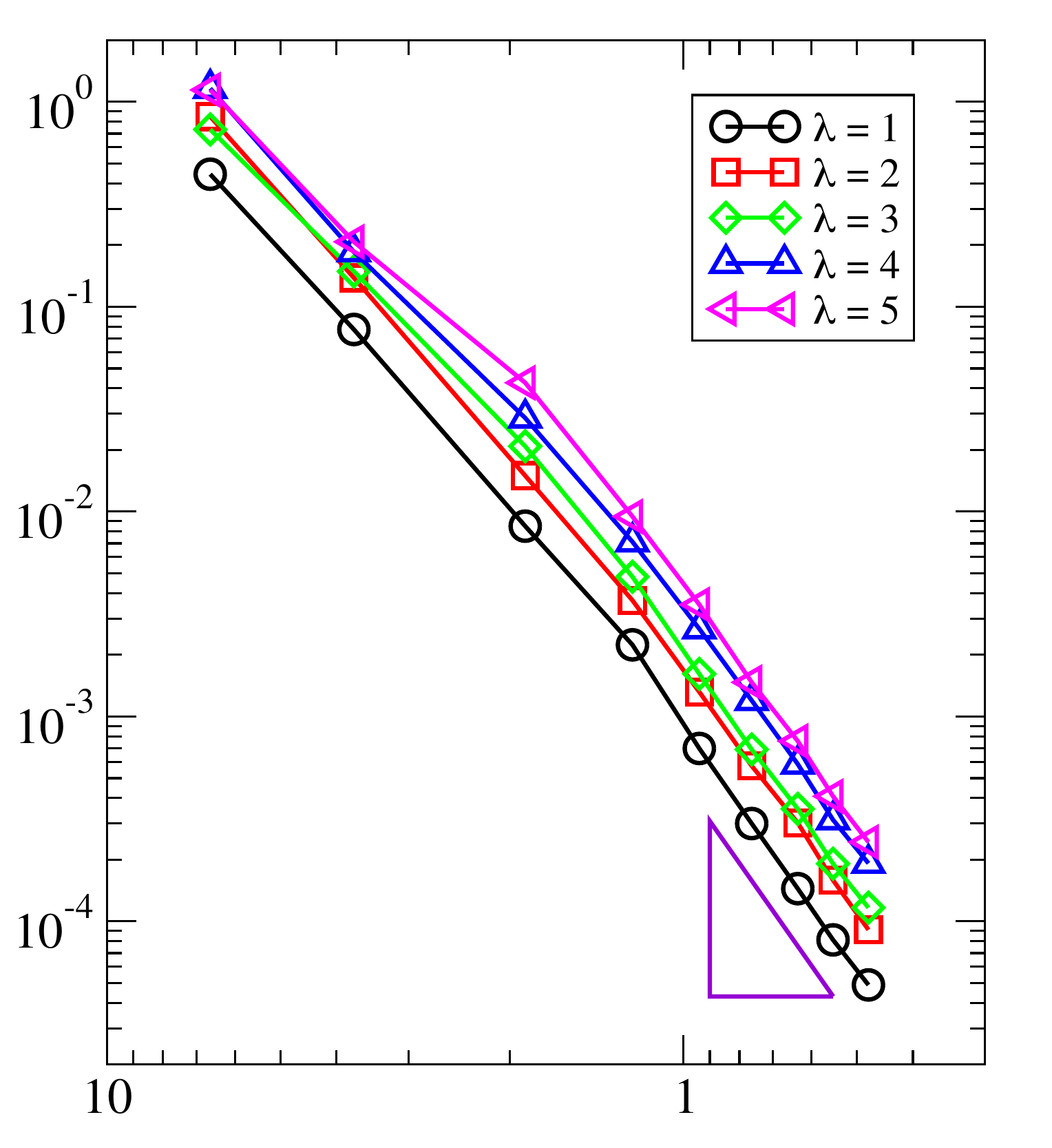}
      \put(-5,15){\begin{sideways}\textbf{Relative approximation error}\end{sideways}}
      \put(32,-5) {\textbf{Mesh size $\mathbf{h}$}}
      \put(57,19){\textbf{4}}
    \end{overpic}
    \\[0.5em]\textbf{$(k=1)$} & \textbf{$(k=2)$} \\[1em]
    \begin{overpic}[scale=0.35]{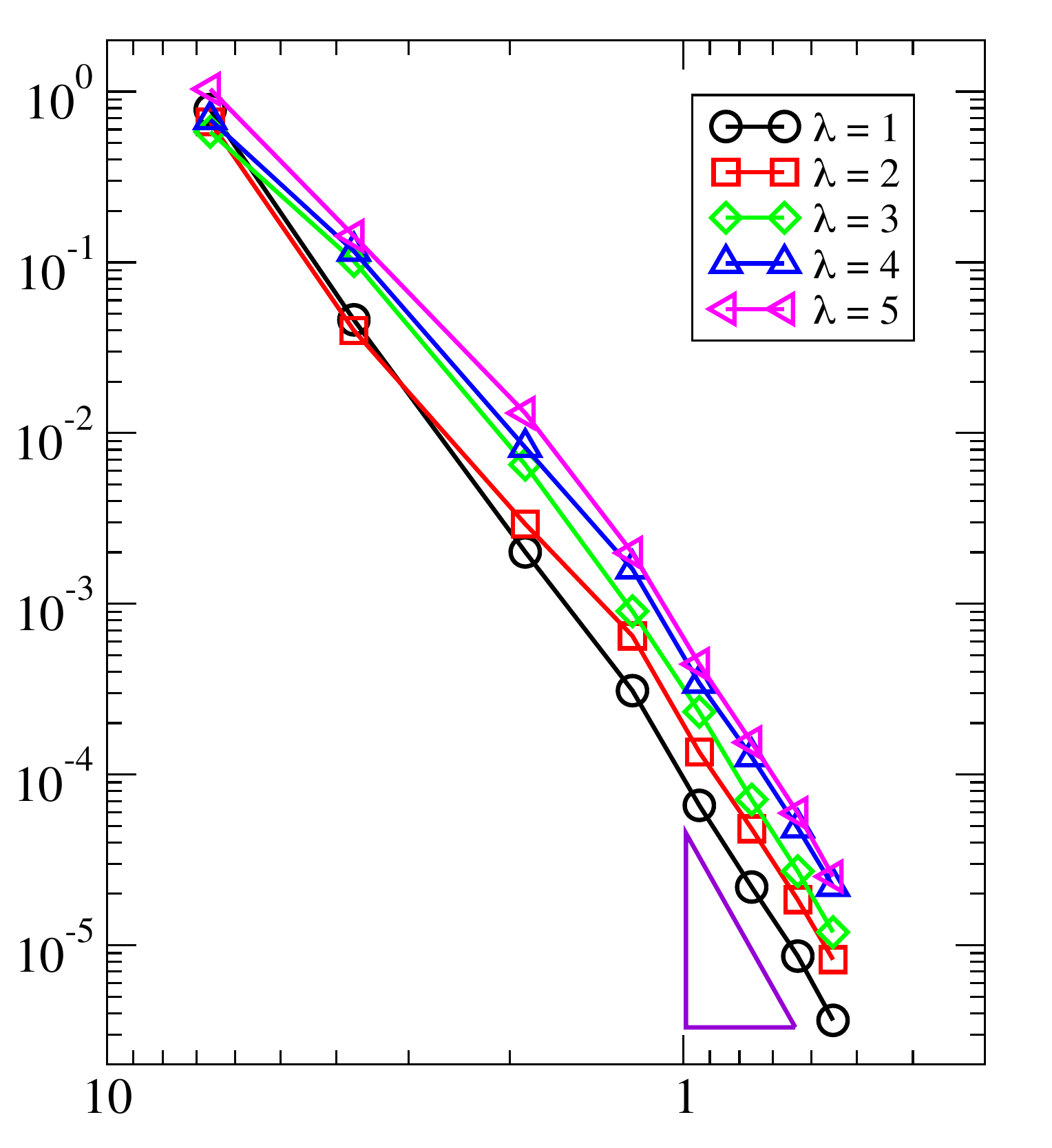}
      \put(-5,15){\begin{sideways}\textbf{Relative approximation error}\end{sideways}}
      \put(32,-5) {\textbf{Mesh size $\mathbf{h}$}}
      \put(55,18){\textbf{6}}
    \end{overpic} 
    &\qquad
    \begin{overpic}[scale=0.35]{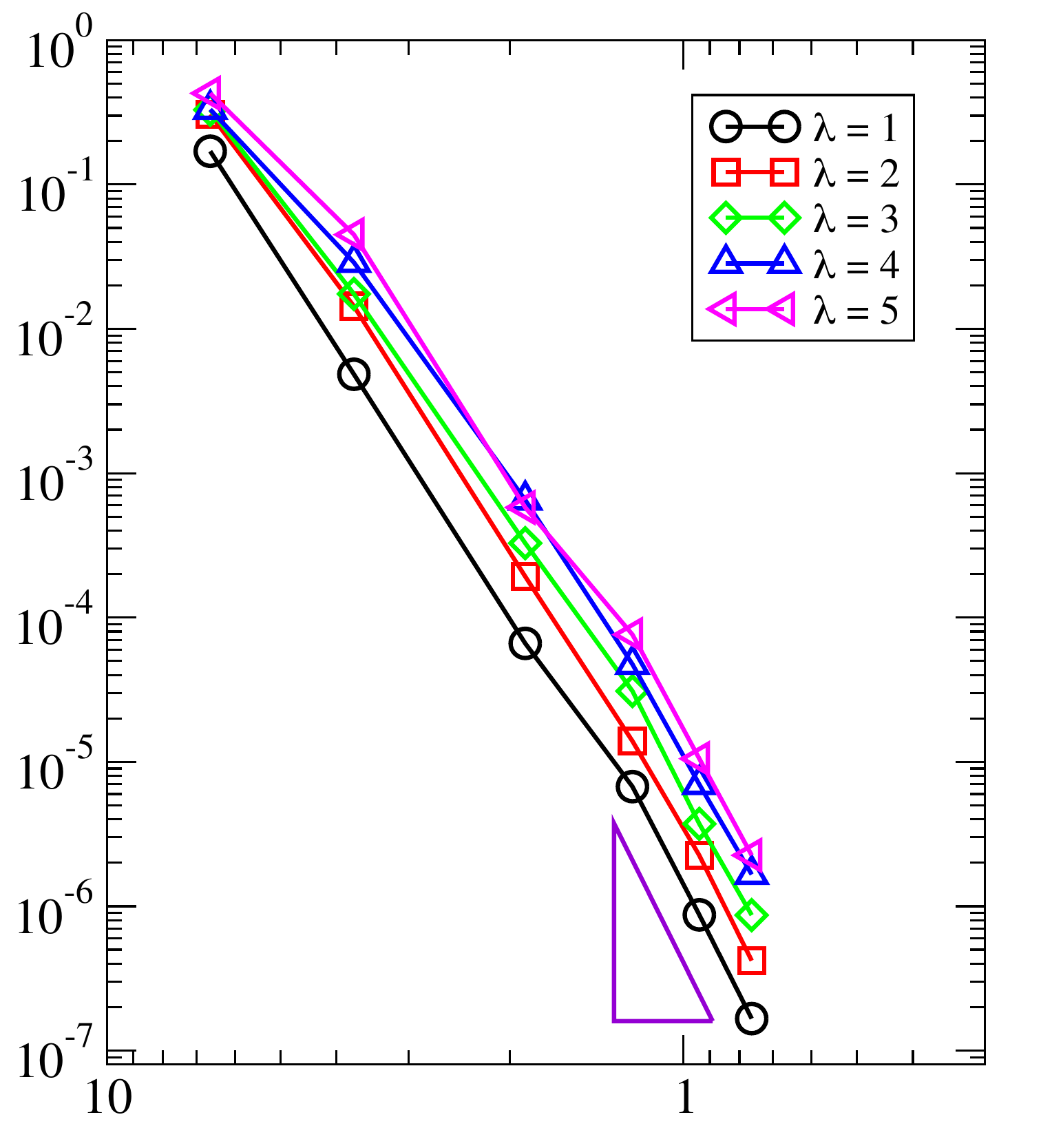}
      \put(-5,15){\begin{sideways}\textbf{Relative approximation error}\end{sideways}}
      \put(32,-5) {\textbf{Mesh size $\mathbf{h}$}}
      \put(49,18){\textbf{8}}
    \end{overpic}
    \\[0.5em]\textbf{$(k=3)$} & \textbf{$(k=4)$}
  \end{tabular}
  \caption{Test Case~1: Convergence plots for the approximation of the
    first five distinct eigenvalues $\lambda=1,2,3,4,5$ using the 
    randomized quadrilateral mesh and the virtual spaces $\VS{\hh}_{k}$, with $k=1$
    (top-leftmost panel); $k=2$ (top-rightmost panel); $k=3$
    (bottom-leftmost panel); $k=4$ (bottom-rightmost panel).
    The generalized eigenvalue problem uses the nonstabilized bilinear
    form $\bsh(\cdot,\cdot)$. }
  \label{fig:quads:rates}
\end{figure}

\begin{figure}
  \centering
  \begin{tabular}{cc}
    \begin{overpic}[scale=0.35]{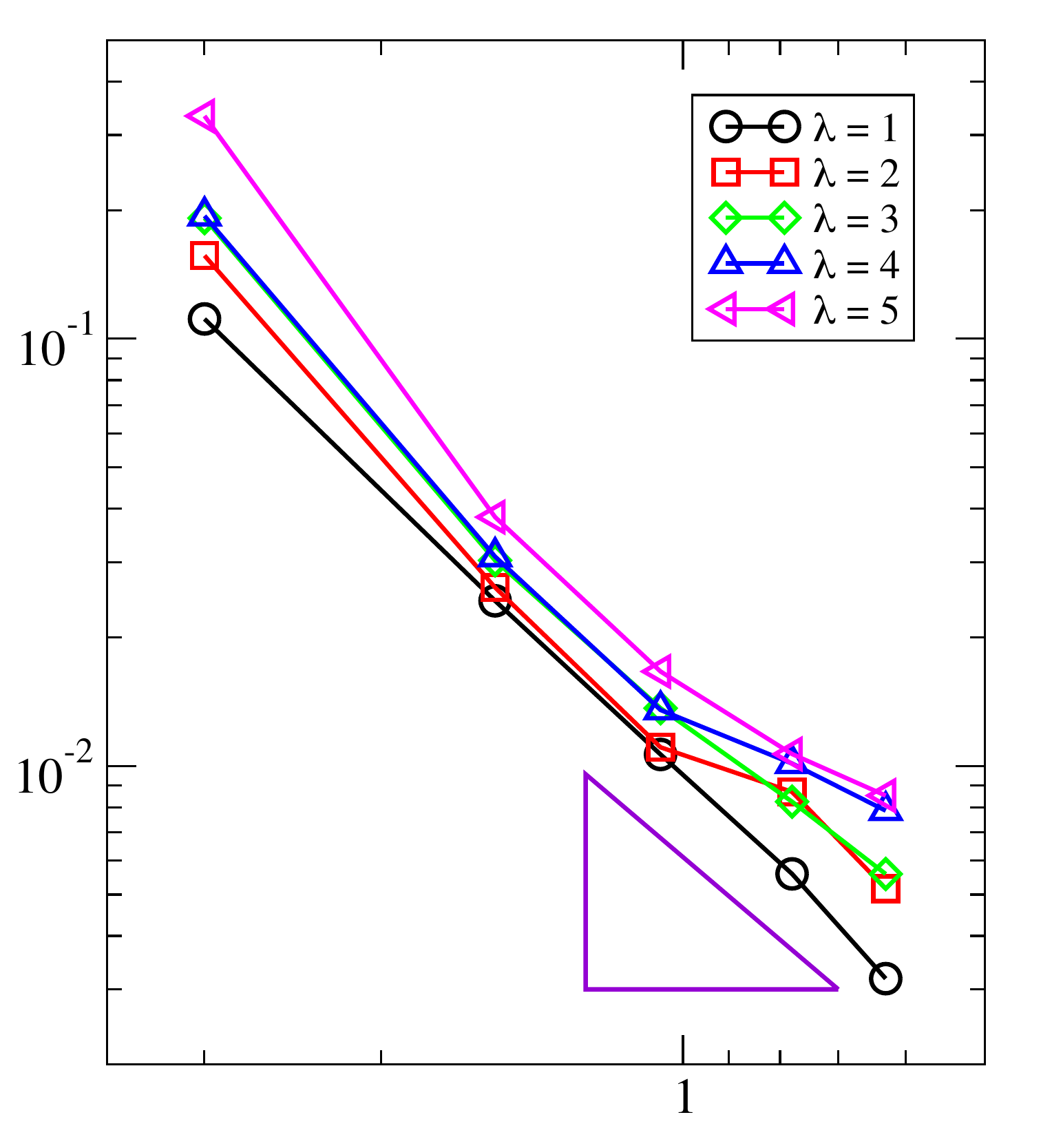}
      \put(-5,14){\begin{sideways}\textbf{Relative approximation error}\end{sideways}}
      \put(32,-5) {\textbf{Mesh size $\mathbf{h}$}}
      \put(47,21){\textbf{2}}
    \end{overpic} 
    &\qquad
    \begin{overpic}[scale=0.35]{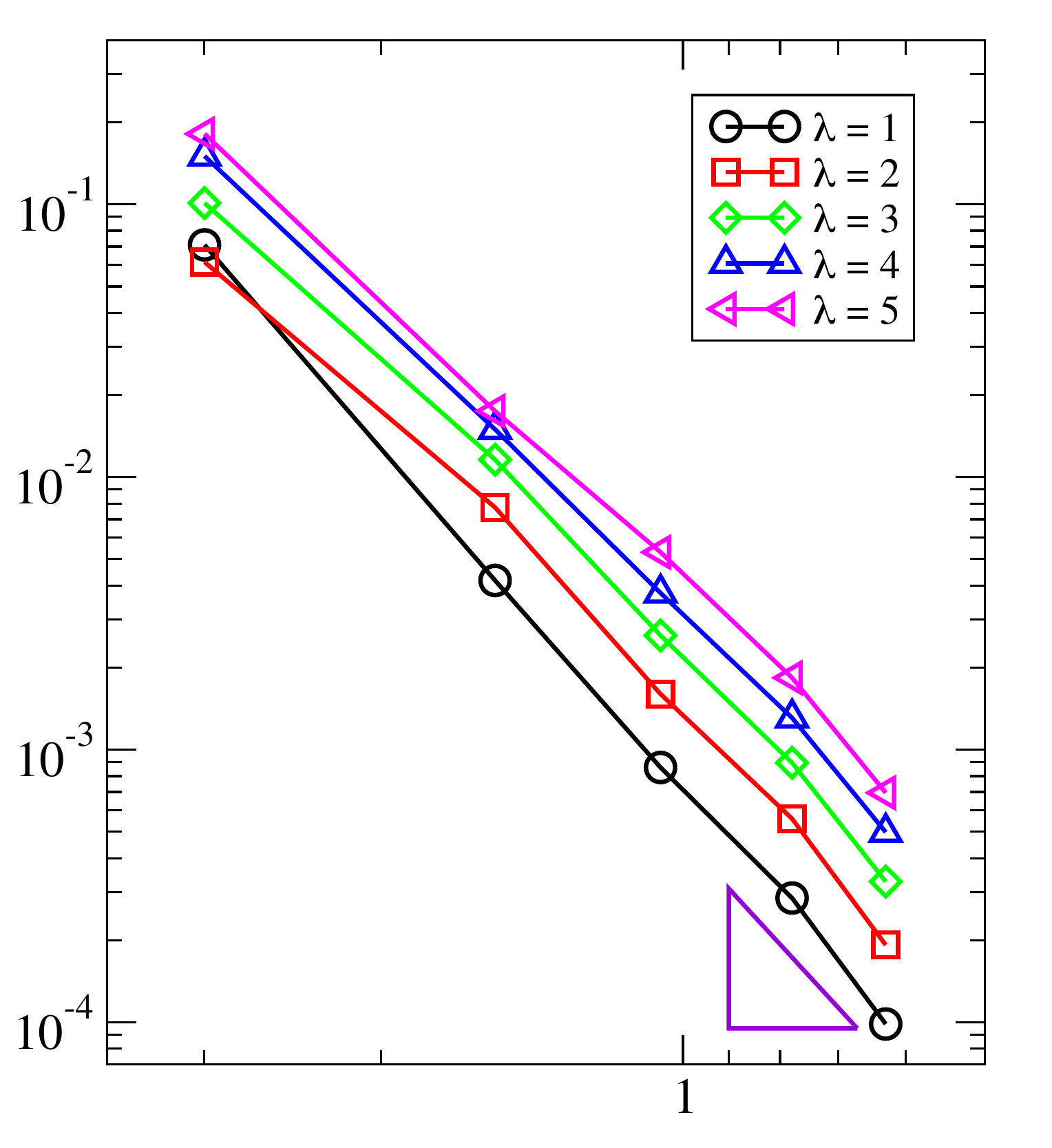}
      \put(-5,15){\begin{sideways}\textbf{Relative approximation error}\end{sideways}}
      \put(32,-5) {\textbf{Mesh size $\mathbf{h}$}}
      \put(59,15){\textbf{4}}
    \end{overpic}
    \\[0.5em] \textbf{$(k=1)$} & \textbf{$(k=2)$}\\[1.em]
    \begin{overpic}[scale=0.35]{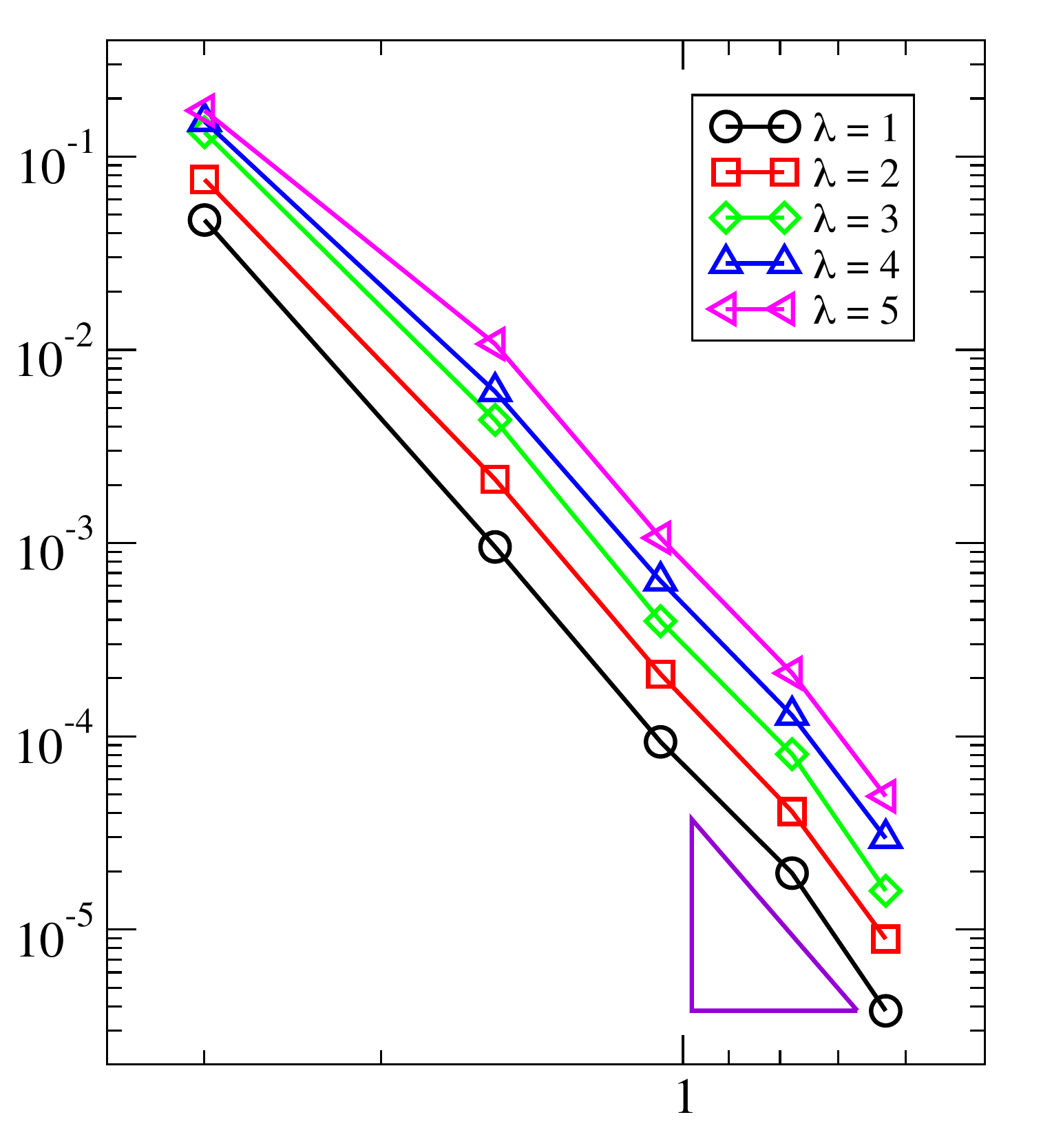}
      \put(-5,15){\begin{sideways}\textbf{Relative approximation error}\end{sideways}}
      \put(32,-5) {\textbf{Mesh size $\mathbf{h}$}}
      \put(56,18){\textbf{6}}
    \end{overpic}
    &\qquad
    \begin{overpic}[scale=0.35]{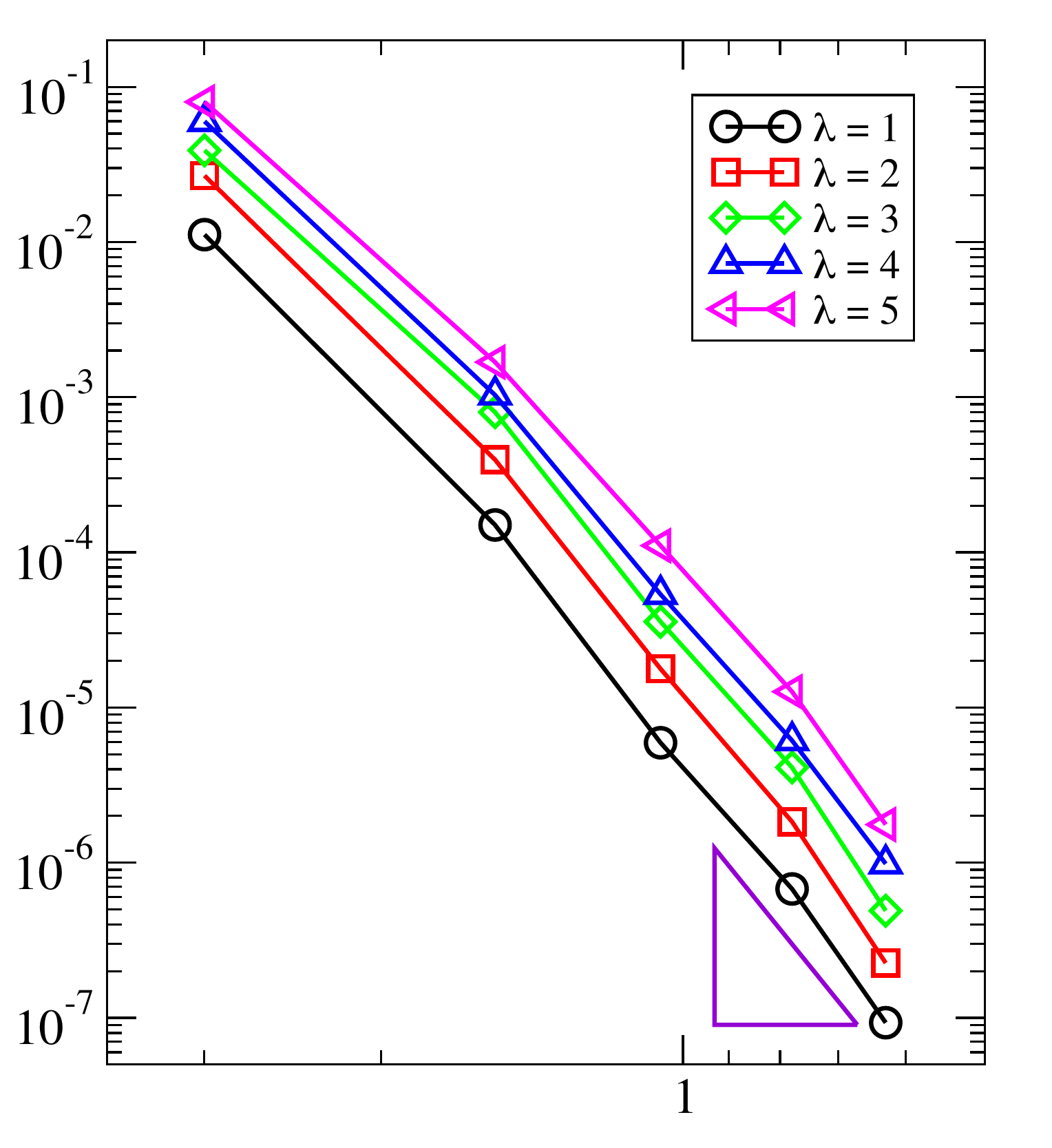}
      \put(-5,15){\begin{sideways}\textbf{Relative approximation error}\end{sideways}}
      \put(32,-5) {\textbf{Mesh size $\mathbf{h}$}}
      \put(58,17){\textbf{8}}
    \end{overpic}
    \\[0.5em] \textbf{$(k=3)$} & \textbf{$(k=4)$}
  \end{tabular}
  \caption{Test Case~1: Convergence plots for the approximation of the
    first five distinct eigenvalues $\lambda=1,2,3,4,5$ using the Voronoi mesh
    and the virtual spaces $\VS{\hh}_{k}$, with $k=1$
    (top-leftmost panel); $k=2$ (top-rightmost panel); $k=3$
    (bottom-leftmost panel); $k=4$ (bottom-rightmost panel).
    The generalized eigenvalue problem uses the nonstabilized bilinear
    form $\bsh(\cdot,\cdot)$. }
  \label{fig:voro:rates}
\end{figure}

The expected rate of convergence is shown in each panel by the
triangle closed to the error curve and indicated by an explicit label.
For these calculations, we used the VEM approximation based on the
conforming virtual element space $\Vhk{}$, $k=1,2,3,4$, and the VEM
formulation~\eqref{eq:discreteEigPbm} using the nonstabilized bilinear
form $\bsh(\cdot,\cdot)$.
As already observed in~\cite{Gardini-Vacca:2017} for the conforming
VEM approximation of the Laplace eigenvalue problem, the same
computations using formulation~\eqref{eq:discreteEigPbm2} and the
stabilized bilinear $\bsht(\cdot,\cdot)$ produce almost identical
results, which, for this reason, are not shown here.
These plots confirm that the conforming VEM formulations proposed in
this work provide a numerical approximation with optimal convergence
rate on a set of representative mesh sequences, including deformed and
nonconvex cells, of the Schrodinger equation problem, i.e., the
standard eigenvalue problem with a regular potential term in the
Hamilton operator at left hand-side.

\subsection{Test~2 (piecewise constant diffusivity tensor)}

The present  test problem is taken from the benchmark singular solution set in \cite{dauge}. 
We here consider the square domain $\Omega = {\color{black}(-1, 1)}^2$ split into two subdomains $\Omega_{\delta}$ and $\Omega_1$ (see the left plot in Figure \ref{split_domain}), and we study the eigenvalue problem on the square with  discontinuous diffusivity tensor and zero potential $V$ coupled with Neumann homogeneous boundary conditions i.e.
we consider the following problem in strong form:
\begin{equation*}
  -\nabla \cdot \left( \K(\xv) \nabla \us(\xv) \right)   = \eigv\us(\xv) \quad \text{in }\Omega\quad 
  \text{ and} \qquad \frac{\partial u}{\partial n} = 0 \quad \text{on $\Gamma$},
\end{equation*}
\begin{figure}
\centering
\begin{overpic}[scale= 0.35]{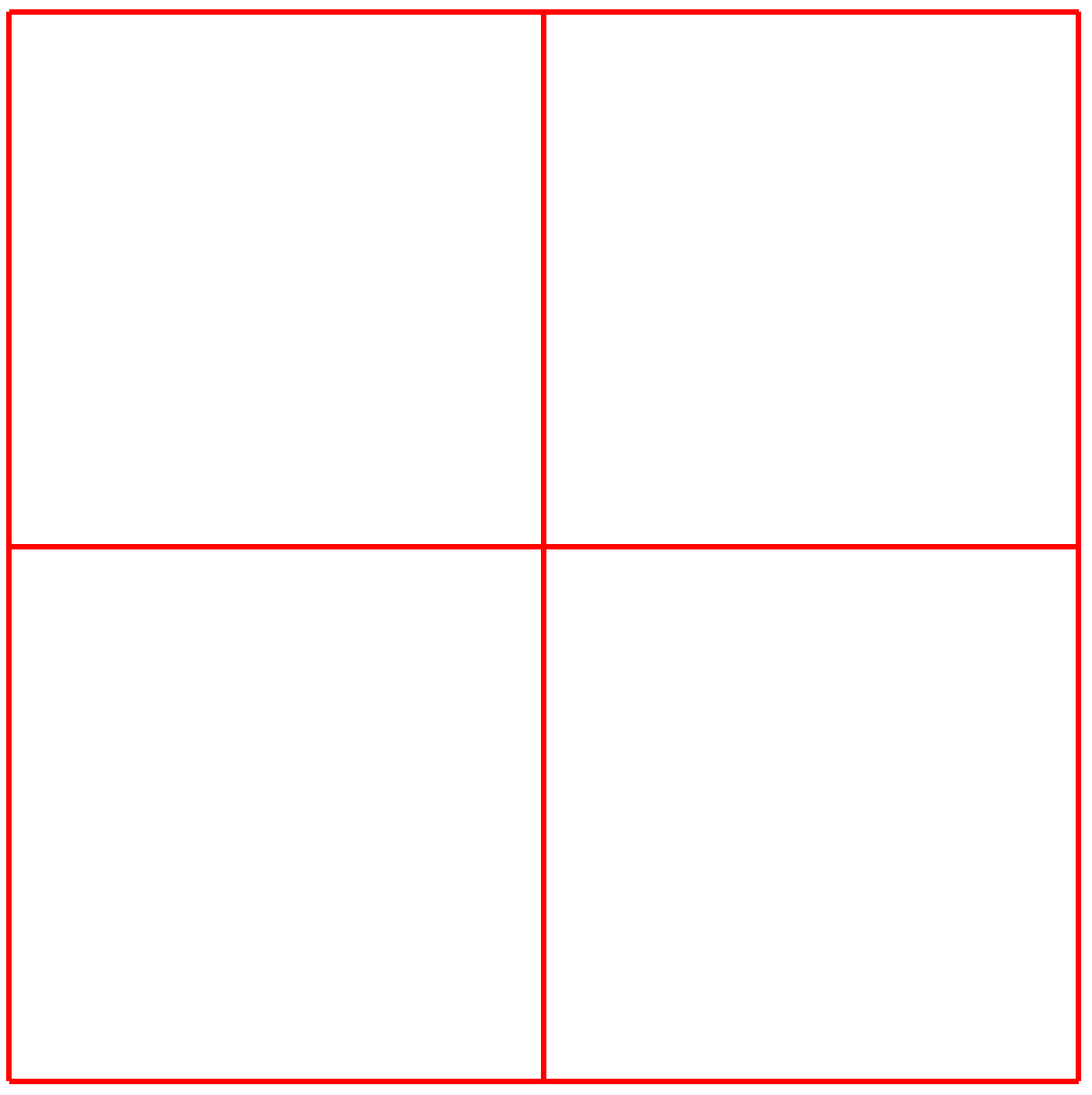}
\put (20,70) {\huge{$\Omega_1$}}
\put (70,20) {\huge{$\Omega_1$}}
\put (20,20) {\huge{$\Omega_{\delta}$}}
\put (70,70) {\huge{$\Omega_{\delta}$}}
\end{overpic}
\qquad
\qquad
\begin{overpic}[scale= 0.35]{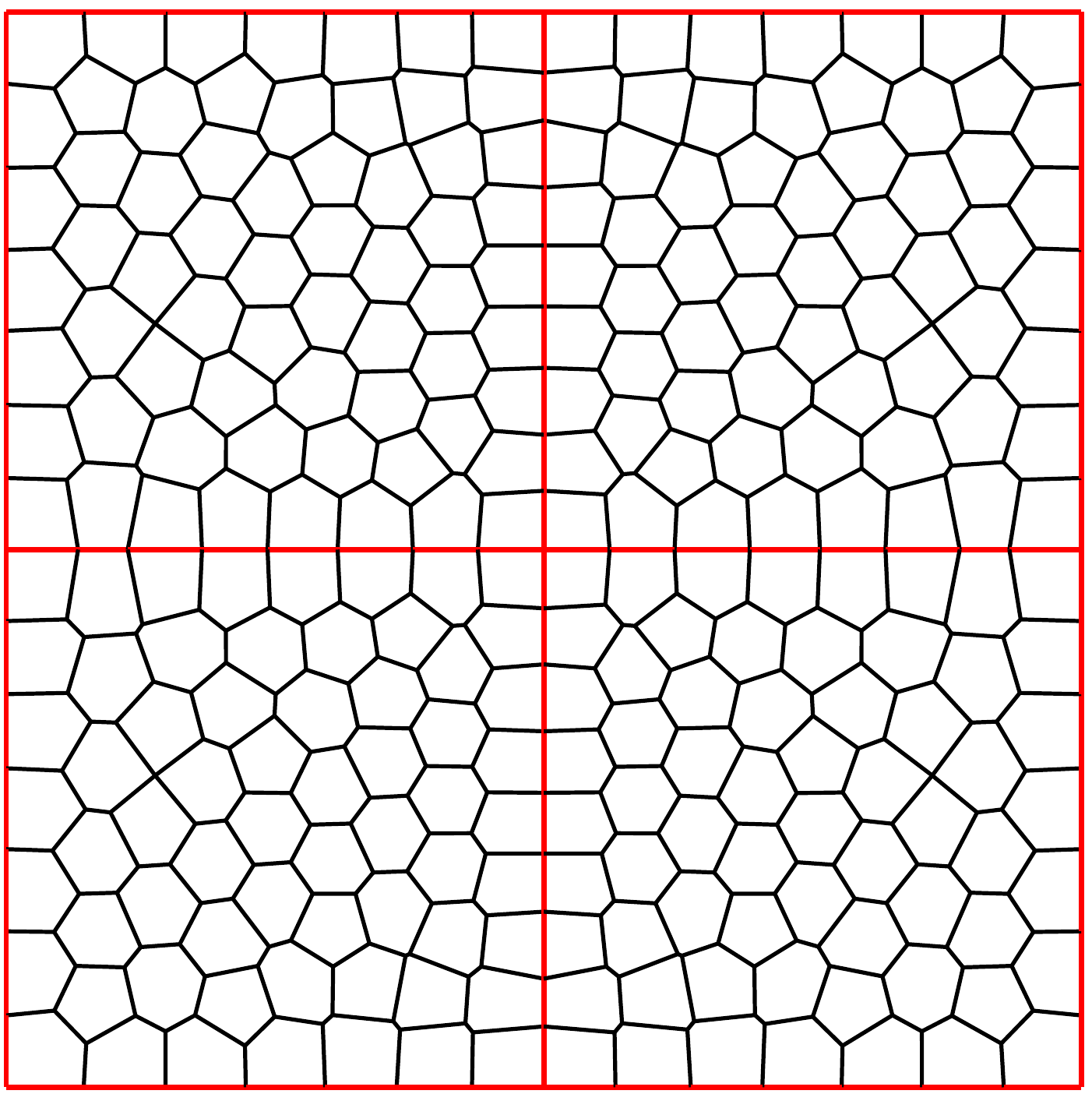}
\end{overpic}
\caption{Test {\color{black} Case 2:} Left plot: subdivision of $\Omega$ into the subdomains $\Omega_{\delta}$ and $\Omega_1$. Right plot: Example of locally Voronoi decomposition of $\Omega$.}
\label{split_domain}
\end{figure}
Therefore  the continuous bilinear form associated to the eigenvalue problem is
\[
a_{\K}^{\P}(u, v):= \int_{\P} \K \nabla u \cdot \nabla v \, {\rm d}\bf{x}
\]
whose virtual approximation (see \cite{vemgeneral, vemgeneral3d}) is given by
\begin{equation}\label{eq:ahgeneral}
a_{h, \K}^{\P}(u_h,v_h)=  \int_{\P} \K \Pi_{k-1}^{0, \P} \nabla u_h \cdot \Pi_{k-1}^{0, \P} \nabla v_h \, {\rm d}{\bf{x}}
+ \overline{\K} S^{\P}\Big((I-\Pi_k^{\nabla, \P})u_h,(I-\Pi_k^{\nabla}, \P)v_h\Big)
\end{equation}
to be used in place of $a_{h}^{\P}(u_h,v_h)$ (cf. \eqref{eq:discreteforms}) in Problem \eqref{eq:discreteEigPbm2}, where $\overline{\K} = \|\K\|_{\infty, \P}$.
We consider $\K_{|\Omega_1} = I$ and $\K_{|\Omega_{\delta}} = \delta^{-1} I$ with four different
values of $\delta$, namely $\delta = 0.50, 0.10, 0.01, 1e-8$.

We apply the Virtual Element method \eqref{eq:discreteEigPbm2} using a sequence of Voronoi meshes with mesh diameter $h=1/2$, $1/4$, $1/8$, $1/16$ (see the right plot in Figure \ref{split_domain} for an example of the adopted meshes). 
We show the plot of the convergence for the first eight computed eigenvalues in Figures \ref{fig:d1} , \ref{fig:d2}, \ref{fig:d3}, \ref{fig:d4}. We compute the relative errors by comparing our results with the values given in \cite{dauge}.

%

\begin{figure}
  \centering
  \begin{tabular}{ccc}
    \begin{overpic}[scale=0.32]{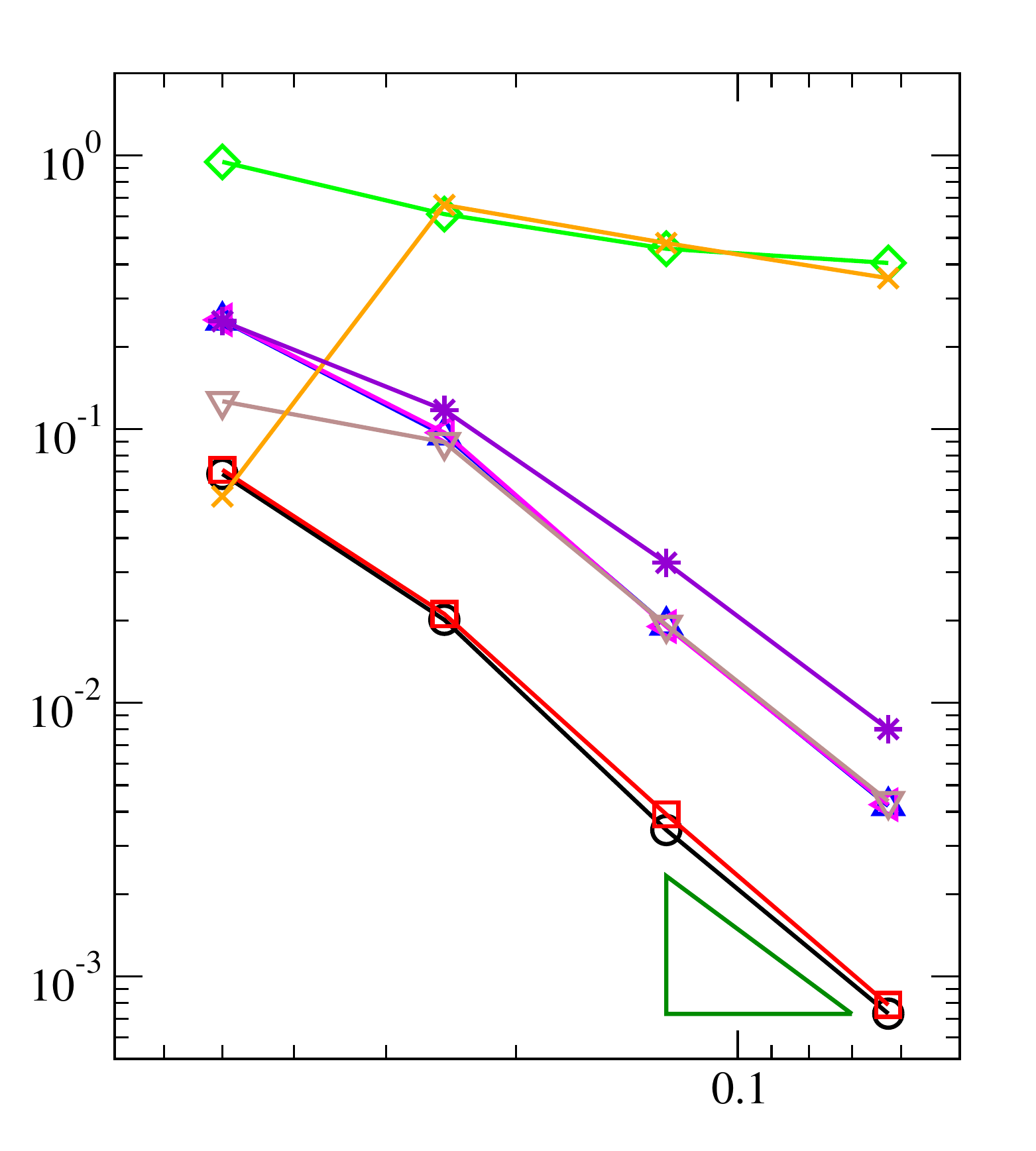}
      \put(-5,12){\begin{sideways}\textbf{Relative approximation error}\end{sideways}}
      \put(32,-5) {\textbf{Mesh size $\mathbf{h}$}}
      \put(53,17){\textbf{2}}
    \end{overpic}
    &
    \begin{overpic}[scale=0.32]{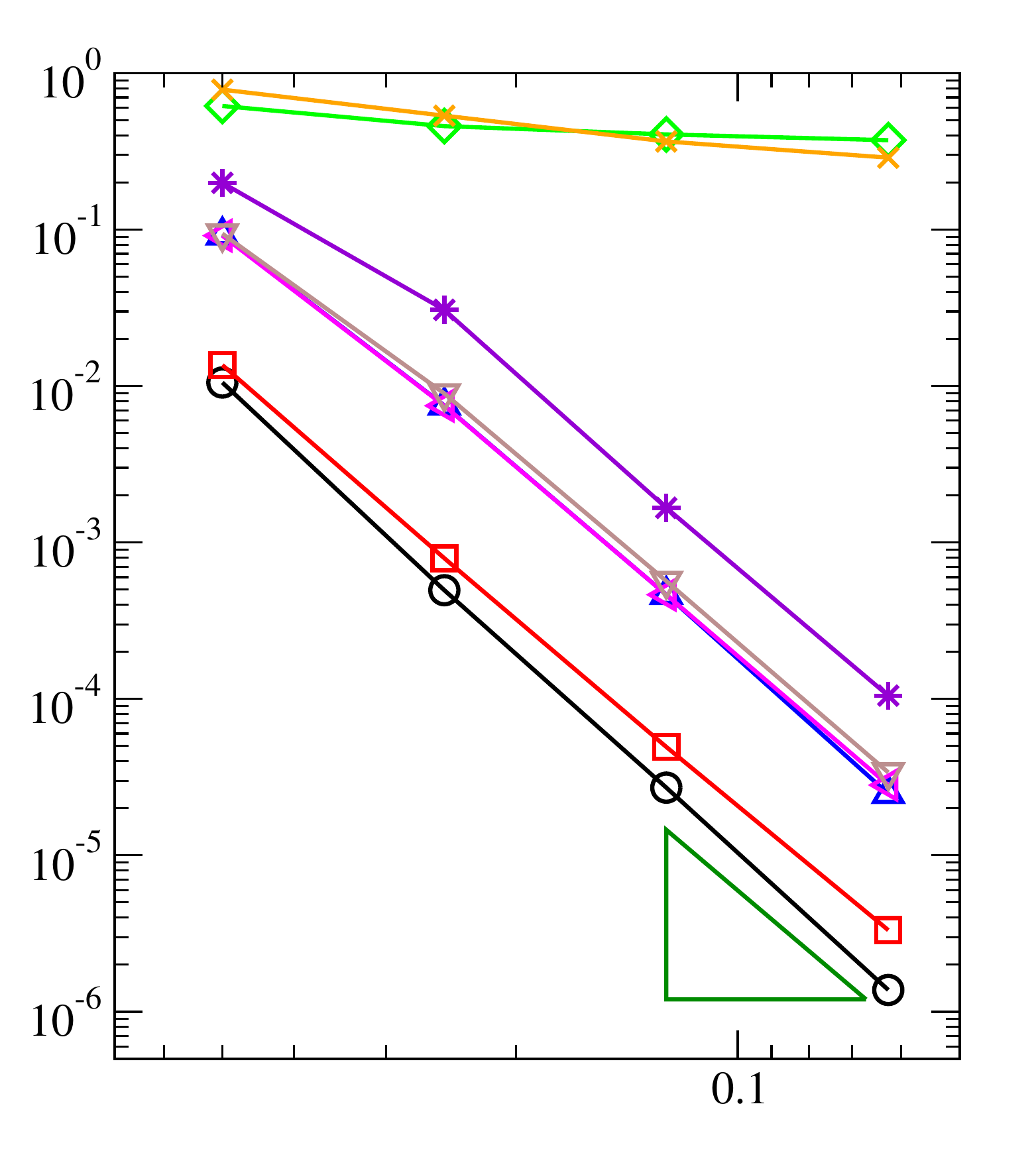}
      \put(32,-5) {\textbf{Mesh size $\mathbf{h}$}}
      \put(53,19){\textbf{4}}
    \end{overpic}
    &
    \begin{overpic}[scale=0.32]{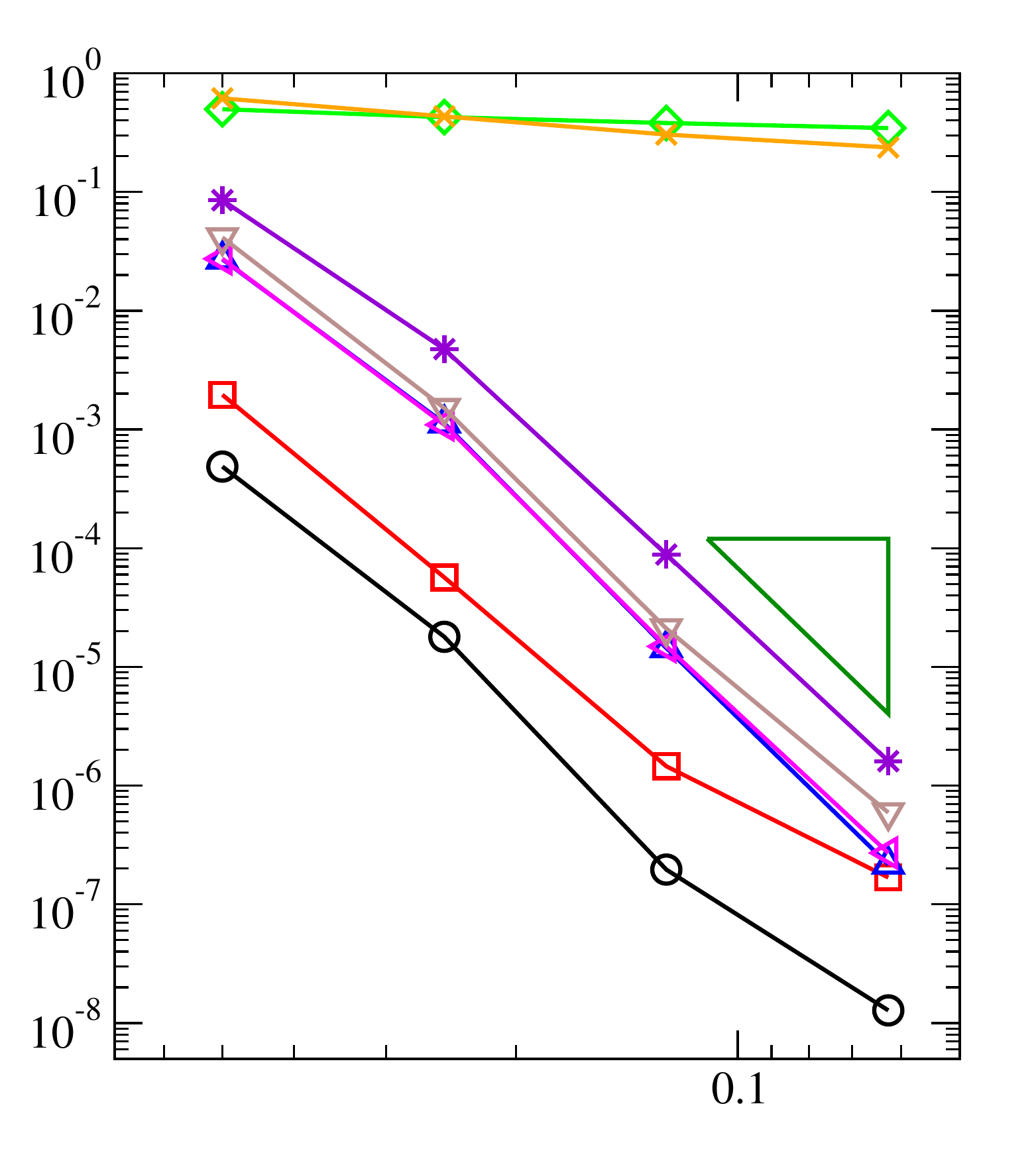}
      \put(32,-5) {\textbf{Mesh size $\mathbf{h}$}}
      \put(68,55){\textbf{6}}
    \end{overpic}
    \\[0.35em]
    $(k=1)$ & $(k=2)$ & $(k=3)$ \\
  \end{tabular}
  \caption{ Discontinuous diffusion problem, $\epsilon=0.01$; the symbols that labels the
            eigenvalues are in the following order: circle, square, diamond, triangle up,
            triangle left, triangle down, cross, star. 
            The generalized eigenvalue problem uses the stabilized bilinear
    form $\bsht(\cdot,\cdot)$. }
  \label{fig:d1}
\end{figure}

\begin{figure}
  \centering
  \begin{tabular}{ccc}
    \begin{overpic}[scale=0.32]{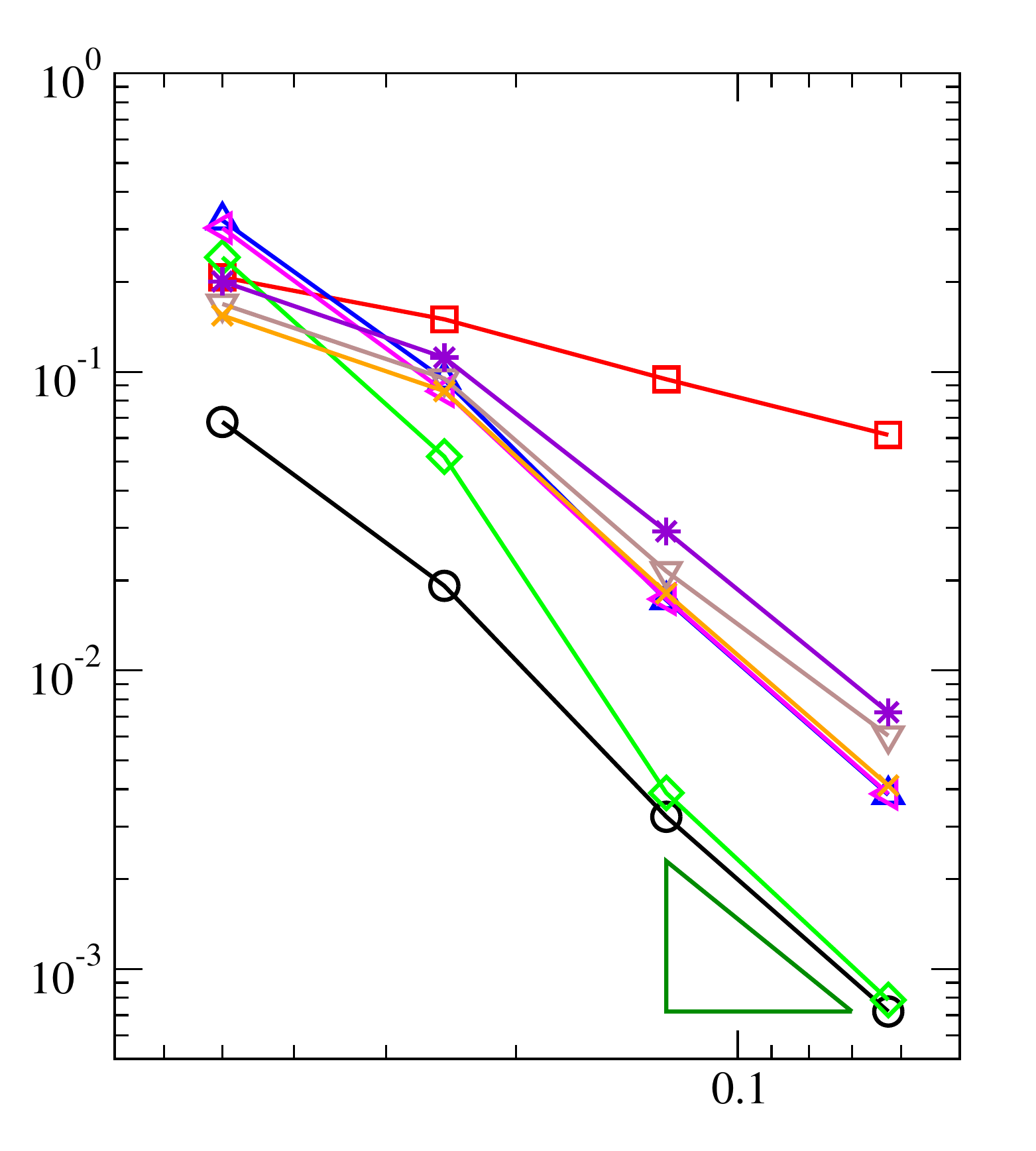}
      \put(-5,12){\begin{sideways}\textbf{Relative approximation error}\end{sideways}}
      \put(32,-5) {\textbf{Mesh size $\mathbf{h}$}}
      \put(53,17){\textbf{2}}
    \end{overpic}
    &
    \begin{overpic}[scale=0.32]{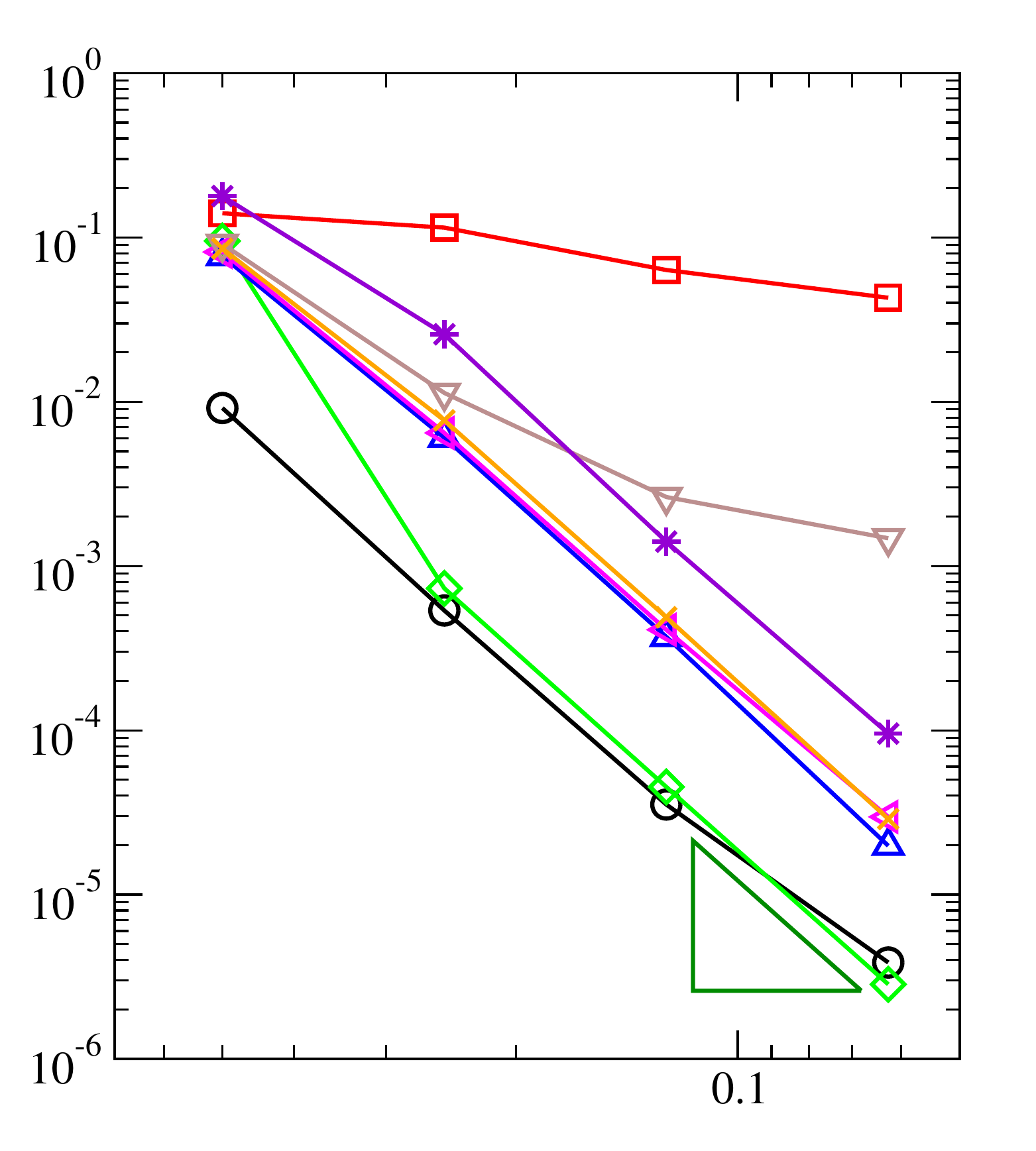}
      \put(32,-5) {\textbf{Mesh size $\mathbf{h}$}}
      \put(56,19){\textbf{4}}
    \end{overpic}
    &
    \begin{overpic}[scale=0.32]{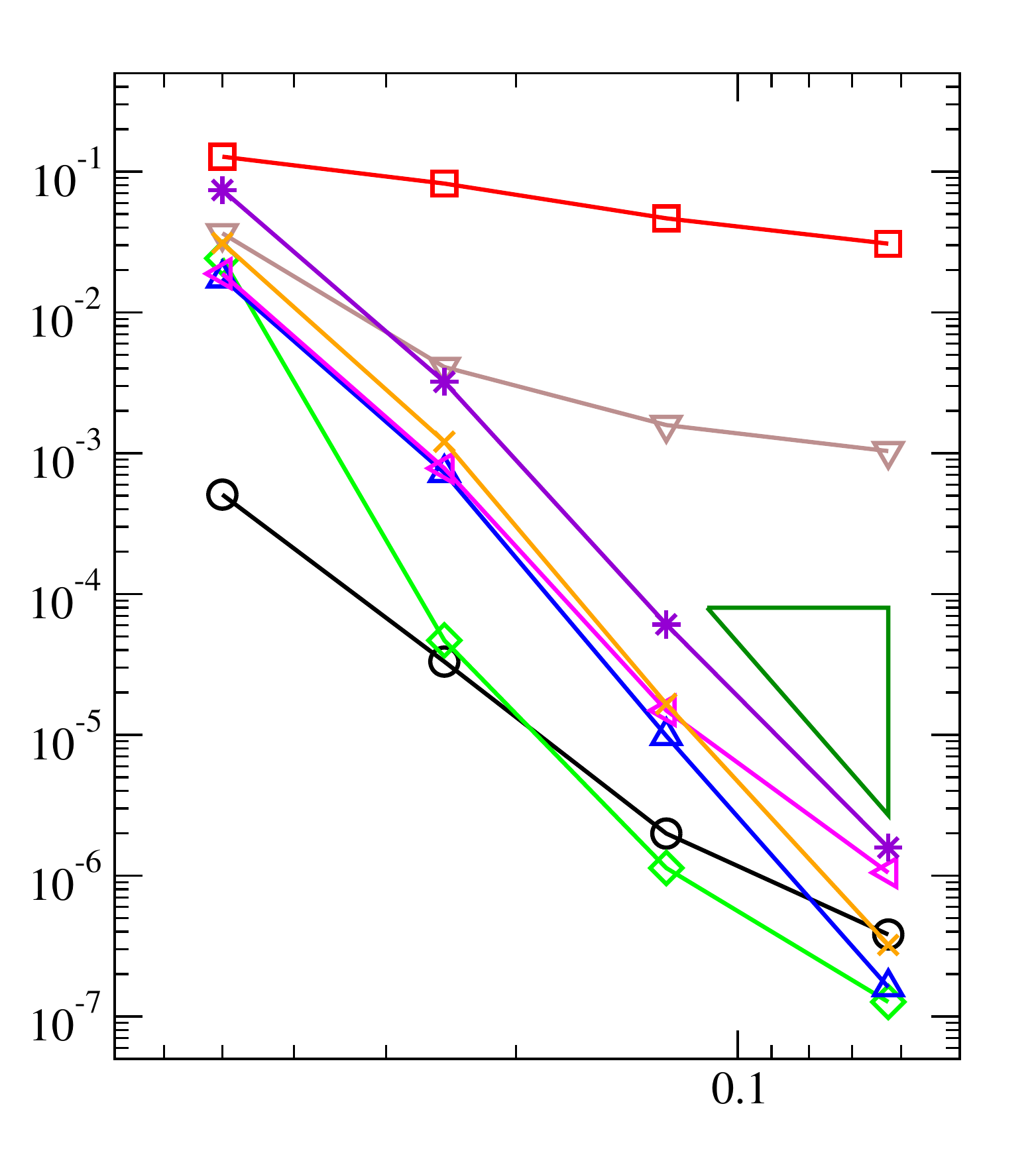}
      \put(32,-5) {\textbf{Mesh size $\mathbf{h}$}}
      \put(68,49){\textbf{6}}
    \end{overpic}
    \\[0.35em]
    $(k=1)$ & $(k=2)$ & $(k=3)$ \\
  \end{tabular}
  \caption{ Discontinuous diffusion problem, $\epsilon=0.10$; the symbols that labels the 
            eigenvalues are in the following order: circle, square, diamond, triangle up, 
            triangle left, triangle down, cross, star. 
            The generalized eigenvalue problem uses the stabilized bilinear
    form $\bsht(\cdot,\cdot)$.  }
  \label{fig:d2}
\end{figure}

\begin{figure}
  \centering
  \begin{tabular}{ccc}
    \begin{overpic}[scale=0.32]{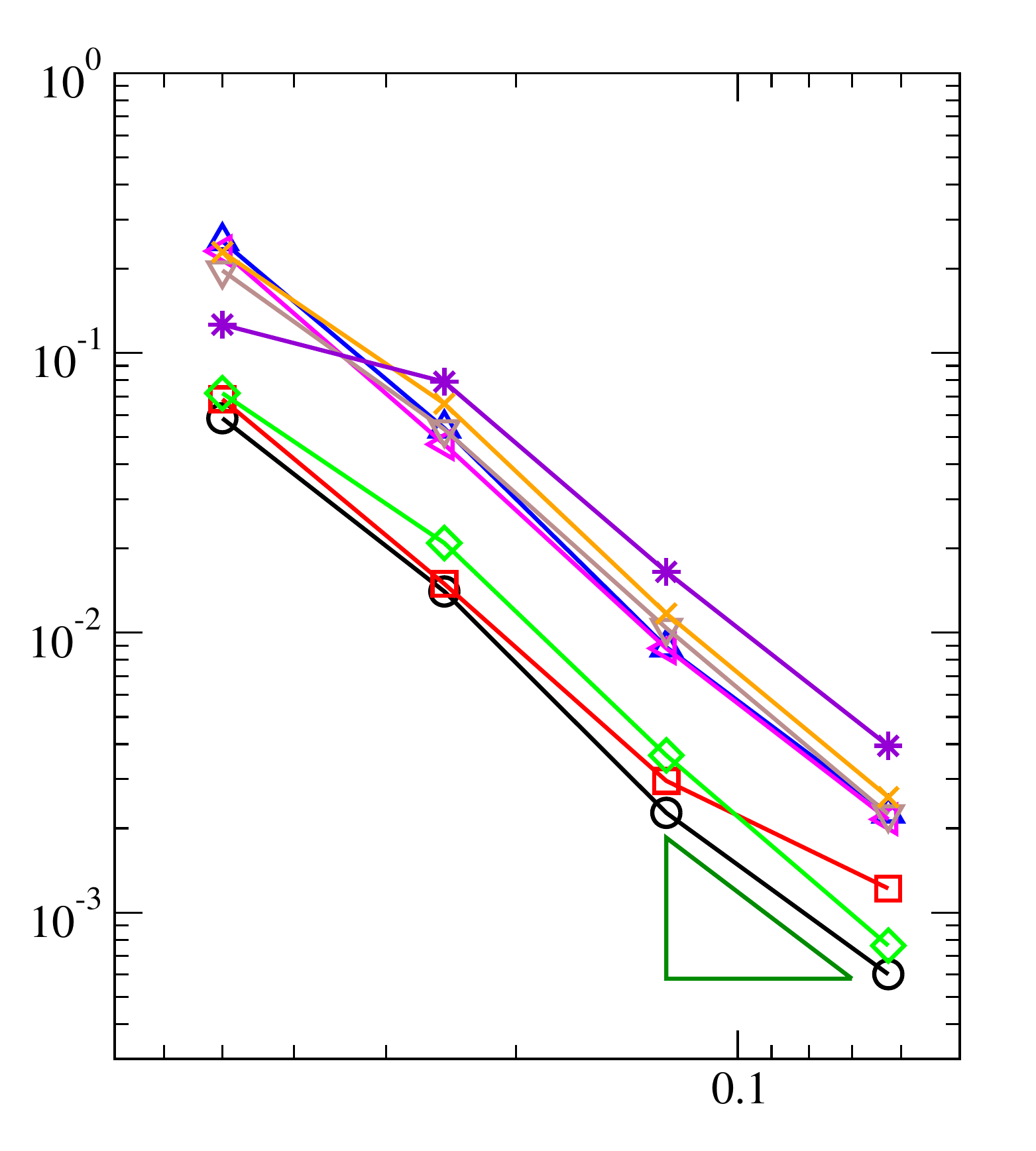}
      \put(-5,12){\begin{sideways}\textbf{Relative approximation error}\end{sideways}}
      \put(32,-5) {\textbf{Mesh size $\mathbf{h}$}}
      \put(53,19){\textbf{2}}
    \end{overpic}
    &
    \begin{overpic}[scale=0.32]{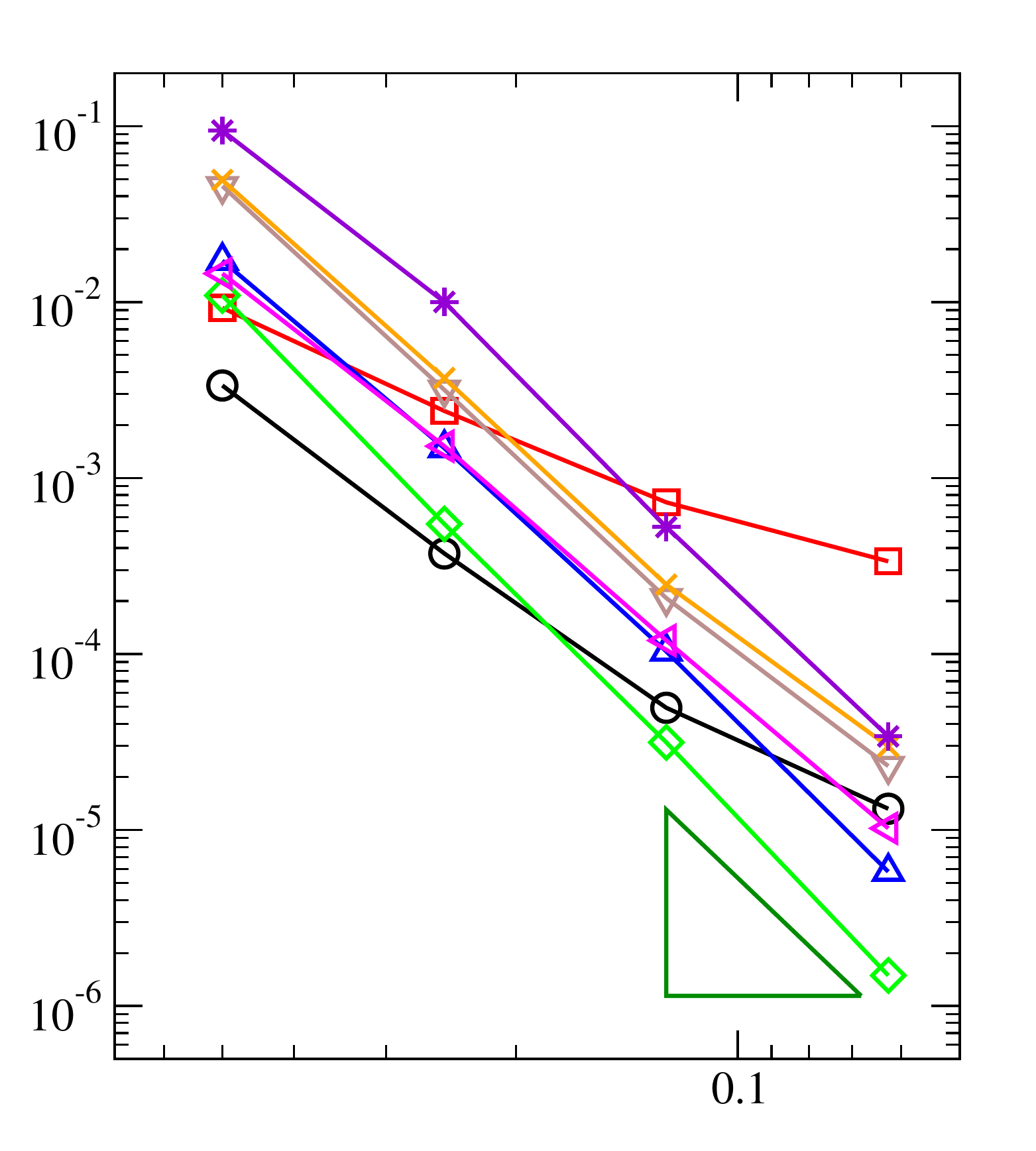}
      \put(32,-5) {\textbf{Mesh size $\mathbf{h}$}}
      \put(53,19){\textbf{4}}
    \end{overpic}
    &
    \begin{overpic}[scale=0.32]{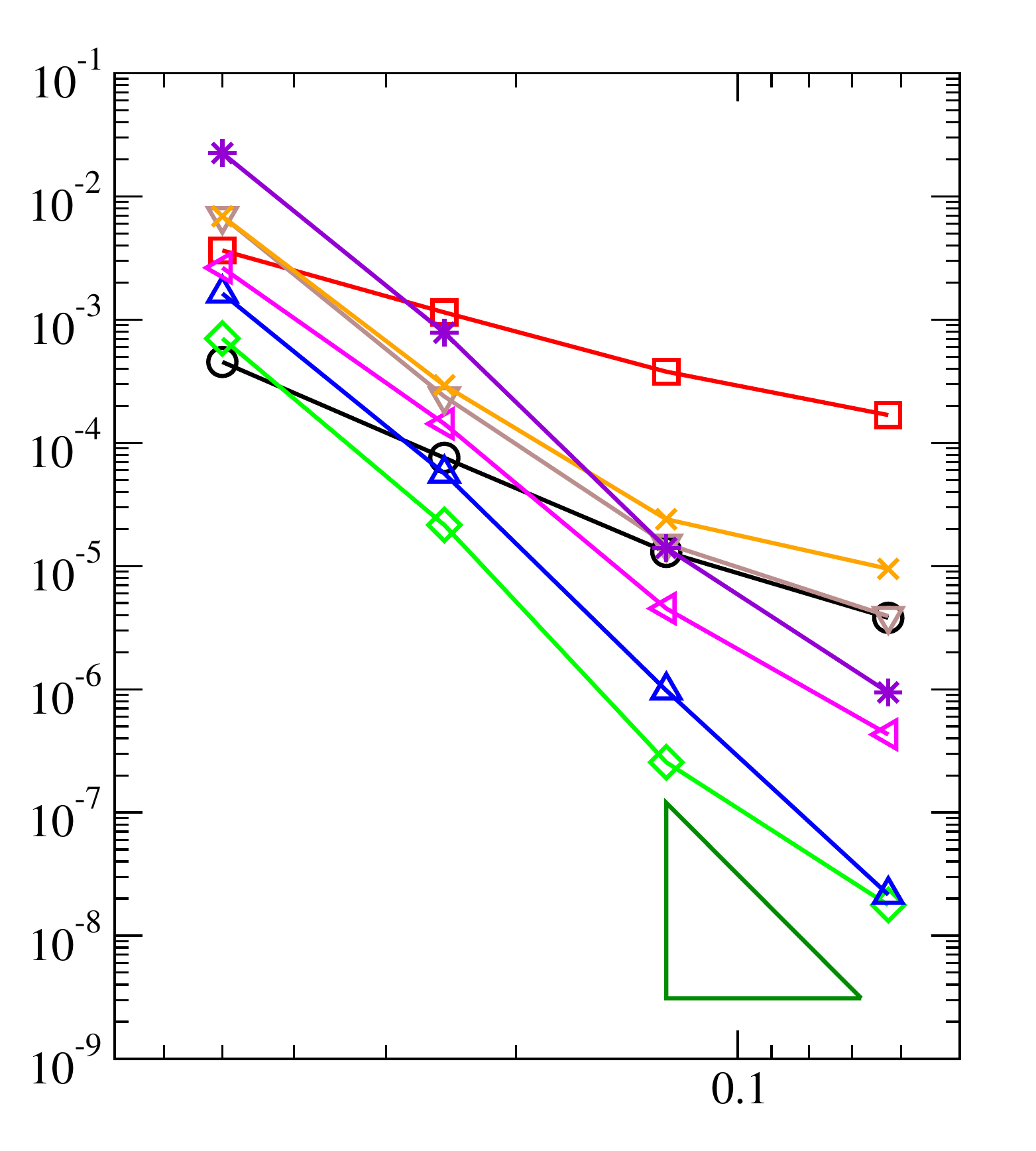}
      \put(32,-5) {\textbf{Mesh size $\mathbf{h}$}}
      \put(53,19){\textbf{6}}
    \end{overpic}
    \\[0.35em]
    $(k=1)$ & $(k=2)$ & $(k=3)$ \\
  \end{tabular}
  \caption{ Discontinuous diffusion problem, $\epsilon=0.50$; the symbols that labels the 
            eigenvalues are in the following order: circle, square, diamond, triangle up, 
            triangle left, triangle down, cross, star. 
            The generalized eigenvalue problem uses the stabilized bilinear
    form $\bsht(\cdot,\cdot)$.  }
  \label{fig:d3}
\end{figure}

\begin{figure}
  \centering
  \begin{tabular}{ccc}
    \begin{overpic}[scale=0.32]{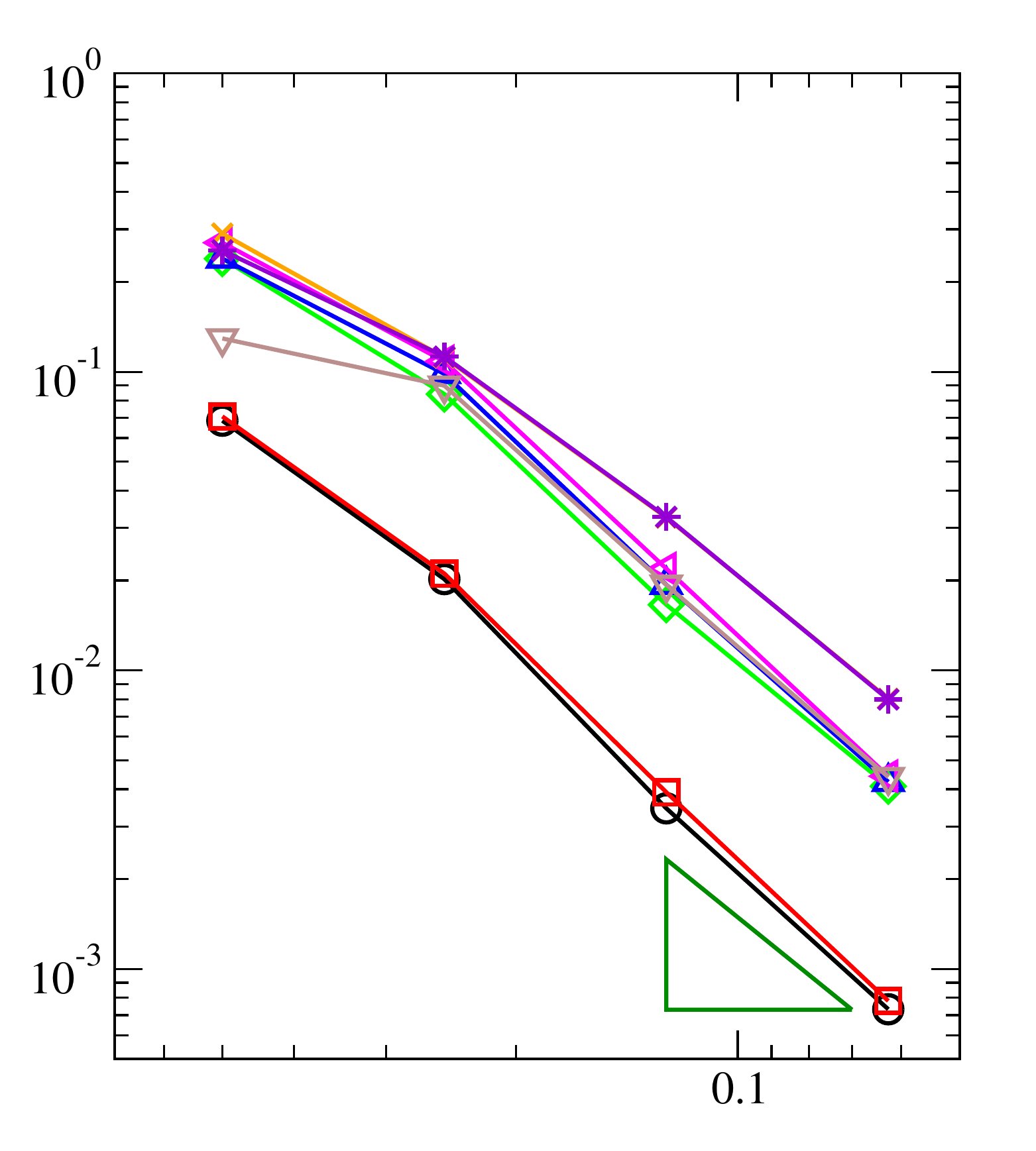}
      \put(-5,12){\begin{sideways}\textbf{Relative approximation error}\end{sideways}}
      \put(32,-5) {\textbf{Mesh size $\mathbf{h}$}}
      \put(53,17){\textbf{2}}
    \end{overpic}
    &
    \begin{overpic}[scale=0.32]{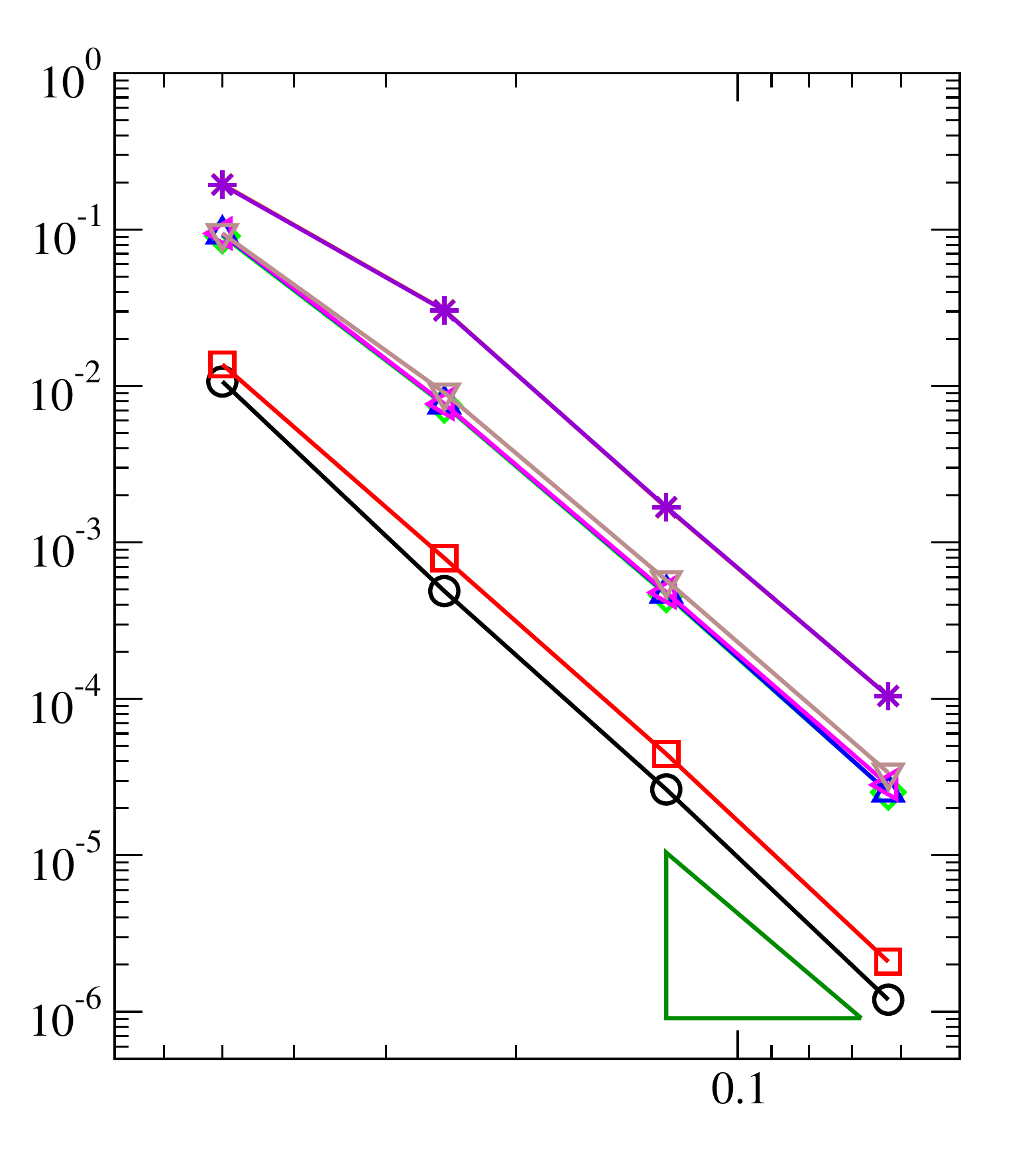}
      \put(32,-5) {\textbf{Mesh size $\mathbf{h}$}}
      \put(53,17){\textbf{4}}
    \end{overpic}
    &
    \begin{overpic}[scale=0.32]{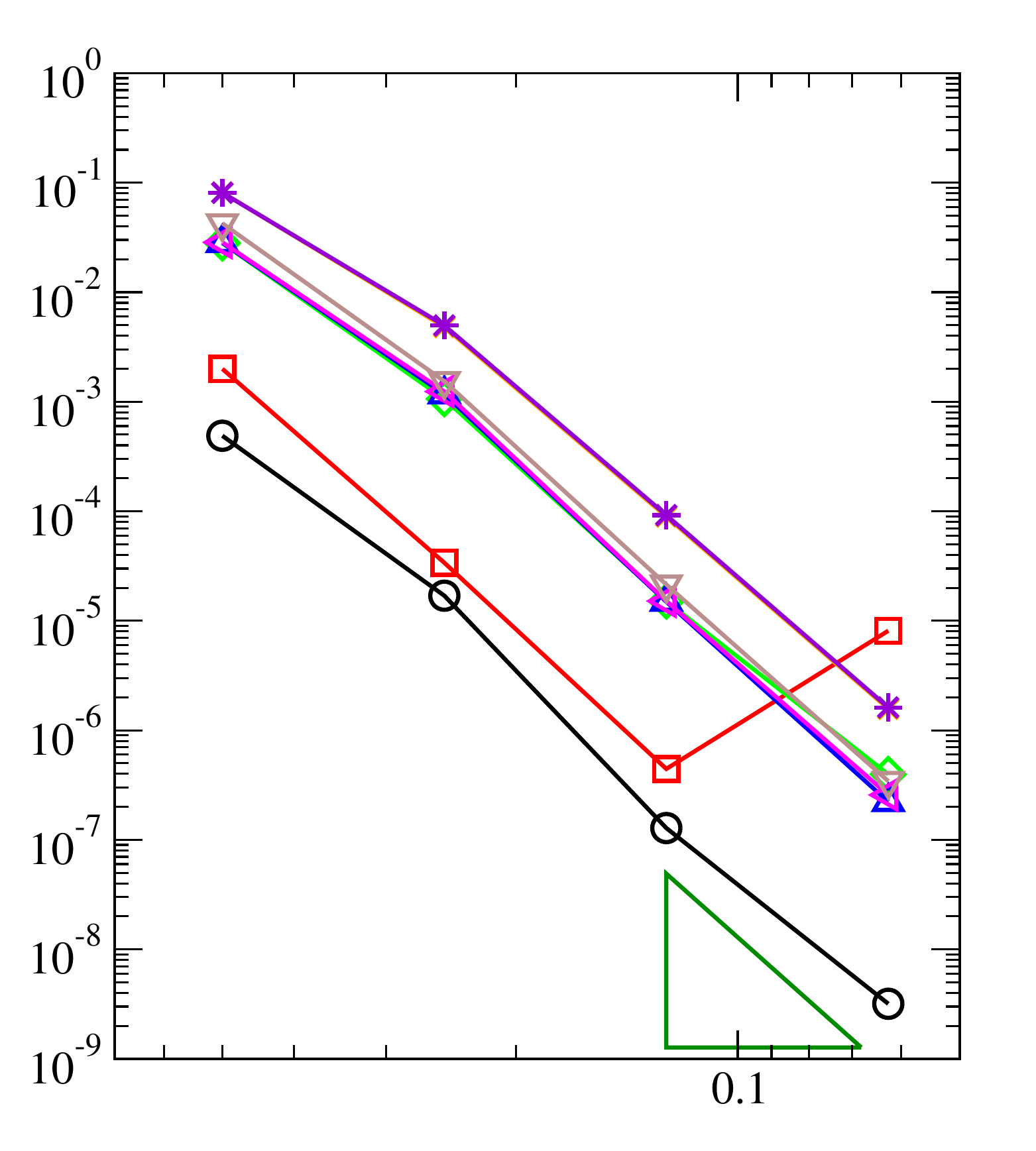}
      \put(32,-5) {\textbf{Mesh size $\mathbf{h}$}}
      \put(53,16){\textbf{6}}
    \end{overpic}
    \\[0.35em]
    $(k=1)$ & $(k=2)$ & $(k=3)$ \\
  \end{tabular}
  \caption{ Discontinuous diffusion problem, $\epsilon=10^{-8}$; the symbols that labels the 
            eigenvalues are in the following order: circle, square, diamond, triangle up, 
            triangle left, triangle down, cross, star. 
            The generalized eigenvalue problem uses the stabilized bilinear
    form $\bsht(\cdot,\cdot)$. }
  \label{fig:d4}
\end{figure}

In accordance with Theorem \ref{theorem:double:convergence:rate}, we obtain different rates of convergence that are determined by the polynomial order of the method and by the regularity of the corresponding exact eigenfunctions \cite{dauge}. Taking this into account, we show that the method is overall optimal, and thus stable with respect to discontinuities in the diffusivity tensor.

\section{Conclusions}
\label{sec:conclusion}

We have discussed the application of the conforming virtual element
method to the numerical resolution of eigenvalue problems with
potential terms on polytopal meshes.
The most notable case is that of the Schr\"odinger equation with a
suitable pseudopotential, which is fundamental in the numerical
treatment of more complex problems in the Density Functional Theory.
The VEM approximation of such problem was discussed from both the
theoretical and the numerical viewpoint, proving that the 
method provides a correct spectral 
approximation with optimal rates of convergence.
The performance of the method was shown by computing the first
eigenvalues of the Quantum Harmonic Oscillator provided by the
harmonic potential and a singular eigenvalue problem with zero
potential.

\section*{Acknowledgements}
The work of the first and third author was supported by the Laboratory
Directed Research and Development Program (LDRD), U.S. Department of
Energy Office of Science, Office of Fusion Energy Sciences, and the
DOE Office of Science Advanced Scientific Computing Research (ASCR)
Program in Applied Mathematics Research, under the auspices of the
National Nuclear Security Administration of the U.S. Department of
Energy by Los Alamos National Laboratory, operated by Los Alamos
National Security LLC under contract DE-AC52-06NA25396.
The fourth author was partially supported by the European Research
Council through the H2020 Consolidator Grant (grant no. 681162) CAVE -
Challenges and Advancements in Virtual Elements. This support is
gratefully acknowledged.

\bibliographystyle{abbrv}
\bibliography{vem-eigv}

\end{document}